\documentclass[final,3p,times]{elsarticle}

\usepackage{tcolorbox}
\usepackage[T1]{fontenc}
\usepackage[utf8]{inputenc}
\usepackage{amsmath,amssymb}
\usepackage{amsthm}
\usepackage{tikz,pgfplots}
\pgfplotsset{compat=1.17} 
\usepackage{standalone}
\usepackage{subcaption}
\graphicspath{{./Figures/}}
\usepackage{tabularx}
\usepackage{multirow}
\usepackage{lscape}
\usepackage{rotating}
\usepackage{algorithm}
\usepackage{algpseudocode}
\usepackage{hyperref}
\usetikzlibrary{arrows}

\newcommand{\abs}[1]{\left\lvert #1 \right\rvert}

\newtheorem{theorem}{Theorem}

\newtheorem{remark}{Remark}
\newenvironment{equationwithLineno}
    {
    \begin{equation}
    }
    {
    \end{equation}
    }

\newenvironment{equationwithLineno*}
    {
    \begin{equation*}
    }
    {
    \end{equation*}
    }

\begin{document}

\begin{frontmatter}



\title{Algebraic Inverse Fast Multipole Method: A fast direct solver that is better than HODLR based fast direct solver}


\author[inst1]{Vaishnavi Gujjula}

\affiliation[inst1]{organization={Department of Mathematics},
            addressline={Indian Institute of Technology Madras}, 
            city={Chennai},
            postcode={600036}, 
            state={Tamil Nadu},
            country={India}}

\author[inst1]{Sivaram Ambikasaran}
\begin{abstract}
This article presents a fast direct solver, termed Algebraic Inverse Fast Multipole Method (from now on abbreviated as AIFMM), for linear systems arising out of $N$-body problems. AIFMM relies on the following three main ideas: (i) Certain sub-blocks in the matrix corresponding to $N$-body problems can be efficiently represented as low-rank matrices; (ii) The low-rank sub-blocks in the above matrix are leveraged to construct an extended sparse linear system; (iii) While solving the extended sparse linear system, certain fill-ins that arise in the elimination phase are represented as low-rank matrices and are "redirected" though other variables maintaining zero fill-in sparsity. The main highlights of this article are the following: (i) Our method is completely algebraic (as opposed to the existing Inverse Fast Multipole Method~\cite{ambikasaran2014inverse,doi:10.1137/15M1034477,TAKAHASHI2017406}, from now on abbreviated as IFMM). We rely on our new Nested Cross Approximation~\cite{gujjula2022nca} (from now on abbreviated as NNCA) to represent the matrix arising out of $N$-body problems. (ii) A significant contribution is that the algorithm presented in this article is more efficient than the existing IFMMs. In the existing IFMMs, the fill-ins are compressed and redirected as and when they are created. Whereas in this article, we update the fill-ins first without affecting the computational complexity. We then compress and redirect them only once. (iii) Another noteworthy contribution of this article is that we provide a comparison of AIFMM with Hierarchical Off-Diagonal Low-Rank (from now on abbreviated as HODLR) based fast direct solver and NNCA powered GMRES based fast iterative solver. (iv) Additionally, AIFMM is also demonstrated as a preconditioner.
\end{abstract}


\end{frontmatter}


\section{Introduction}
\label{sec:introduction}
This article focuses on solving linear systems that arise out of $N$-body problems. Such $N$-body problems arise frequently in many applications such as electrostatics, integral equation solvers, radial basis function interpolation, inverse problems, Gaussian process regression, wave scattering, etc. 
A linear system can be solved using a direct solver or an iterative solver. Both have their advantages over the other and it is highly problem specific to choose the solver that is better.

 An important and widely used iterative technique is the Krylov subspace technique, which involves matrix-vector products. To speed up these matrix-vector products, fast summation techniques such as Fast Multipole Method (FMM)~\cite{greengard1987fast}, Barnes-Hut~\cite{barnes1986hierarchical}, FFT, etc. are used. Further for fast convergence in problems with high condition numbers, an iterative solver is coupled with a preconditioner.

On the other hand, direct solvers involve a factorization step followed by a solve step. The factorization step comprises of an LU factorization or QR factorization, etc., which generally is computationally more expensive than an iterative technique. But direct solvers are more robust and accurate than iterative solvers. Further, direct solvers are advantageous when one is interested in multiple right-hand sides. 
A naive direct solver costs $\mathcal{O}(N^{3})$, which is prohibitively large for large system sizes. To reduce the computational complexity, fast methods are used. Many dense matrices arising out of $N$-body problems possess a hierarchical low-rank structure.
This low-rank structure is exploited to construct hierarchical matrices and hierarchical matrices based fast direct solvers ~\cite{hackbusch1999sparse, grasedyck2003construction,borm2003introduction,bebendorf2007hierarchical}.

While constructing a hierarchical matrix, the low-rank bases of the sub-blocks that are compressed can be obtained in a nested or a non-nested approach. In the nested approach, the low-rank bases at a parent level in the hierarchy are constructed from the bases at the child level. The class of hierarchical matrices that follow the nested approach are called $\mathcal{H}^{2}$ matrices.

The matrix sub-blocks that are low-rank approximated in a hierarchical matrix are identified based on an admissibility condition.
The most widely used admissibility conditions are the weak admissibility and strong or standard admissibility conditions.

Hierarchically Off-Diagonal Low-Rank (HODLR) and Hierarchically Semi-Separable (HSS) matrices~\cite{ho2012fast,kong2011adaptive,martinsson2005fast,ambikasaran2013mathcal,chandrasekaran2002fast,chandrasekaran2006fast,chandrasekaran2007fast} are sub-classes of Hierarchical matrices that follow the weak admissibility condition, wherein all the off-diagonal sub-blocks are approximated by low-rank matrices. The former follows a non-nested approach in the construction of the bases and the latter follows the nested approach. The major drawback with these classes of Hierarchical matrices is that the ranks of the compressed sub-matrices are not "truly" low-rank. For instance in 2D, the ranks of the compressed sub-matrices grow as $\mathcal{O}\left(\sqrt{N}\log^{2}(N)\right)$~\cite{kandappan2022hodlr2d}, in 3D it is $\mathcal{O}\left(N^{\frac{2}{3}}\log^{3}(N)\right)$, where $N$ is the size of the compressed sub-matrix. In a $d$-dimensional setting, the rank of the compressed sub-matrices grow as $\mathcal{O}\left(N^{\frac{d-1}{d}}\log^{d}(N)\right)$~\cite{khan2022numerical}. So, the direct solvers developed for HODLR and HSS classes of matrices are not linear in complexity. 


Another sub-class of Hierarchical matrices is the $\mathcal{H}^{2}$ matrices with strong admissibility condition~\cite{borm2003introduction,bebendorf2008hierarchical,hackbusch2015hierarchical}, wherein interactions between neighboring clusters of particles are not compressed and interactions between well-separated clusters of particles are approximated by low-rank matrices. 
This sub-class of Hierarchical matrices can be considered to be the algebraic generalization of FMM, so they are also referred as FMM matrices.
The strong admissibility condition guarantees that the interactions between well-separated clusters of particles (when the underlying matrix is generated from singular kernels)~\cite{khan2022numerical} do not scale with the number of particles in the cluster.

There exists a vast literature on $\mathcal{H}^{2}$ matrices and solvers for linear systems involving $\mathcal{H}^{2}$ matrices in almost linear complexity~\cite{hackbusch2002h2,bebendorf2005hierarchical,borm2006matrix,borm2010efficient,borm2013efficient}. The pre-factors in the scaling term of these methods tend to be large. Inverse fast multipole method (from now on abbreviated as IFMM), a fast direct solver for FMM matrices with linear complexity, was introduced~\cite{ambikasaran2014inverse}, whose pre-factor in the scaling term is not that large.
Another related work on fast direct solvers for FMM matrices is the strong skeletonization based factorization method~\cite{minden2017recursive,sushnikova2022fmm}. 

The work developed in this article is a variant of the IFMM developed in~\cite{ambikasaran2014inverse,doi:10.1137/15M1034477,TAKAHASHI2017406}.
One of the key ideas based on which the IFMM is developed is the extended sparsification technique that was earlier used in~\cite{ho2012fast,10.5555/1087445,greengard2009fast}. 
In IFMM, an extended sparse system of size $\mathcal{O}(N)$ is developed by introducing auxiliary variables - which are the locals and multipoles of nodes at various levels of the FMM tree. 
The advantage of sparsification is that the computational complexity of the solver gets improved, provided the fill-ins are minimal. 
In IFMM, the fill-ins corresponding to well-separated clusters of particles are compressed and redirected via the existing non-zero entries, which contributes to its linear scaling.
With the auxiliary variables in IFMM being the locals and multipoles of nodes at various levels of the FMM tree, the extended sparse matrix is an assembly of the FMM operators, which are the L2P/L2L (Local-To-Particle/Local-To-Local), M2L (Multipole-To-Local), P2M/M2M (Particle-To-Multipole/Multipole-To-Multipole), \\and P2P (Particle-To-Particle) operators. 

In this article, we develop an Algebraic Inverse Fast Multipole Method (from now on abbreviated as AIFMM), wherein we employ a \textbf{new} Nested Cross Approximation (NNCA)~\cite{gujjula2022nca}, an algebraic technique, to obtain the L2P/L2L, M2L and P2M/M2M operators. 
Nested Cross Approximation (from now on abbreviated as NCA)~\cite{bebendorf2012constructing,zhao2019fast,gujjula2022nca} is a nested version of Adaptive Cross Approximation (from now on abbreviated as ACA), that forms low-rank bases in a nested fashion. NNCA~\cite{gujjula2022nca} differs from the NCAs described in~\cite{bebendorf2012constructing,zhao2019fast} in the technique of choosing pivots, a key step of the approximation. The search space for far-field pivots of a hypercube (a node belonging to the $2^{d}$ tree) is considered to be its interaction list region in the former and is considered to be its entire far-field region in the latter. So the time to build the former approximation is lower than that of the latter, for no significant difference in accuracy.

There have been a few articles that have presented IFMM for various applications. 
Following are the differences between this article and the earlier articles~\cite{ambikasaran2014inverse,doi:10.1137/15M1034477,TAKAHASHI2017406}:
\begin{enumerate}
    \item 
    In this present article, the FMM operators which are used to form the extended sparse matrix, are obtained using NNCA~\cite{gujjula2022nca}, a purely algebraic technique. 
    While in~\cite{doi:10.1137/15M1034477}, it is done using Chebyshev interpolation, an analytic technique. And in~\cite{TAKAHASHI2017406}, the extended sparse matrix is assembled using the Low Frequency FMM (LFFMM) operators, which is also an analytic technique.
    The advantages of an algebraic technique are that i) the method can be used in a black box fashion irrespective of the application ii) the ranks obtained are usually lower than that of the analytic techniques as the bases obtained through an algebraic method are problem and domain specific.
    \item
    In this article, the fill-in compression and redirection is performed using rank revealing QR (RRQR). While in~\cite{doi:10.1137/15M1034477,TAKAHASHI2017406} it is done using randomized SVD.
    \item
    In this article, a more efficient elimination algorithm than the one stated in the existing IFMMs~\cite{ambikasaran2014inverse,doi:10.1137/15M1034477,TAKAHASHI2017406} is presented. 
    In the existing IFMMs, fill-ins are compressed and redirected as and when they are created, which could happen multiple times in the elimination process. Whereas in this article, we do not compress and redirect a fill-in as and when created. We update the fill-ins without affecting the computational complexity. We then compress and redirect only once. 
    \item
    In this article, we demonstrate AIFMM as a preconditioner in the high frequency scattering problem. While in~\cite{doi:10.1137/15M1034477,TAKAHASHI2017406}, IFMM is studied as a preconditioner in a Stokes flow problem and a 3D Helmholtz BEM at low frequencies respectively.
\end{enumerate}

Below are the highlights of the AIFMM presented in this article:
\begin{enumerate}
\item It is a completely algebraic method, i.e., it does not use any analytic techniques such as the interpolation techniques~\cite{doi:10.1137/15M1034477}, multipole expansions~\cite{TAKAHASHI2017406}, etc, to obtain the low-rank factorizations.
\item AIFMM is demonstrated as a direct solver for linear systems involving non-oscillatory Green's functions and the 2D Helmholtz function at low frequency. 
\item To the best of our knowledge, this work is one of the first to provide a comparison of the performance of AIFMM with that of i) a HODLR based fast direct solver~\cite{ambikasaran2013mathcal} ii) GMRES, an iterative solver. It is observed that AIFMM is faster than HODLR, and GMRES is faster than AIFMM. But when one is interested in solving for multiple right hand sides, AIFMM is faster than GMRES.
\item AIFMM is also demonstrated as a preconditioner in an iterative scheme for high frequency scattering problem. It is observed that AIFMM as a preconditioner is better than the block-diagonal preconditioner, but not as good as the HODLR preconditioner.
\end{enumerate}

The rest of the article is organized as follows: Section~\ref{sec:preliminaries} describes the preliminaries to develop AIFMM, which are the construction of FMM tree, identification of the low-rank sub-blocks, and assembly of the various FMM operators using NNCA. Section~\ref{sec:AIFMM} describes AIFMM, which includes the construction of the extended sparse system, elimination phase, and back substitution phase. Section~\ref{sec:numericalResults}, illustrates various numerical benchmarks of AIFMM in comparison to those of GMRES and HODLR.

\section{Preliminaries}\label{sec:preliminaries}
Let $u\in \mathbb{R}^{N\times d}$ be the coordinates of $N$ targets in $d$ dimensions (we will be referring them as target points), $v\in \mathbb{R}^{N\times d}$ be the coordinates of $N$ sources in $d$ dimensions (we will be referring them as source points). Let $A\in\mathbb{C}^{N\times N}$ be the matrix that captures the pair-wise interaction between these points, i.e., $A_{ij}$ is the interaction between the source and target located at $v_{j}$ and $u_{i}$ respectively. Such interaction matrices arise in many applications; to name a few integral equation solvers, particle simulations, covariance matrices, electrostatics, scattering, etc.
Electrostatic problems are studied extensively in the literature and the naming conventions in most of the research articles are based on it. Hence, in this article, we follow the nomenclature of Electrostatics. 

We assume unknown charges of strength $x\in \mathbb{C}^{N\times 1}$ are located at source points $v$ and the potential $b\in \mathbb{C}^{N\times 1}$ at the target points $u$ is known. We are interested in finding the unknown charges $x$, given the potential $b$ or in other words solve the system of equations, 
\begin{equationwithLineno}
    Ax=b.
\end{equationwithLineno}

A key idea of the inverse fast multipole method is to introduce auxiliary variables and then create an extended sparse system of size $\mathcal{O}(N)$. The advantage of sparsification is that it reduces the complexity of the problem
as some of the fill-ins that get created in the elimination phase are compressed and redirected through the existing non-zero entries, resulting in a linear complexity algorithm.


The extended sparse matrix is created by constructing the FMM matrix representation of $A$. The multipoles and locals that are formed at various levels of the FMM tree are considered to be the auxiliary variables. 

The steps involved in constructing the extended sparse matrix are i) construction of FMM tree ii) identification of the low-rank matrix sub-blocks, and iii) assembly of the various FMM operators. We now describe each of these steps below.

\subsection{Construction of FMM tree}\label{ssec:Tree}
We consider a smallest hypercube that contains the support of the particles to be the domain $\Omega\in\mathbb{R}^{d}$. A $2^{d}$ uniform tree is constructed over $\Omega$. 
The hypercube at level $0$, is the domain $\Omega$ itself. A hypercube at level $l$ is subdivided into $2^{d}$ hypercubes, which are considered to be at level $l+1$ of the tree. The former is considered to be the parent of the latter and the latter are considered to be the children of the former. And this subdivision is carried on hierarchically until a level $L$ is reached where the hypercubes contain no more than $n_{max}$ particles. We will be referring to the hierarchical tree as $\mathcal{T}^{L}$. The notations associated with a hypercube $i$ at level $l$ are described in Table~\ref{table:Notations1a}. For $d=2$, the construction of quad-tree and the numbering of the nodes till level 2 is illustrated in Figure~\ref{fig:numbering}.

   \begin{table}[!htbp]
\centering
  \begin{tabular}{|l|l|}
    \hline
    $i^{(l)}$ & Hypercube (also referred to as node or box) $i$ at level $l$ of the tree\\ \hline
    $\mathcal{P}(i^{(l)})$ & Parent of $i^{(l)}$ \\ \hline
    $\mathcal{C}(i^{(l)})$ & $\{j^{(l+1)}:\text{ }j^{(l+1)}\text{ is a child of }i^{(l)}\}$ \\ \hline
  \end{tabular}
  \caption{Notations associated with hypercube $i$ at level $l$}
    \label{table:Notations1a}
 \end{table}

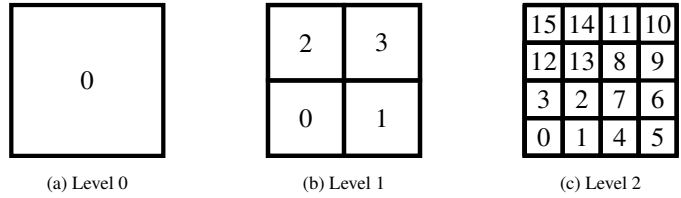
\begin{figure}[!htbp]
\centering
	\begin{subfigure}[b]{0.2\textwidth}
    \centering
    \begin{tikzpicture}[scale=0.5]
        \draw[ultra thick] (-2,-2) rectangle (2,2);
        \node at (0,0) {$0$};
    \end{tikzpicture}
    \caption{Level 0}
	\end{subfigure}
	\begin{subfigure}[b]{0.2\textwidth}
\centering
\begin{tikzpicture}[scale=0.5]
    \draw[ultra thick] (-2,-2) rectangle (2,2);
    \draw[ultra thick] (-2,-2) rectangle (0,0);
    \draw[ultra thick] (0,-2) rectangle (2,0);
    \draw[ultra thick] (0,0) rectangle (2,2);
    \draw[ultra thick] (-2,0) rectangle (0,2);
    \node at (-1,-1) {$0$};
    \node at (1,-1) {$1$};
    \node at (-1,1) {$2$};
    \node at (1,1) {$3$};
\end{tikzpicture}
\caption{Level 1}
	\end{subfigure}
	\begin{subfigure}[b]{0.2\textwidth}
\centering
\begin{tikzpicture}[scale=0.5]
    \draw[ultra thick] (-2,-2) rectangle (2,2);
    \draw[ultra thick] (-2,-2) rectangle (1,1);
    \draw[ultra thick] (-2,-2) grid (2,2);
    \draw[ultra thick] (-2,-2) grid (1,1);
    \node at (-1.5,-1.5) {$0$};
    \node at (-0.5,-1.5) {$1$};
    \node at (-0.5,-0.5) {$2$};
    \node at (-1.5,-0.5) {$3$};

    \node at (0.5,-1.5) {$4$};
    \node at (1.5,-1.5) {$5$};
    \node at (1.5,-0.5) {$6$};
    \node at (0.5,-0.5) {$7$};

    \node at (0.5,0.5) {$8$};
    \node at (1.5,0.5) {$9$};
    \node at (1.5,1.5) {$10$};
    \node at (0.5,1.5) {$11$};
    
    \node at (-1.5,0.5) {$12$};
    \node at (-0.5,0.5) {$13$};
    \node at (-0.5,1.5) {$14$};
    \node at (-1.5,1.5) {$15$};
    
\end{tikzpicture}
\caption{Level 2}
    \end{subfigure}
	\caption{The numbering convention followed at levels 0, 1, and 2 of a quad-tree.}
	\label{fig:numbering}
\end{figure}

\subsection{Identification of the low-rank matrix sub-blocks}\label{ssec:admCond}
Let $t^{X}$ and $s^{X}$ be the index sets that store indices of points $u$ and $v$ respectively that lie in hypercube $X^{(l)}$. 
\begin{align}
    t^{X} &= \{m:u_{m}\in X^{(l)}\}\\
    s^{X} &= \{n:v_{n}\in X^{(l)}\}
\end{align}
In this section and the upcoming sections, we omit the superscript that represents the level at some places, to improve the readability of notations, in the hope that the level can be understood from the context.
We follow the strong admissibility condition to identify the sub-blocks of the matrix that can be efficiently approximated by a low-rank matrix, i.e., the interaction between the clusters of particles located in hypercubes $X^{(l)}$ and $Y^{(l)}$, $A_{t^{X}s^{Y}}$, is approximated by a low-rank matrix, only if

\begin{equationwithLineno}\label{eq:adm_condn}
    \max\{\text{diam}(X^{(l)}), \text{diam}(Y^{(l)})\} \leq \eta \text{dist}(X^{(l)},Y^{(l)}), \text{ where}
\end{equationwithLineno}
 \begin{align*}
    \text{diam}(X^{(l)}) &= \sup\{\lVert x-y\rVert_{2}: x,y\in X^{(l)}\}, \\
    \text{dist}(X^{(l)},Y^{(l)}) &= \inf\{\lVert x-y\rVert_{2}: x\in X^{(l)}, y\in Y^{(l)}\}.
  \end{align*}
If $X^{(l)}$ and $Y^{(l)}$ satisfy the above stated strong admissibility criterion, then  $X^{(l)}$ and $Y^{(l)}$ are said to be well-separated and the interaction matrix $A_{t^{X}s^{Y}}$ is said to be admissible. Further $A_{t^{X}s^{Y}}$ is considered to be a far-field interaction. If $X^{(l)}$ and $Y^{(l)}$ do not agree with the strong admissibility criterion, then $A_{t^{X}s^{Y}}$ is said to be non-admissible and is considered to be a near-field interaction. In this article, we consider $\eta = \sqrt{d}$.

\subsubsection{FMM matrix structure}
For each node $i$ at level $l$, we introduce the neighbors and interaction list, described in Table~\ref{table:Notations1b}. We illustrate the same for a node in 2D in Figure~\ref{fig:adm_condn}. 
   \begin{table}[!htbp]
\centering
  \begin{tabularx}{\textwidth}{|l|X|}
    \hline
    $\mathcal{N}(i^{(l)})$ & Neighbors of $i^{(l)}$ that consists of hypercubes at level $l$, that do not satisfy the admissibility condition for low-rank.\\ \hline
    $\mathcal{IL}(i^{(l)})$ & Interaction list of hypercube $i^{(l)}$ that consists of children of $i^{(l)}$'s parent's neighbors that are not its neighbors.\\ \hline
  \end{tabularx}
  \caption{Neighbors and interaction list of hypercube $i$ at level $l$}
    \label{table:Notations1b}
 \end{table}

\begin{figure}[!htbp]
\centering
            \subfloat[Level=0]{
    \begin{tikzpicture}[scale=0.5]
        \draw[ultra thick,fill=orange] (-2,-2) rectangle (2,2);
    \end{tikzpicture}
}
            \subfloat[Level=1]{
\begin{tikzpicture}[scale=0.5]
    \draw[ultra thick] (-2,-2) rectangle (2,2);
    \draw[ultra thick,fill=red!80] (-2,-2) rectangle (0,0);
    \draw[ultra thick,fill=red!80] (0,-2) rectangle (2,0);
    \draw[ultra thick,fill=red!80] (0,0) rectangle (2,2);
    \draw[ultra thick,fill=red!80] (-2,0) rectangle (0,2);
    \draw[ultra thick,fill=orange] (-2,-2) rectangle (0,0);
\end{tikzpicture}
}
            \subfloat[Level=2]{
\begin{tikzpicture}[scale=0.5]
    \draw[ultra thick,fill=cyan!50] (-2,-2) rectangle (2,2);
    \draw[ultra thick,fill=red!80] (-2,-2) rectangle (1,1);
    \draw[ultra thick,fill=cyan!50] (-2,-2) grid (2,2);
    \draw[ultra thick,fill=red!80] (-2,-2) grid (1,1);
    \draw[ultra thick,fill=orange] (-1,-1) rectangle (0,0);    
\end{tikzpicture}
}
            \subfloat[Level=3]{
\begin{tikzpicture}[scale=0.5]
\draw[ultra thick,fill=cyan!50] (-2,-2) rectangle (1,1);
    \draw[fill=red!80] (-1,-1) rectangle (0.5,0.5);
    \draw[step=2, ultra thick,fill=cyan!50] (-2,-2) grid (2,2);
    \draw[thick,fill=cyan!50] (-2,-2) grid (2,2);
    \draw[step=0.5, fill=red!80] (-2,-2) grid (2,2);
    \draw[fill=orange] (-0.5,-0.5) rectangle (0,0);
\end{tikzpicture}
}
        \subfloat{
		\label{Himg:des}
		\begin{tikzpicture}
			[
			box/.style={rectangle,draw=black, minimum size=0.25cm},scale=0.25
			]
			\node[box,fill=orange,label=right:\small{Box $B$}, anchor=west] at (0,4){};
			\node[box,fill=orange,label=right:\small{$\&$},  anchor=west] at (0,2){};
                \node[box,fill=red!80,label=right:\small{Neighbors of $B$}, anchor=west] at (3,2){};
			\node[box,fill=cyan!50,label=right:\small{$\mathcal{IL}$ of $B$}, anchor=west] at (0,0){};
		\end{tikzpicture}  
}

	\caption{Illustration of neighbors and interaction list at different levels in 2D.}
	\label{fig:adm_condn}
\end{figure}
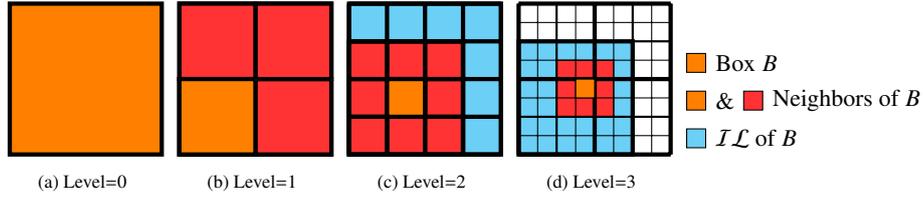

Let $K^{(l)}$ be the FMM matrix representation of $A$ at level $l$. 
The sub-matrix notation in 2D is shown in Equation~\eqref{eq:subMatrix}, where the ordering of boxes as shown in Figure~\ref{fig:numbering} is followed. The notation $K_{ab}^{(l)}$ represents the interaction between the source points and target points of nodes $b$ and $a$ of level $l$ respectively. Consider the sub-matrix $K_{01}^{(1)}$. It gets approximated at level $2$, as in Equation~\eqref{eq:FMM_LR2}, where only those interactions between boxes which are in each other's interaction list are approximated by a low-rank matrix. In this way, at each level, all the interactions between boxes which are in each other's interaction list are approximated by a low-rank matrix. The resulting low-rank structure of the matrix $A$ at levels 2 and 3 arising in 2D problems is shown in Figure~\ref{fig:lowRankStructure}.

\begin{equationwithLineno}\label{eq:subMatrix}
    A = K^{(0)} = K^{(1)} = \begin{bmatrix}
K_{00}^{(1)} & K_{01}^{(1)} & K_{02}^{(1)} & K_{03}^{(1)}\\
K_{10}^{(1)} & K_{11}^{(1)} & K_{12}^{(1)} & K_{13}^{(1)}\\
K_{20}^{(1)} & K_{21}^{(1)} & K_{22}^{(1)} & K_{23}^{(1)}\\
K_{30}^{(1)} & K_{31}^{(1)} & K_{32}^{(1)} & K_{33}^{(1)}
\end{bmatrix}
\end{equationwithLineno}

   \begin{align}\label{eq:FMM_LR}
    K_{01}^{(1)} &= \begin{bmatrix}
K_{04}^{(2)} & K_{05}^{(2)} & K_{06}^{(2)} & K_{07}^{(2)}\\
K_{14}^{(2)} & K_{15}^{(2)} & K_{16}^{(2)} & K_{17}^{(2)}\\
K_{24}^{(2)} & K_{25}^{(2)} & K_{26}^{(2)} & K_{27}^{(2)}\\
K_{34}^{(2)} & K_{35}^{(2)} & K_{36}^{(2)} & K_{37}^{(2)}
\end{bmatrix}\\
&\approx \begin{bmatrix}
U_{0}^{(2)}A_{04}^{(2)} V_{4}^{(2)^{*}} & U_{0}^{(2)}A_{05}^{(2)} V_{5}^{(2)^{*}} & U_{0}^{(2)}A_{06}^{(2)} V_{6}^{(2)^{*}} & U_{0}^{(2)}A_{07}^{(2)} V_{7}^{(2)^{*}}\\
K_{14}^{(2)} & U_{1}^{(2)}K_{15}^{(2)} V_{5}^{(2)^{*}} & U_{1}^{(2)}K_{16}^{(2)} V_{6}^{(2)^{*}} & K_{17}^{(2)}\\
K_{24}^{(2)} & U_{2}^{(2)}K_{25}^{(2)} V_{5}^{(2)^{*}} & U_{2}^{(2)}K_{26}^{(2)} V_{6}^{(2)^{*}} & K_{27}^{(2)}\\
U_{3}^{(2)}K_{34}^{(2)} V_{4}^{(2)^{*}} & U_{3}^{(2)}K_{35}^{(2)} V_{5}^{(2)^{*}} & U_{3}^{(2)}K_{36}^{(2)} V_{6}^{(2)^{*}} & U_{3}^{(2)}K_{37}^{(2)} V_{7}^{(2)^{*}}
\end{bmatrix}\label{eq:FMM_LR2}
\end{align}

\begin{figure}[!htbp]
    \centering
    \begin{subfigure}[b]{0.3\textwidth}
    \includegraphics[width=\linewidth]{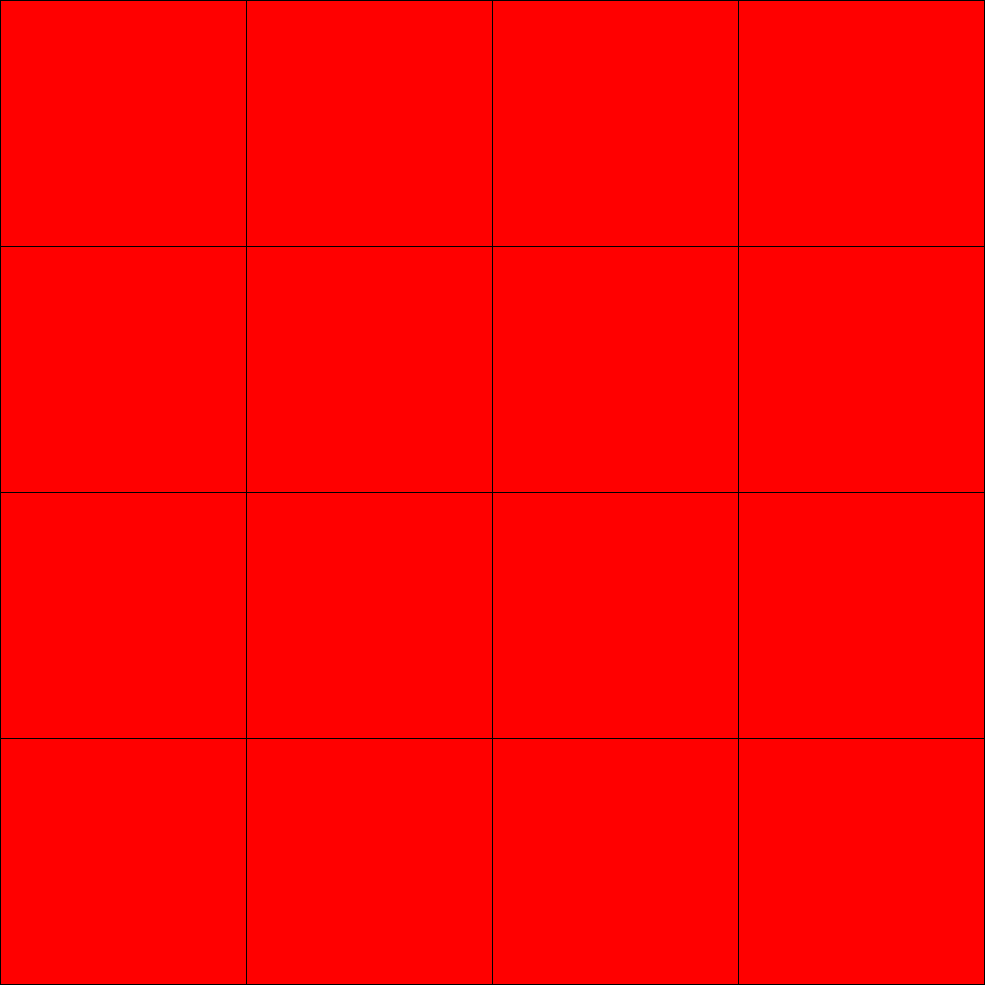}
    \end{subfigure}
    \begin{subfigure}[b]{0.3\textwidth}
    \includegraphics[width=\linewidth]{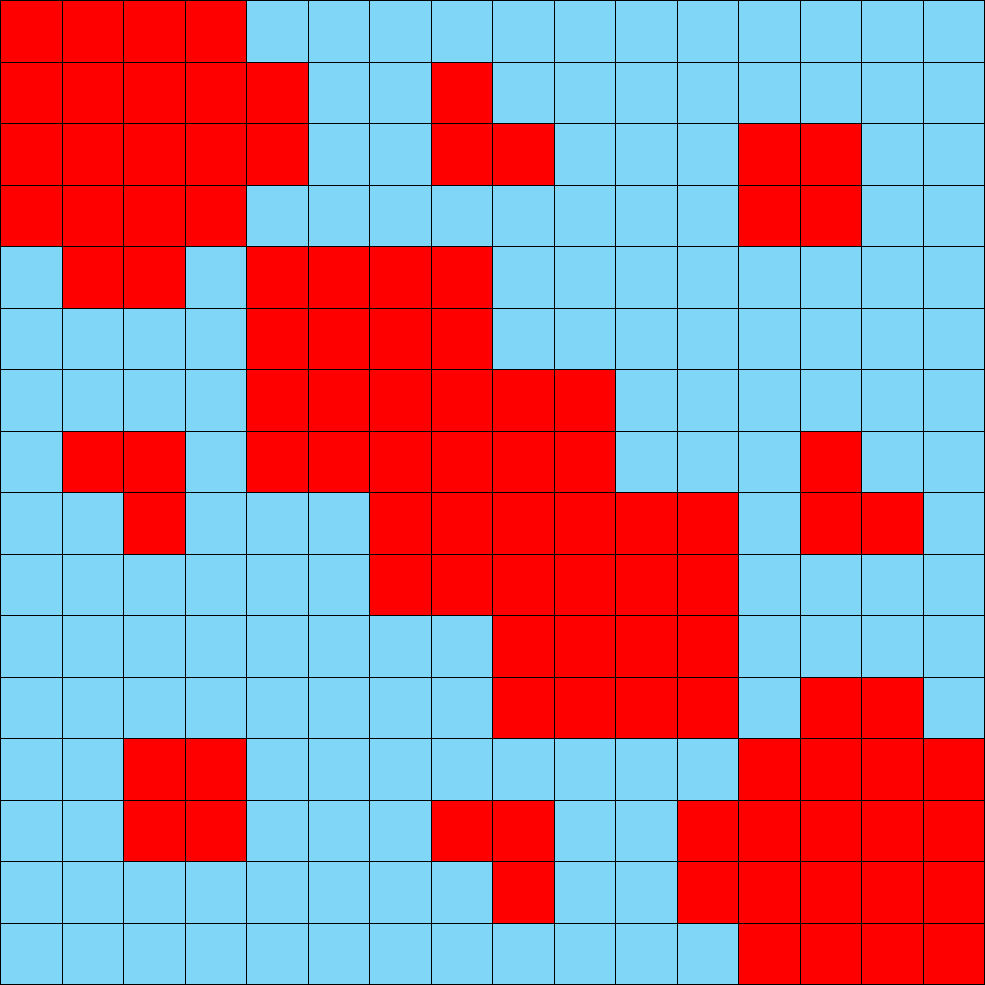}
    \end{subfigure}
    \begin{subfigure}[b]{0.3\textwidth}
    \includegraphics[width=\linewidth]{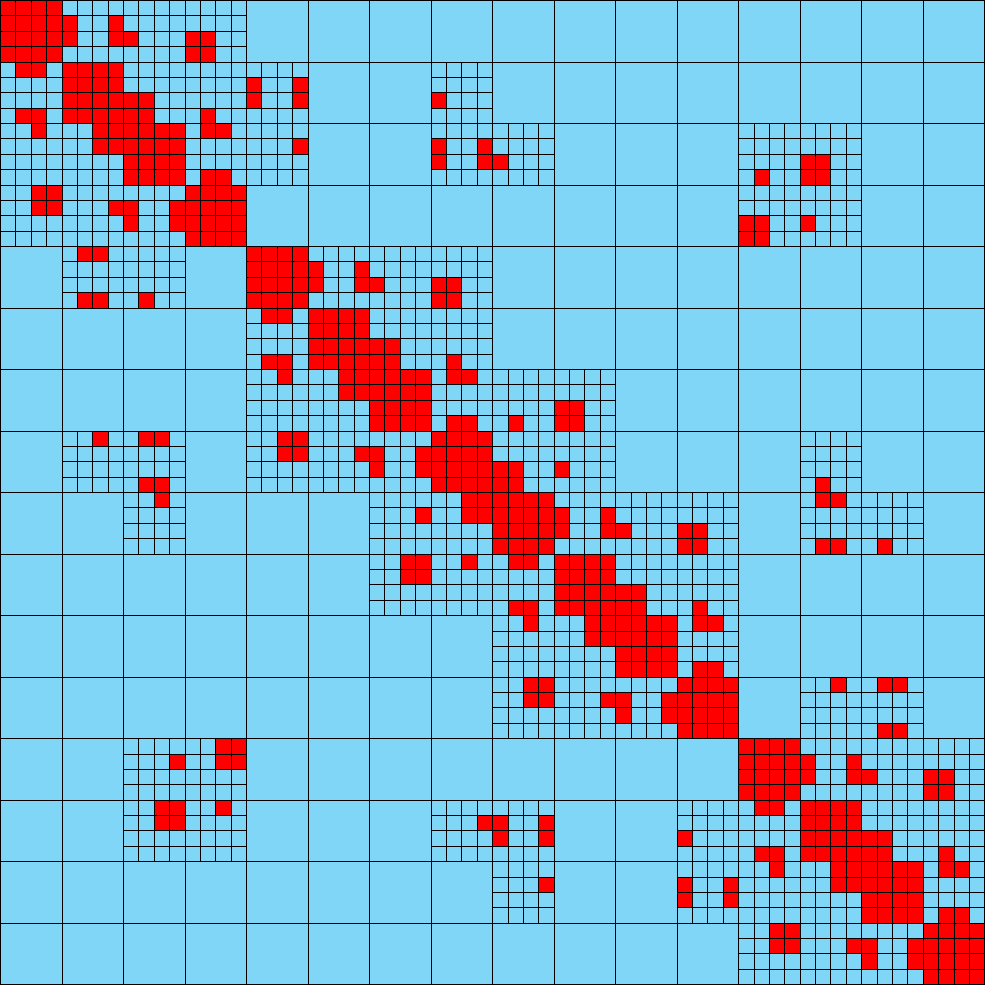}
    \end{subfigure}
    \caption{Low-Rank structure of the matrix $A$ at levels 2 and 3 arising in 2D problems}
    \label{fig:lowRankStructure}
\end{figure}

\subsection{Assembly of FMM operators}\label{ssec:NCA}
To assemble the various FMM operators we use a new Nested Cross Approximation (NNCA)~\cite{gujjula2022nca}, a nested version of Adaptive Cross Approximation (ACA), which produces nested bases.
The low-rank approximation of an admissible sub-block $A_{t^{X}s^{Y}}$, with a controlled error of $\mathcal{O}(\epsilon_{A})$, using NNCA takes the form
\begin{equationwithLineno} \label{eq:NCA}
  A_{t^{X}s^{Y}} \approx A_{t^{X}s^{X,i}} (A_{t^{X,i}s^{X,i}})^{-1} A_{t^{X,i}s^{Y,o}}(A_{t^{Y,o}s^{Y,o}})^{-1} A_{t^{Y,o}s^{Y}}
\end{equationwithLineno}
where $t^{X,i}$, $s^{X,i}$, $t^{Y,o}$, and $s^{Y,o}$ are termed pivots, and $t^{X,i}\subset t^{X}$, $s^{X,i}\subset \mathcal{F}^{X,i}$, $t^{Y,o}\subset \mathcal{F}^{Y,o}$ and $s^{Y,o}\subset s^{Y}$.
$\mathcal{F}^{X,i}$ and $\mathcal{F}^{Y,o}$ are defined as
\begin{equationwithLineno}\label{eq:Fj}
  \mathcal{F}^{X,i} = \{s^{X'}:X'\in \mathcal{IL}(X^{(l)})\} \text{ and}
\end{equationwithLineno}
\begin{equationwithLineno}\label{eq:Fi}
  \mathcal{F}^{Y,o} = \{t^{Y'}:Y'\in \mathcal{IL}(Y^{(l)})\}.
\end{equationwithLineno}

For more details on the construction of NNCA, error bounds, and the technique we use in identifying the pivots, we direct the readers to~\cite{gujjula2022nca,bebendorf2012constructing}. Here we summarise the various FMM operators that are constructed using NNCA.
\begin{alignat}{3}
    & V_{\scriptscriptstyle B}^{{(L)^{*}}} &&= (A_{t^{B,o} s^{B,o}})^{-1} A_{t^{B,o}s^{B}} &&\text{\;\; of $B^{(L)}$}\\
    & V_{\scriptscriptstyle B}^{{\dagger(l)}^{*}} &&= (A_{t^{P,o} s^{P,o}})^{-1} A_{t^{P,o}s^{B,o}} &&\text{\;\; of $B^{(l)}$ where $P^{(l)}=\mathcal{P}(B^{(l)})$}\\
    & A_{\scriptscriptstyle BD}^{(l)} &&= A_{t^{Bi}s^{Do}} &&\text{\;\; of $B^{(l)}$ where $D\in\mathcal{IL}(B^{(l)})$} \\
    & U_{\scriptscriptstyle B}^{\dagger(l)} &&=  A_{t^{B,i}s^{P,i}} (A_{t^{P,i} s^{P,i}})^{-1}  &&\text{\;\; of $B^{(l)}$}\\
    & U_{\scriptscriptstyle B}^{(L)} &&= A_{t^{B}s^{B,i}} (A_{t^{B,i} s^{B,i}})^{-1} &&\text{\;\; of $B^{(L)}$} \\
    & K_{\scriptscriptstyle BX}^{(l)} &&= A_{t^{B}s^{X}} &&\text{\;\; of $B^{(l)}$ where $X\in\mathcal{N}(B^{(l)})$}
\end{alignat}
where $l\in\{0,1,2,\mathellipsis L\}$. We describe each of these operators in Table~\ref{table:Notations2}.

   \begin{table}[!htbp]
\centering
  \begin{tabularx}{\textwidth}{|l|X|}
    \hline
    $V_{\scriptscriptstyle B}^{{(L)}^{*}}$ & P2M (Particle To Multipole) operator of hypercube $B^{(L)}$ that translates the sources of hypercube $B^{(L)}$ to its multipoles\\ \hline
    $V_{\scriptscriptstyle B}^{{\dagger(l)}^{*}}$ & M2M (Multipole To Multipole) operator of hypercube $B^{(l)}$ that translates the multipoles of hypercube $B^{(l)}$ to it parent's multipoles \\ \hline
    $A_{\scriptscriptstyle BD}^{(l)}$ & M2L (Multipole To Local) operator between hypercubes $B^{(l)}$ and $D^{(l)}$ that finds the locals (local potential) of hypercube $B^{(l)}$ due to the multipoles of hypercube $D^{(l)}$ \\ \hline
    $U_{\scriptscriptstyle B}^{\dagger(l)}$ & L2L (Local To Local) operator of hypercube $B^{(l)}$ that translates the locals of its parent to its locals \\ \hline
    $U_{\scriptscriptstyle B}^{(L)}$ & L2P (Local To Particle) operator of hypercube $B^{(L)}$ that translates its locals to its potential \\ \hline
    $K_{\scriptscriptstyle BX}^{(l)}$ & P2P (Particle To Particle) operator between hypercubes $B^{(l)}$ and $X^{(l)}$ that finds the potential in hypercube $B^{(l)}$ due to the sources in hypercube $X^{(l)}$ \\ \hline
  \end{tabularx}
  \caption{Various FMM operators}
    \label{table:Notations2}
 \end{table}

\section{The algebraic inverse fast multipole method (AIFMM)}\label{sec:AIFMM}
AIFMM has three main steps. The first step is to construct the extended sparse matrix from the given matrix using NNCA. The second step is to perform elimination. The third step is to find the unknowns using back substitution. We now describe each of these in the following subsections.

\subsection{Construction of the extended sparse system}
The construction of the extended sparse matrix representation of the dense matrix $A$ involves
\begin{enumerate}
    \item the construction of a $2^{d}$ hierarchical tree of depth $L$, $\mathcal{T}^{(L)}$, as described in Subsection~\ref{ssec:Tree}.
    \item the identification of neighbors and interaction list of each hypercube at all levels of the tree, as described in Subsection~\ref{ssec:admCond}.
    \item the introduction of auxiliary variables:
    \begin{enumerate}
        \item multipoles at levels $2\leq l\leq L$, i.e., $\{y^{(l)}\}_{l=2}^{L}$ where \\$y^{(l)}=[y_{i_{1}}^{(l)}; y_{i_{2}}^{(l)}; \mathellipsis y_{i_{2^{dl}}}^{(l)}]$ and $\{y_{i_{c}}^{(l)}\}$ indicates the multipoles of hypercube $i_{c}^{(l)}$.
        \item locals at levels $2\leq l\leq L$, i.e., $\{z^{(l)}\}_{l=2}^{L}$ where \\$z^{(l)}=[z_{i_{1}}^{(l)}; z_{i_{2}}^{(l)}; \mathellipsis z_{i_{2^{dl}}}^{(l)}]$ and $\{z_{i_{c}}^{(l)}\}$ indicates the locals of hypercube $i_{c}^{(l)}$.
    \end{enumerate}
    Here we followed MATLAB notation to represent the column vectors $y^{(l)}$ and $z^{(l)}$. 
    \begin{remark}
        For any hypercube at levels $0$ and $1$, its interaction list is empty, so only the multipoles and locals of hypercubes at levels $2\leq l\leq L$ are considered to be the auxiliary variables. 
    \end{remark}

\end{enumerate}

\subsubsection{Unknown variables of the extended sparse system}
For a hypercube $i^{(L)}$, its particles $x_{i}^{(L)}$ are defined as
\begin{equationwithLineno}
    x_{i}^{(L)} = x(t^{i})
\end{equationwithLineno}
For all non-leaf levels, the multipoles at the child level are interpreted to be the particles at the parent level, i.e., we define the particles of a hypercube $B^{(l)}$ where $l<L$, to be the multipoles of its children, as defined in Equation~\eqref{eq:particleAtNonLeafLevel}.
\begin{equationwithLineno}\label{eq:particleAtNonLeafLevel}
x_{B}^{(l)} = [y_{B_{1}}^{{(l+1)}}; y_{B_{2}}^{{(l+1)}};\mathellipsis y_{B_{2^{d}}}^{{(l+1)}}], \;\; B_{c}^{(l+1)}\in\mathcal{C}(B^{(l)})\;\;\forall c\in\{1,2,\mathellipsis,2^{d}\}
\end{equationwithLineno}
Here MATLAB notation is followed to represent the column vector $x_{B}^{(l)}$.
Accordingly, the child M2M operators get combined to form the parent's P2M operator, and similarly, the child L2L operators get combined to form the parent's L2P operator. This has been written in detail in Table~\ref{table:Notations2_2}.

   \begin{table}[!htbp]
\centering
  \begin{tabularx}{\textwidth}{|l|X|}
    \hline$V_{\scriptscriptstyle B}^{{(l)}^{*}}$ & P2M (Particle To Multipole) operator of hypercube $B^{(l)}$ that translates the particles of hypercube $B^{(l)}$ to its multipoles. $V_{\scriptscriptstyle B}^{{(l)}^{*}}=[V_{\scriptscriptstyle B_{1}}^{{\dagger(l+1)}^{*}}\; V_{\scriptscriptstyle B_{2}}^{{\dagger(l+1)}^{*}} \mathellipsis V_{\scriptscriptstyle B_{2^{d}}}^{{\dagger(l+1)}^{*}}]$ where $l\in\{0,1,2\mathellipsis, L-1\}$, $B_{c}^{(l+1)}\in\mathcal{C}(B^{(l)})$, and $c\in\{1,2,3,\mathellipsis,2^{d}\}$\\ \hline
    $U_{\scriptscriptstyle B}^{(l)}$ & L2P (Local To Particle) operator of hypercube $B^{(l)}$ that translates the locals of hypercube $B^{(l)}$ to the potential of its children. $U_{\scriptscriptstyle B}^{{(l)}}=[U_{\scriptscriptstyle B_{1}}^{{\dagger(l+1)}}\; U_{\scriptscriptstyle B_{2}}^{{\dagger(l+1)}} \mathellipsis U_{\scriptscriptstyle B_{2^{d}}}^{{\dagger(l+1)}}]$ where $l\in\{0,1,2\mathellipsis, L-1\}$, $B_{c}^{(l+1)}\in\mathcal{C}(B^{(l)})$, and $c\in\{1,2,3,\mathellipsis,2^{d}\}$\\ \hline
  \end{tabularx}
  \caption{P2M and L2P operators at non-leaf level}
    \label{table:Notations2_2}
 \end{table}
 
Each hypercube $i^{(l)}$ at level $2\leq l\leq L$ is therefore associated with the unknown variables described in Table~\ref{table:Notations3}.
   \begin{table}[!htbp]
\centering
  \begin{tabular}{|l|l|}
    \hline
     $x_{i}^{(l)}$ & particles of hypercube $i^{(l)}$\\ \hline
     $y_{i}^{(l)}$ & multipoles of hypercube $i^{(l)}$\\ \hline
     $z_{i}^{(l)}$ & locals of hypercube $i^{(l)}$\\ \hline
  \end{tabular}
  \caption{Unknown variables of the extended sparse system}
    \label{table:Notations3}
 \end{table}

\subsubsection{Governing equations of the extended sparse system}
The equations governing the multipoles and the potential are given below.

At the leaf level,
\begin{alignat}{3}
    & y_{i}^{(L)} &&= V_{i}^{{(L)}^{*}}x_{i}^{(L)} && \text{ denoted as Equation $\text{Y}_{i}^{(L)}$}\\
    & b_{i}^{(L)} &&= U_{i}^{(L)}z_{i}^{(L)} + \sum_{j^{(L)}\in\mathcal{N}(i^{(L)})}K_{ij}^{(L)} x_{j}^{(L)} && \text{ denoted as Equation $\text{X}_{i}^{(L)}$}\label{eq:EqX}
\end{alignat}
where $b_{i}^{(L)} = [b_{c_{1}}, b_{c_{2}},\mathellipsis,b_{c_{m_{i}}}] \text{ and } \{c_{e}\}_{e=1}^{m_{i}}= t^{i}$.

For $l\in\{0,\mathellipsis,L-1\}$,
\begin{subequations}
\begin{alignat}{3}
    & y_{i}^{(l)} && = \sum_{i'^{(l)}\in\mathcal{C}(i^{(l)})}V_{i'}^{\dagger(l)^{*}}y_{i'}^{(l+1)}\label{eq:particlesatl} && \\
    &            && = V_{i}^{{(l)}^{*}}x_{i}^{(l)}\label{eq:particlesatl2}   && \text{ denoted as Equation $\text{Y}_{i}^{(l)}$}
\end{alignat}\
\end{subequations}
wherein we have combined the multipoles at a child level to form the particles at the parent level as in Equation~\eqref{eq:particleAtNonLeafLevel}.

For $l\in\{0,1,\mathellipsis,L\}$, the equation governing the locals is given by
\begin{equationwithLineno}\label{eq:EqXl}
    z_{i}^{(l)} = U_{i}^{\dagger(l)}z_{i'}^{(l-1)} + \sum_{j^{(l)}\in IL(i^{(l)})}A_{ij}^{(l)}y_{j}^{(l)} 
\end{equationwithLineno}
where $i'^{(l-1)}=\mathcal{P}(i^{(l)})$. 

For $l\in\{1,2,\mathellipsis,L-1\}$, Equation~\eqref{eq:EqXl} takes the form of Equation~\eqref{eq:EqX}, written out in Equation~\eqref{eq:EqXl2}, when the multipoles at level $l$ are combined to form the particles at level $l-1$ as in Equation~\eqref{eq:particleAtNonLeafLevel}.
\begin{equationwithLineno}
 b_{i}^{(l)} = U_{i}^{(l)}z_{i}^{(l)} + \sum_{j^{(l)}\in N(i^{(l)})}K_{ij}^{(l)}x_{j}^{(l)}  \text{ denoted as Equation $\text{X}_{i}^{(l)}$}\label{eq:EqXl2}
\end{equationwithLineno}
Here $z_{i}^{(1)}=\vec{0}$, $U_{i}^{(l)}$ is defined in Table~\ref{table:Notations2} and
\begin{equationwithLineno}
    K_{ij}^{(l)} = 
    \begin{bmatrix}
        A_{i_{1}j_{1}}^{(l+1)} & A_{i_{1}j_{2}}^{(l+1)} & \hdots & A_{i_{1}j_{2^{d}}}^{(l+1)}\\
        A_{i_{2}j_{1}}^{(l+1)} & A_{i_{2}j_{2}}^{(l+1)} & \hdots & A_{i_{2}j_{2^{d}}}^{(l+1)}\\
        \vdots & \ddots &        & \vdots\\
        A_{i_{2^{d}}j_{1}}^{(l+1)} & A_{i_{2^{d}}j_{2}}^{(l+1)} & \hdots & A_{i_{2^{d}}j_{2^{d}}}^{(l+1)}        
    \end{bmatrix},\;
        b_{i}^{(l)} = 
    \begin{bmatrix}
        z_{i_{1}}^{(l+1)}\\
        z_{i_{2}}^{(l+1)}\\
        \vdots\\
        z_{i_{2^{d}}}^{(l+1)}
    \end{bmatrix}.
\end{equationwithLineno}

    \begin{tcolorbox}[colback=white!5!white,colframe=black!75!black,title=\textbf{Summary of equations $\forall l\in\{1,2,\mathellipsis,L\}$}]
\begin{alignat}{3}
    & y_{i}^{(l)} &&= V_{i}^{{(l)}^{*}}x_{i}^{(l)} && \text{ denoted as Equation $\text{Y}_{i}^{(l)}$}\label{eq:EquationY}\\
    & b_{i}^{(l)} &&= U_{i}^{(l)}z_{i}^{(l)} + \sum_{j^{(l)}\in\mathcal{N}(i^{(l)})}K_{ij}^{(l)} x_{j}^{(l)} && \text{ denoted as Equation $\text{X}_{i}^{(l)}$}
    \label{eq:EquationX}
\end{alignat} 
    \end{tcolorbox}

The various FMM operators: L2L/L2P, M2M/P2M, and M2L are obtained using NNCA as described in Subsection~\ref{ssec:NCA}.

The system of equations with the unknowns $x$ and the auxiliary variables is given in Equation~\eqref{eq:extendedMatrix}.
\begin{equationwithLineno}\label{eq:extendedMatrix}
   \tilde{\tilde{A}}
    \begin{bmatrix}
        x\\
        z^{(L)}\\
        y^{(L)}\\
        z^{(L-1)}\\
        y^{(L-1)}\\
        \vdots\\
        z^{(l)}\\
        y^{(l)}\\
        \vdots\\
        z^{(2)}\\
        y^{(2)}
    \end{bmatrix}= 
    \begin{bmatrix}
        b\\
        0\\
        0\\
        0\\
        0\\
        \vdots\\
        0\\
        0\\
        \vdots\\
        0\\
        0
    \end{bmatrix}
\end{equationwithLineno}
The ordering of equations or rows is given by
\begin{equationwithLineno*}
\begin{split}
\{ & \text{X}_{i_1}^{(L)}, \text{ Y}_{i_1}^{(L)}, \text{ X}_{i_2}^{(L)}, \text{ Y}_{i_2}^{(L)}, \mathellipsis, \text{ X}_{i_{2^{dL}}}^{(L)}, \text{ Y}_{i_{2^{dL}}}^{(L)}, \\
   & \text{ X}_{i_1}^{(L-1)}, \text{ Y}_{i_1}^{(L-1)}, \text{ X}_{i_2}^{(L-1)}, \text{ Y}_{i_2}^{(L-1)}, \mathellipsis, \text{ X}_{i_{2^{d(L-1)}}}^{(L-1)}, \text{ Y}_{i_{2^{d(L-1)}}}^{(L-1)}, \\
   & \mathellipsis,\\
   & \text{ X}_{i_1}^{(2)}, \text{ Y}_{i_1}^{(2)}, \text{ X}_{i_2}^{(2)}, \text{ Y}_{i_2}^{(2)}, \mathellipsis, \text{ X}_{i_{2^{2d}}}^{(2)}, \text{ Y}_{i_{2^{2d}}}^{(2)}, \\
   & \text{ X}_{i_1}^{(1)}, \text{ X}_{i_2}^{(1)}, \mathellipsis, \text{ X}_{i_{2^{d}}}^{(1)}
\}.
\end{split}
\end{equationwithLineno*}
A reordering of unknowns or columns of Equation~\eqref{eq:extendedMatrix} is performed, such that $z^{(l)}$ is interleaved in between $x^{(l)}$ for all levels $l$ from $L$ to $2$ as follows\\
\begin{equationwithLineno*}
\begin{split}
[ & x_{i_1}^{{(L)}^{*}} z_{i_1}^{{(L)}^{*}} x_{i_2}^{{(L)}^{*}} z_{i_2}^{{(L)}^{*}} \mathellipsis x_{i_{2^{dL}}}^{{(L)}^{*}} z_{i_{2^{dL}}}^{{(L)}^{*}} \\
  & x_{i_1}^{{(L-1)}^{*}} z_{i_1}^{{(L-1)}^{*}} x_{i_2}^{{(L-1)}^{*}} z_{i_2}^{{(L-1)}^{*}} \mathellipsis x_{i_{2^{d(L-1)}}}^{{(L-1)}^{*}} z_{i_{2^{d(L-1)}}}^{{(L-1)}^{*}} \\
  & \mathellipsis \\ 
  & x_{i_1}^{{(2)}^{*}} z_{i_1}^{{(2)}^{*}} x_{i_2}^{{(2)}^{*}} z_{i_2}^{{(2)}^{*}} \mathellipsis x_{i_{2^{2d}}}^{{(2)}^{*}} z_{i_{2^{2d}}}^{{(2)}^{*}} \\ 
  & x_{i_1}^{{(1)}^{*}} x_{i_2}^{{(1)}^{*}} \mathellipsis x_{i_{2^{d}}}^{{(1)}^{*}}
]^{*}.
\end{split}
\end{equationwithLineno*}
It ensures that, when an elimination in standard ordering is performed, the fill-ins occur symmetrically. 
Let the new system after reordering be 
\begin{equationwithLineno}\label{eq:extendedSparseSystem}
\tilde{A}\tilde{x}=\tilde{b}
\end{equationwithLineno}
The structures of the matrices $\tilde{A}$ constructed out of a 2D problem at levels 2 and 3 are illustrated in Figures~\ref{fig:extendedSparseLevel2} and~\ref{fig:extendedSparse} respectively.

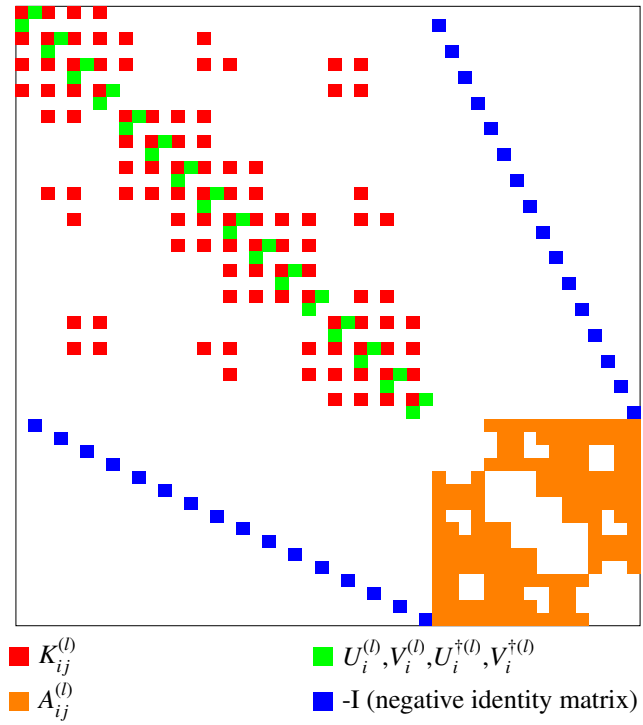
\begin{figure}[!htbp]
    \centering
        \begin{subfigure}[b]{\textwidth}
            \centering
    \includegraphics[width=0.5\linewidth]{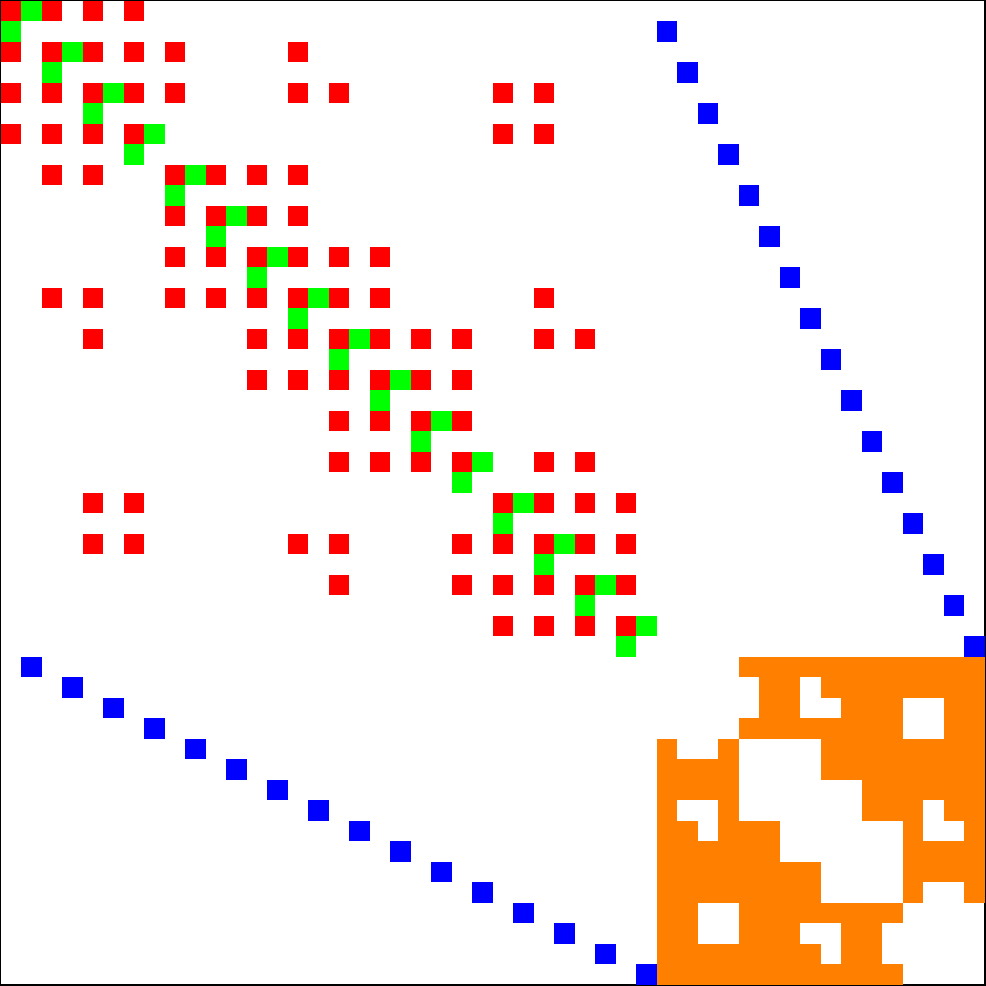}
        \end{subfigure}
        \newline
        \begin{subfigure}[b]{0.72\textwidth}
            \centering
        \begin{tikzpicture}
            [
            box/.style={rectangle,draw=none, minimum size=0.25cm},scale=0.5
            ]
            \node[box,fill=red,label=right:$K_{ij}^{(l)}$, anchor=west] at (-4,6){};
            \node[box,fill=orange,label=right:$A_{ij}^{(l)}$, anchor=west] at (-4,4.8){};
            \node[box,fill=green,label=right:$U_{i}^{(l)}\text{,}V_{i}^{(l)}\text{,}U_{i}^{\dagger(l)}\text{,}V_{i}^{\dagger(l)}$, anchor=west] at (4,6){};
            \node[box,fill=blue,label=right:-I (negative identity matrix), anchor=west] at (4,4.8){};
        \end{tikzpicture}
        \end{subfigure}
    \caption{Structure of the extended sparse matrix at level 2 arising in 2D problems}
    \label{fig:extendedSparseLevel2}
\end{figure}

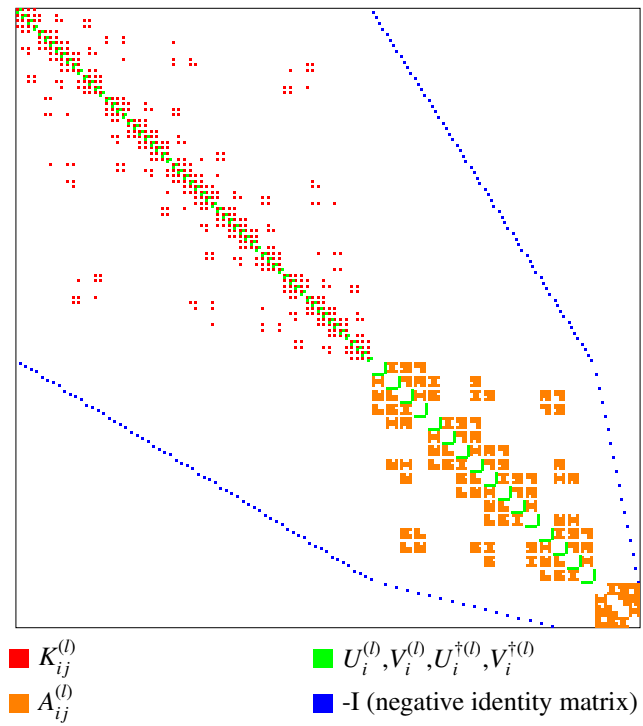
\begin{figure}[!htbp]
    \centering
        \begin{subfigure}[b]{\textwidth}
            \centering
    \includegraphics[width=0.5\linewidth]{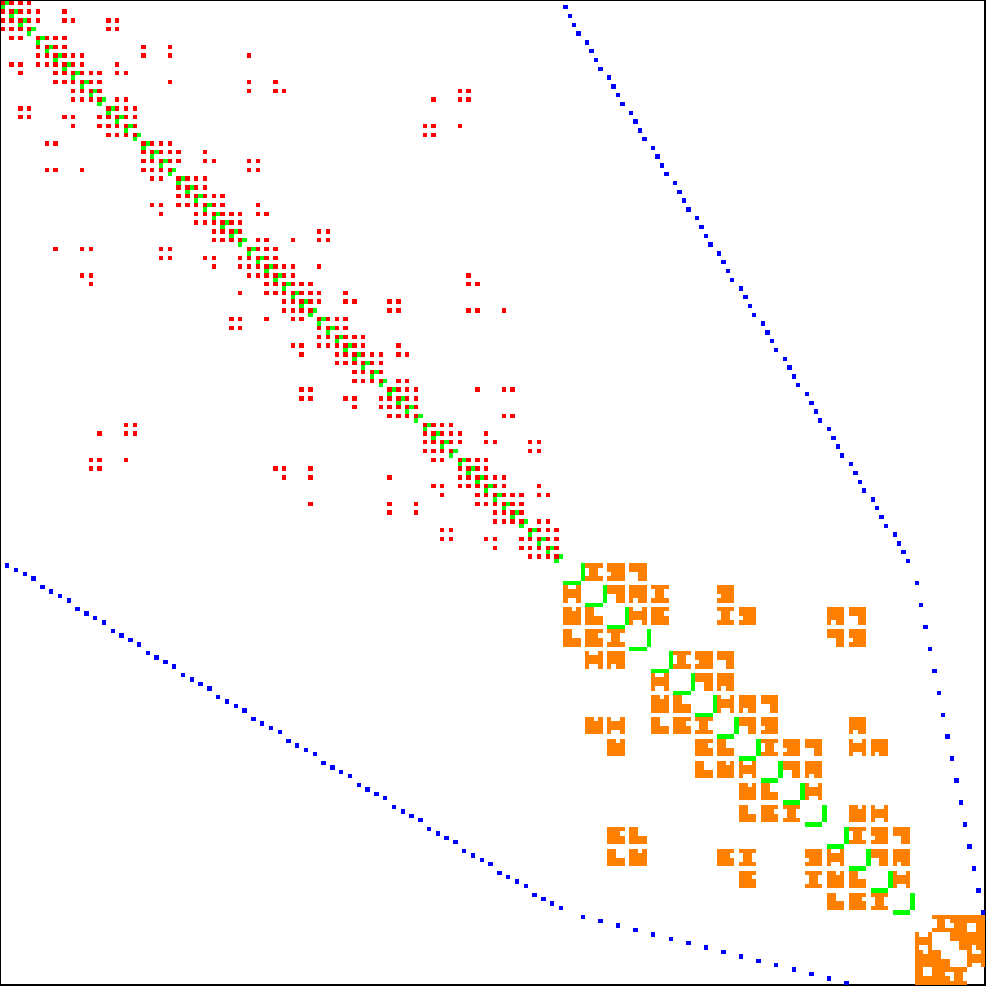}
        \end{subfigure}
        \newline
        \begin{subfigure}[b]{0.72\textwidth}
            \centering
        \begin{tikzpicture}
            [
            box/.style={rectangle,draw=none, minimum size=0.25cm},scale=0.5
            ]
            \node[box,fill=red,label=right:$K_{ij}^{(l)}$, anchor=west] at (-4,6){};
            \node[box,fill=orange,label=right:$A_{ij}^{(l)}$, anchor=west] at (-4,4.8){};
            \node[box,fill=green,label=right:$U_{i}^{(l)}\text{,}V_{i}^{(l)}\text{,}U_{i}^{\dagger(l)}\text{,}V_{i}^{\dagger(l)}$, anchor=west] at (4,6){};
            \node[box,fill=blue,label=right:-I (negative identity matrix), anchor=west] at (4,4.8){};
        \end{tikzpicture}
        \end{subfigure}
    \caption{Structure of the extended sparse matrix at level 3 arising in 2D problems}
    \label{fig:extendedSparse}
\end{figure}

\subsection{Elimination or the Factorization Phase}
Given the extended sparse system of Equations~\eqref{eq:extendedSparseSystem}, the next task is to solve for the unknowns $\tilde{x}$. It is solved using Gaussian Elimination followed by Back Substitution. 
The elimination phase is not the naive Gaussian Elimination but the elimination process is interleaved with the compression and redirection of fill-ins corresponding to well-separated hypercubes through non-zero entries. In this subsection, we describe this process. 

We show in Figures~\ref{fig:GraphLevel2} and~\ref{fig:graph}, the graphical representation of the extended sparse matrix, constructed out of a 2D problem at levels 2 and 3 respectively. For better clarity, we show in Figure~\ref{fig:cropGraph} a partial graph of the graph shown in Figure~\ref{fig:graph}. 
The nodes correspond to the variables and the incoming edges to a node constitute an equation.
The same color notation followed in Figures~\ref{fig:extendedSparseLevel2} and~\ref{fig:extendedSparse} is followed in Figures~\ref{fig:GraphLevel2}, ~\ref{fig:cropGraph} and~\ref{fig:graph}.

\begin{figure}[!htbp]
    \centering
    \includegraphics[width=0.7\linewidth]{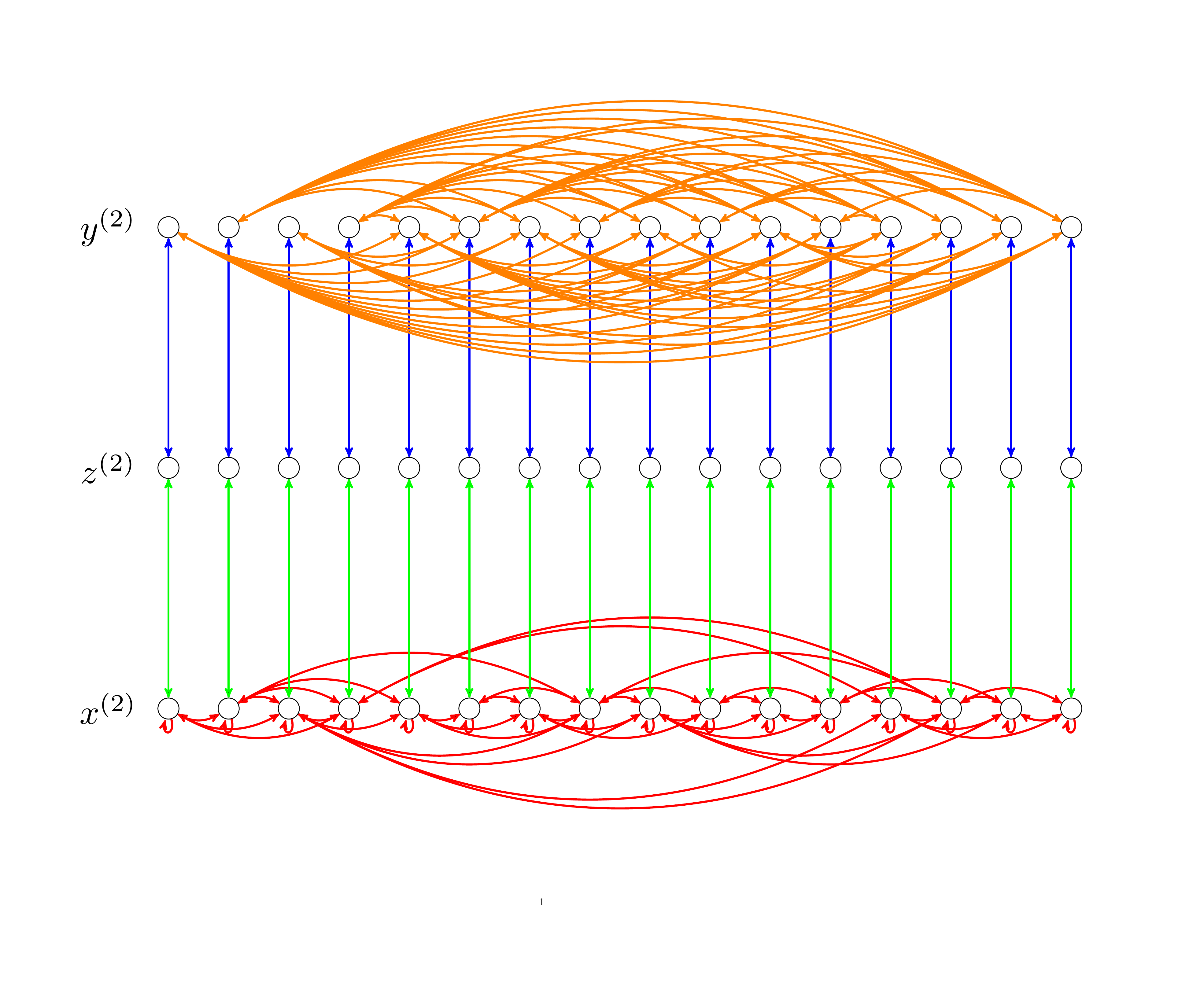}
    \caption{Graph of the Extended Sparse system at level 2}
    \label{fig:GraphLevel2}
\end{figure}

\begin{figure}[!htbp]
    \centering
    \includegraphics[width=0.8\linewidth]{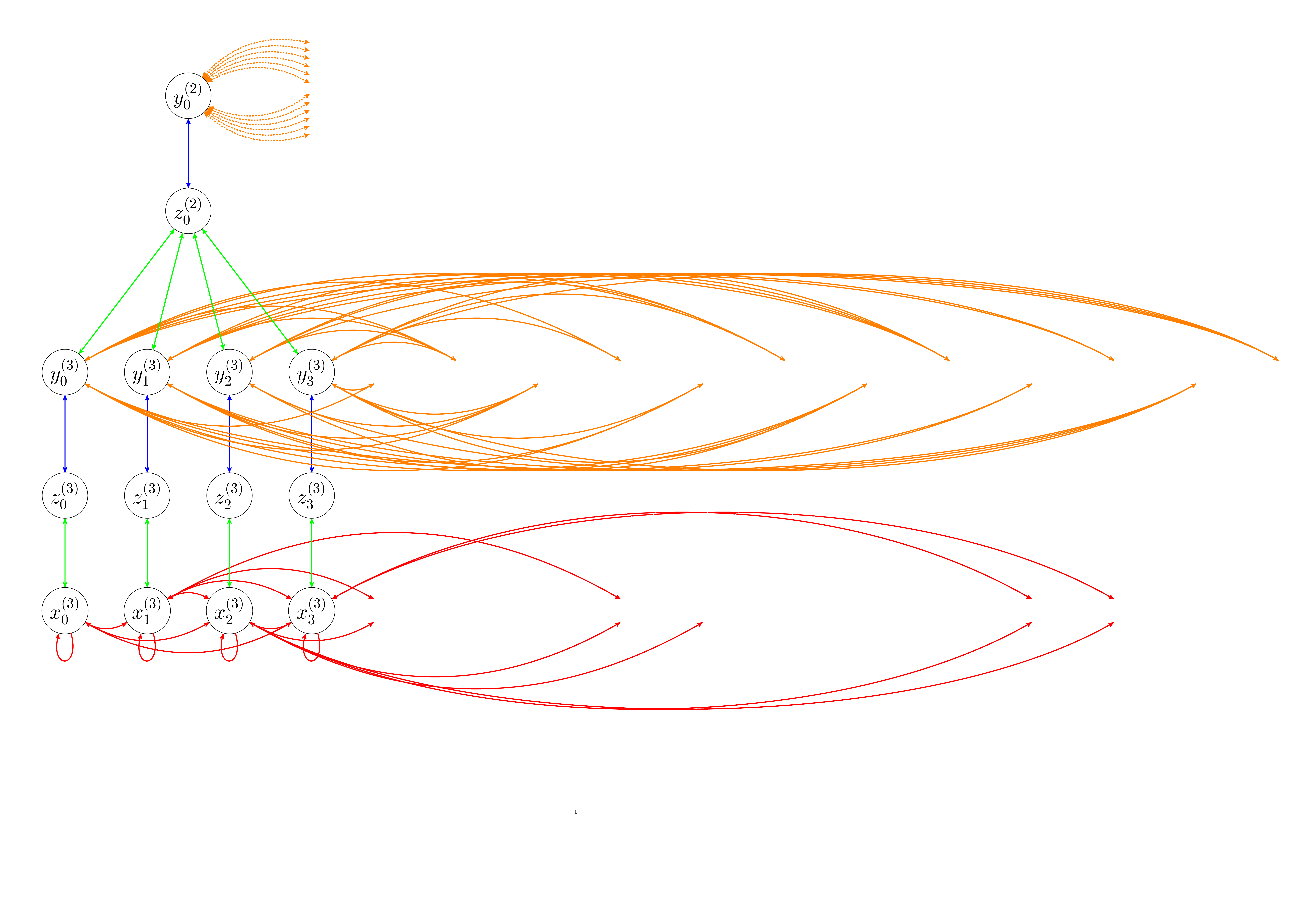}
    \caption{Partial graph of the Extended Sparse system at level 3, showing few nodes and edges for better readability}
    \label{fig:cropGraph}
\end{figure}

\begin{landscape}
\begin{figure}[!htbp]
    \centering
    \includegraphics[width=0.9\linewidth]{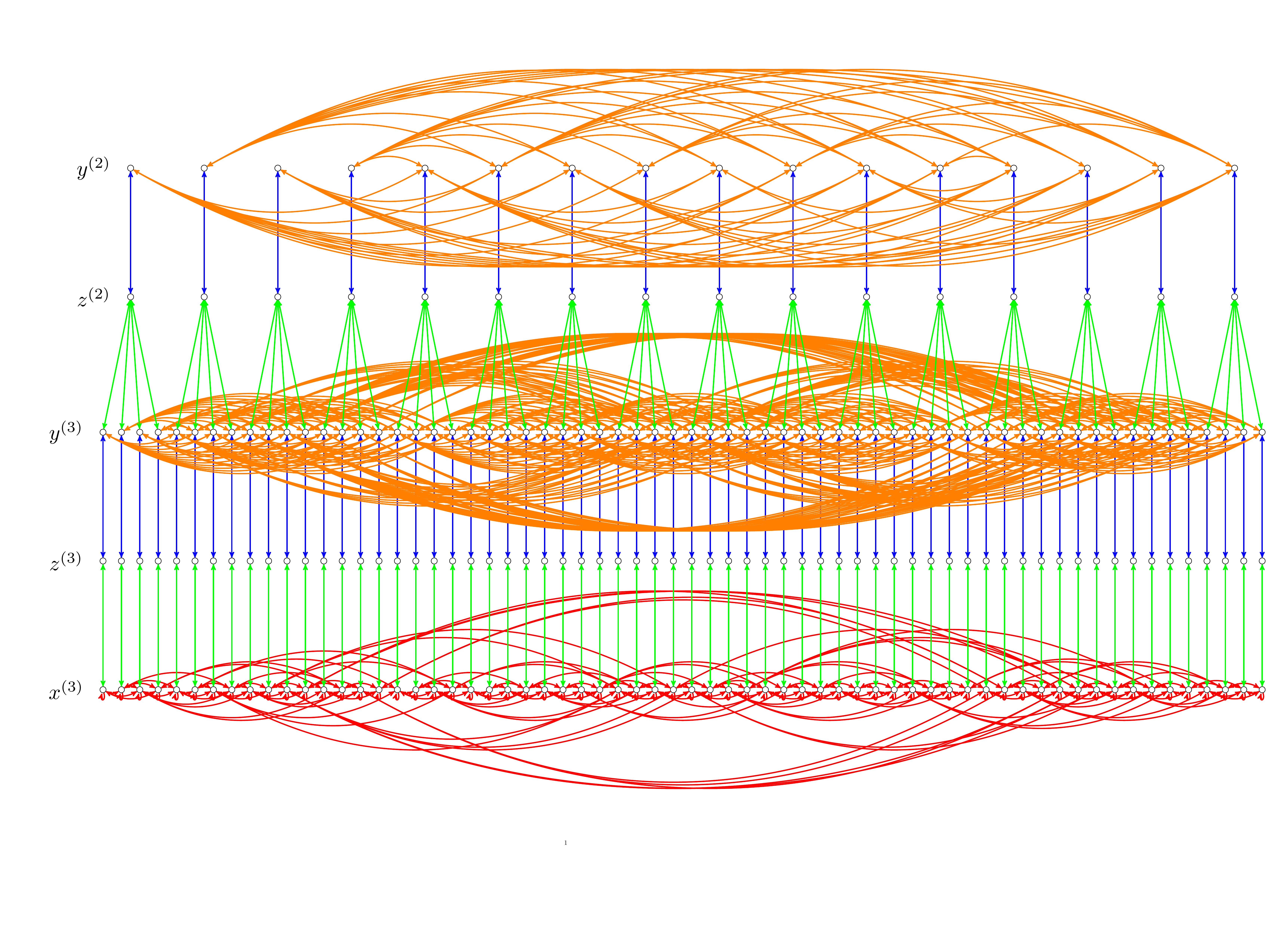}
    \caption{Graph of the Extended Sparse system at level 3}
    \label{fig:graph}
\end{figure}
\end{landscape}


We eliminate the variables in standard ordering, i.e., the order in which the variables are arranged. 
When variables $x_{i}^{(l)}$ and $z_{i}^{(l)}$ get eliminated, which we term as the hypercube getting eliminated, it results in an update of the graph that involves nullification of some edges, updation of some edges and creation of new edges. The new edges that get created are termed fill-ins. The various fill-ins that get created are described in Table~\ref{table:Notations}.
   \begin{table}[!htbp]
\centering
  \begin{tabularx}{\textwidth}{|l|X|}
    \hline
    $P2P_{ij}^{(l)}$ & The P2P fill-in that represents the potential of hypercube $i^{(l)}$ due to the particles of hypercube $j^{(l)}$\\ \hline  
    $M2P_{ij}^{(l)}$ & The M2P fill-in that represents the potential of hypercube $i^{(l)}$ due to the multipoles of hypercube $j^{(l)}$\\ \hline  
    $P2L_{ij}^{(l)}$ & The P2L fill-in that represents the local potential of hypercube $i^{(l)}$ due to the particles of hypercube $j^{(l)}$ \\ \hline 
    $M2L_{ij}^{(l)}$ & The M2L fill-in that represents the local potential of hypercube $i^{(l)}$ due to the multipoles of hypercube $j^{(l)}$\\ \hline  
  \end{tabularx}
  \caption{Fill-in interactions}
    \label{table:Notations}
 \end{table}

Upon elimination of hypercube $i^{(l)}$, fill-ins among its neighboring hypercubes get created as described below. For $p^{(l)},q^{(l)}\in\mathcal{N}(i^{(l)})$,
\begin{enumerate}
    \item if $p^{(l)}$ and $q^{(l)}$ have been eliminated, $M2L_{pq}^{(l)}$ and $M2L_{qp}^{(l)}$ get created.  
    \item if $p^{(l)}$ has been eliminated and $q^{(l)}$ has not been eliminated, fill-ins $P2L_{pq}^{(l)}$ and $M2P_{qp}^{(l)}$ get created.  
    \item if $q^{(l)}$ has been eliminated and $p^{(l)}$ has not been eliminated, fill-ins $P2L_{qp}^{(l)}$ and $M2P_{pq}^{(l)}$ get created.  
    \item if $p^{(l)}$ and $q^{(l)}$ have not been eliminated, $P2P_{pq}^{(l)}$ and $P2P_{qp}^{(l)}$ get created.
\end{enumerate}
For a more detailed understanding of the fill-in creation, we refer the readers to the graphs in~\cite{ambikasaran2014inverse}. 

\begin{theorem}\label{th:fillin}
    Consider a hypercube $i^{(L)}$ and hypercubes $\{j^{(L)},k^{(L)}\}\in\mathcal{N}(i^{(L)})$ such that $j^{(L)}\in\mathcal{IL}(k^{(L)})$. Let $N$ target points and $N$ source points be distributed uniformly in each of the hypercubes. 
  If the hypercubes $j^{(L)}$ and $k^{(L)}$ are not eliminated from the extended sparse system, then the P2P fill-in $P2P^{(L)}_{jk}$ that gets created upon elimination of hypercube $i^{(L)}$ is rank deficient.
\end{theorem}
\begin{proof}
Upon elimination of hypercube $i^{(L)}$ from the extended sparse system, the P2P fill-in $P2P^{(L)}_{jk}$ gets created as $P2P^{(L)}_{jk} := -K^{(L)}_{ji}K^{(L)}_{ii}K^{(L)}_{ik}$. The rank of $K_{ii}^{(L)}$ is $N$, as it is a self interaction. From~\cite{khan2022numerical}, the rank of interaction between particles of hypercubes that (i) share a vertex scales as
$\mathcal{O}\left(\log(N)\log^{d}(\log(N))\right)$; (ii) share a hypersurface of dim $d'$ scales as $\mathcal{O}\left(N^{\frac{d'}{d}}\log^{d}(N)\right)$, $d'\in\{1,2,\mathellipsis d-1\}$. 
\begin{itemize}
\item 
If at least one of the hypercubes $j^{(L)}, k^{(L)}$ shares a vertex with hypercube $i^{(L)}$ then 
       \begin{align}
    \text{rank}(P2P^{(L)}_{jk}) &\leq \min\{\text{rank}(K^{(L)}_{ji}), \text{rank}(K^{(L)}_{ii}), \text{rank}(K^{(L)}_{ik}\}\\
    &\leq \mathcal{O}\left(\log(N)\log^{d}(\log(N))\right).
    \end{align}  
\item
If $j^{(L)}$ and $k^{(L)}$ share a hypersurface of dim $d_{j}\in\{1,2,\mathellipsis d-1\}$ and $d_{k}\in\{1,2,\mathellipsis d-1\}$ with hypercube $i^{(L)}$ respectively, then
       \begin{align}
    \text{rank}(P2P^{(L)}_{jk}) &\leq \min\{\text{rank}(K^{(L)}_{ji}), \text{rank}(K^{(L)}_{ii}), \text{rank}(K^{(L)}_{ik}\}\\
    &\leq \mathcal{O}\left(N^{\frac{d'}{d}}\log^{d}(N)\right)
    \end{align}  
    where $d'=\min\{d_{j}, d_{k}\}$ and $d'\in\{1,2,\mathellipsis d-1\}$.
\end{itemize}
\end{proof}


We show in Theorem~\ref{th:fillin}, under the assumption that the particles are uniformly distributed, that the P2P fill-ins corresponding to well-separated hypercubes at leaf level are rank deficient.
We assume that this is true at higher levels as well and also when the particles are distributed non-uniformly.
Further, the bounds obtained in Theorem~\ref{th:fillin} are very conservative, as the numerical illustrations in~\cite{ambikasaran2014inverse} show that the ranks are almost constant.

The ranks of P2L and M2P fill-ins do not scale with $N$, as they are equal to the number of locals and the number of multiples respectively.

In conclusion, a fill-in corresponding to an interaction between well-separated hypercubes is low-rank and therefore can be efficiently approximated by a low-rank  matrix. Further, the compression is redirected through existing operators as described later in the section. 

In the process of elimination, due to the creation of fill-ins,
and due to the compression and redirection of fill-ins corresponding to well-separated hypercubes, 
Equation~\eqref{eq:EquationX} gets modified as
\begin{alignat}{3}
     b_{i}^{(l)} = U_{i}^{(l)}z_{i}^{(l)} + \sum_{j^{(l)}\in\mathcal{N}(i^{(l)})}\left((1-E_{j}^{(l)})K_{ij}^{(l)} x_{j}^{(l)} + E_{j}^{(l)}M2P_{ij}^{(l)} y_{j}^{(l)}\right) \label{eq:EquationX2}
\end{alignat} 
where $E_{j}^{(l)}$ takes values $0$ or $1$. It being $1$, indicates that node $j^{(l)}$ is eliminated and $0$, indicates that node $j^{(l)}$ is not eliminated. 
We continue the process of elimination until when the multipoles at level $2$ are the only variables left. This entire elimination process is described in Algorithm~\ref{alg:elimination2}.

 \begin{algorithm}[!htbp]
	\caption{Elimination algorithm}\label{alg:elimination2}
	\begin{algorithmic}[1]
		\Procedure{Elimination}{$n_{\max}$,$\epsilon_{A}$}
		\State{} \Comment{$n_{\max}$ is the maximum number of particles at leaf level}
		\State {Form $\mathcal{T}^L$, where $L= \min \left\{l : \abs{x^{(l)}_{i}} < n_{\max};\forall \text{ hypercubes $i$ at level $l$}\right\}$}
		\State {Perform NNCA with tolerance $\epsilon_{A}$ to find the $L2L/L2P, M2L, M2M/P2M$ operators of all hypercubes at all levels}
            \For{\texttt{$l=L:2$}} 
				\For{\texttt{$i=e_{1}^{(l)},e_{2}^{(l)},\mathellipsis,e_{2^{dl}}^{(l)}$}} 
				    \State{Eliminate $x_{i}^{(l)}$ and $z_{i}^{(l)}$ from the extended sparse system using Equations~\eqref{eq:EquationX2} and~\eqref{eq:EquationY} }
				    \State{$E_{i}^{(l)}:=1$}
					\For{$(p^{(l)},q^{(l)})\text{ in }\{(r^{(l)},s^{(l)}): r^{(l)}\in \mathcal{N}(i^{(l)}),s\in\mathcal{N}(i^{(l)})\}$}
						\If{$p^{(l)}$ is eliminated}
    						\If{$q^{(l)}$ is eliminated}
						        \State Results in the update of $M2L_{pq}^{(l)}$ and $M2L_{qp}^{(l)}$
					        \Else
					            \State Results in fill-ins $P2L_{pq}^{(l)}$ and $M2P_{qp}^{(l)}$
					            \If{$p^{(l)}$ and $q^{(l)}$ are well-separated}
                                        \State Compress $P2L_{pq}^{(l)}$ and $M2P_{qp}^{(l)}$ and update the relevant operators as in Subsections~\ref{ssec:P2L_fill_in} and~\ref{ssec:M2P_fill_in}.
					            \EndIf
				            \EndIf
						\Else
                            \If{$q^{(l)}$ is eliminated}
						        \State Results in fill-ins $P2L_{qp}^{(l)}$ and $M2P_{pq}^{(l)}$
                                        \If{$p^{(l)}$ and $q^{(l)}$ are well-separated}
                                            \State Compress $P2L_{qp}^{(l)}$ and $M2P_{pq}^{(l)}$ and update the relevant operators as in Subsections~\ref{ssec:P2L_fill_in} and~\ref{ssec:M2P_fill_in}.
					                \EndIf
					        \Else
					            \State Results in the update of $P2P_{pq}^{(l)}$ and $P2P_{qp}^{(l)}$
					            \If{$p^{(l)}$ and $q^{(l)}$ are well-separated}
                                        \State Compress $P2P_{pq}^{(l)}$ and $P2P_{qp}^{(l)}$ and update the relevant operators as in Subsection~\ref{ssec:P2P_fill_in}
					            \EndIf
				            \EndIf
			            \EndIf
					\EndFor
				\EndFor
			\EndFor
		\EndProcedure
	\end{algorithmic}
\end{algorithm}


\subsubsection{Compression and redirection of P2P fill-in}\label{ssec:P2P_fill_in}
Consider a P2P fill-in $P2P_{ij}$ where
hypercubes $i$ and $j$ at level $l$ are well-separated.
$P2P_{ij}$ can be efficiently approximated by a low-rank matrix and this interaction can be redirected through an already existing interaction 
via the path $x_{j}\rightarrow z_{j}\rightarrow y_{j} \rightarrow y_{i}\rightarrow z_{i} \rightarrow x_{i}$ as shown in Figure~\ref{fig:P2P_compression}. 
This redirection results in an update of (i) P2M $V_{j}^{*}$; (ii) M2L $A_{ij}$; (iii) L2P $U_{i}$; (iv) M2M $V_{j}^{\dagger *}$; 
(v) L2L $U_{i}^{\dagger}$; 
(vi) Other M2Ls $\{A_{ic}: c\in \{\mathcal{IL}(i)\backslash j\}\bigcup \mathcal{N}(i)\}$ and $\{A_{dj}: d\in \{\mathcal{IL}(j)\}\bigcup \mathcal{N}(j)\}$. We now describe how each of these updates is done.

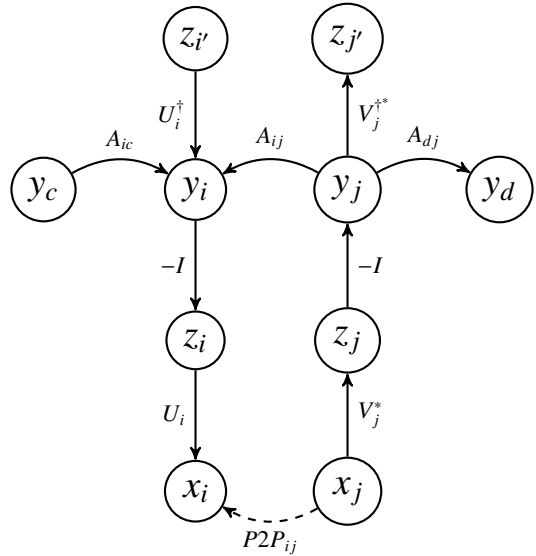
\begin{figure}[!htp]
  \begin{center}
\begin{tikzpicture}[->,>=stealth',auto,node distance=2cm,
  thick,main node/.style={circle,draw,font=\sffamily\Large\bfseries}]

  \node[main node] (1) {$y_{c}$};
  
  \node[main node] (2) [right of=1] {$y_{i}$};
  \node[main node] (3) [right of=2] {$y_{j}$};

  \node[main node] (2a) [below of=2] {$z_{i}$};
  \node[main node] (3a) [below of=3] {$z_{j}$};

  \node[main node] (4) [right of=3] {$y_{d}$};
  \node[main node] (5) [below of=2a] {$x_{i}$};
  \node[main node] (6) [below of=3a] {$x_{j}$};
  \node[main node] (7) [above of=2] {$z_{i'}$};
  \node[main node] (8) [above of=3] {$z_{j'}$};
  
  \path[every node/.style={font=\sffamily\small}]
    (1) edge[draw=black, bend left] node [color=black, pos=0.5, above] {$A_{ic}$} (2)
    (3) edge[draw=black, bend right] node [color=black, pos=0.5, above] {$A_{ij}$} (2)
    (3) edge[draw=black, bend left] node [color=black, pos=0.5, above] {$A_{dj}$} (4)
    (2a) edge[draw=black,<-] node [color=black, pos=0.5,left] {$-I$} (2)
    (3a) edge[draw=black,->] node [color=black, pos=0.5,right] {$-I$} (3)
    (2a) edge[draw=black] node [color=black, pos=0.5,left] {$U_{i}$} (5)
    (6) edge[draw=black] node [color=black, pos=0.5,right] {$V_{j}^{*}$} (3a)
    (7) edge[draw=black] node [color=black, pos=0.5,left] {$U_{i}^{\dagger}$} (2)
    (3) edge[draw=black] node [color=black, pos=0.5,right] {$V_{j}^{\dagger^{*}}$} (8)
    (6) edge[draw=black, bend left, dashed] node [color=black, pos=0.5, below] {$P2P_{ij}$}  (5);
\end{tikzpicture}
  \end{center}
\caption{Illustration of $P2P_{ij}$ compression and its redirection. The fill-in is shown through a dashed edge.
Here $i'=\mathcal{P}(i)$ and $j'=\mathcal{P}(j)$. (It is to be noted that a part of the graph, with only a few connections that get updated as a result of the fill-in redirection, is shown here.)}
\label{fig:P2P_compression}
\end{figure}

\textbf{P2M, M2L, L2P update.}
First, we find the new orthogonal column basis, that spans the existing column basis of $i$ at level $l$, i.e., $U_{i}$, and the columns of the fill-in $P2P_{ij}$ by finding the RRQR decomposition of the augmented matrix
\begin{equationwithLineno}
    [U_{i} | P2P_{ij}] = \tilde{U}_{i} L_{ij}.
\end{equationwithLineno}
A threshold $\epsilon_{A}$ is used as an input to the RRQR routine, such that the relative residual in the decomposition is equal to $\mathcal{O}(\epsilon_{A})$. Considering the matrix $L_{ij}$ to be an augmented matrix of the form $L_{ij} \equiv [L_{i} | \tilde{L}_{ij}]$, results in 
\begin{subequations}
   \begin{align}
    U_{i} &= \tilde{U}_{i} L_{i} \label{eq:P2P_2a}\\
    P2P_{ij} &= \tilde{U}_{i} \tilde{L}_{ij},\label{eq:P2P_2}
\end{align}  
\end{subequations}
Next, we find the orthogonal row basis, that spans the existing row basis of $j$ at levek $l$, i.e., $V_{j}^{*}$, and a row basis of the fill-in $P2P_{ij}$, $\tilde{L}_{ij}$, by finding the RRQR decomposition of the augmented matrix
\begin{equationwithLineno}
    [V_{j} |\tilde{L}_{ij}^{*}] = \tilde{V}_{j} R_{ij}
\end{equationwithLineno}
By expressing $R_{ij}$ as the augmented matrix $[R_{j} | \tilde{R}_{ij}]$, we have
\begin{subequations}
   \begin{align}
    V_{j} &= \tilde{V}_{j} R_{j}, \label{eq:P2P_4a}\\
    \tilde{L}^{*}_{ij} &= \tilde{V}_{j} \tilde{R}_{ij},\label{eq:P2P_4b}
\end{align}  
\end{subequations}
Using Equations~\eqref{eq:P2P_2} and~\eqref{eq:P2P_4b}, we have,
\begin{equationwithLineno}\label{eq:P2P_fill_in}
    P2P_{ij}=\tilde{U}_{i}\tilde{L}_{ij}
    =\tilde{U}_{i}\tilde{R}_{ij}^{*}\tilde{V}_{j}^{*}.
\end{equationwithLineno}
Using Equations~\eqref{eq:P2P_fill_in}, ~\eqref{eq:P2P_2a}, and~\eqref{eq:P2P_4a}, $U_{i} A_{ij}V_{j}^{*} + P2P_{ij}$ can be expressed as
   \begin{align}
    U_{i} A_{ij}V_{j}^{*} + P2P_{ij} &= \tilde{U}_{i} L_{i} A_{ij} R_{j}^{*}\tilde{V}_{j}^{*}  + \tilde{U}_{i}\tilde{R}_{ij}^{*}\tilde{V}_{j}^{*}\\
    &=\tilde{U}_{i}(L_{i} A_{ij}R_{j}^{*}+\tilde{R}_{ij}^{*}) \tilde{V}_{j}^{*}.
\end{align}  

We then make the following assignments, which update the old operators with the new ones.
   \begin{align}
    A_{ij}&:= L_{i} A_{ij}R_{j}^{*}+\tilde{R}_{ij}^{*}\\
    U_{i} &:= \tilde{U}_{i}\\
    V_{j} &:= \tilde{V}_{j}
\end{align}  

\textbf{Other M2L updates.}
For $c\in \{\mathcal{IL}(i)\backslash j\}\bigcup \mathcal{N}(i)$, the value of the old potential due to $y_{c}$ at particles $x_{i}$ should be equal to the value of the new potential due to $y_{c}$ at particles $x_{i}$, as in equation~\ref{eq:otherM2LS_P2PFillin}, because the potential due to $y_{c}$ at $x_{i}$ is not dependent on the fill-in between hypercubes $i$ and $j$ at level $l$.
\begin{equationwithLineno}\label{eq:otherM2LS_P2PFillin}
U_{i}A_{ic}y_{c} = \tilde{U}_{i}\tilde{A}_{ic}y_{c}
\end{equationwithLineno}
As Equation~\eqref{eq:otherM2LS_P2PFillin} holds true $\forall y_{c}\in\mathbb{C}^{k\times 1}$, it can be equivalently written as
\begin{equationwithLineno}\label{eq:otherM2LS_P2PFillin_2}
U_{i}A_{ic} = \tilde{U}_{i}\tilde{A}_{ic}
\end{equationwithLineno}
Further, since $\tilde{U}_{i}^{*}\tilde{U}_{i} = I$, Equation~\eqref{eq:otherM2LS_P2PFillin_2}, can be written as
\begin{equationwithLineno}
\tilde{A}_{ic} = \tilde{U}_{i}^{*}U_{i}A_{ic}=L_{i}A_{ic}.
\end{equationwithLineno}
We then make the following assignment, which updates the old operator with the new one.
\begin{equationwithLineno}
A_{ic} := \tilde{A}_{ic}
\end{equationwithLineno}
Similarly, for $d\in \{\mathcal{IL}(j)\}\bigcup \mathcal{N}(j)\}$, the value of the old locals due to $x_{j}$ should be equal to the value of the new locals due to $x_{j}$, because the locals of $d$ due to particles of $j$ is not dependent on the fill-in between $i$ and $j$.
\begin{equationwithLineno}\label{eq:otherM2LS_P2PFillin2}
A_{dj}V_{j}^{*}x_{j} = \tilde{A}_{dj}\tilde{V}_{j}^{*}x_{j}
\end{equationwithLineno}
As Equation~\eqref{eq:otherM2LS_P2PFillin2} holds true $\forall x_{j}\in\mathbb{C}^{k\times 1}$, it can be equivalently written as
\begin{equationwithLineno}\label{eq:otherM2LS_P2PFillin2_2}
A_{dj}V_{j}^{*} = \tilde{A}_{dj}\tilde{V}_{j}^{*}
\end{equationwithLineno}
Further, since $\tilde{V}_{j}^{*}\tilde{V}_{j} = I$, Equation~\eqref{eq:otherM2LS_P2PFillin2_2}, can be written as
\begin{equationwithLineno}
\tilde{A}_{dj} = A_{dj}V_{j}^{*}\tilde{V}_{j} = A_{dj}R_{j}^{*}.
\end{equationwithLineno}
We then make the following assignment, which updates the old operator with the new one.
\begin{equationwithLineno}
A_{dj} := \tilde{A}_{dj}
\end{equationwithLineno}

\textbf{M2M update.}
The old and new multipoles of hypercube $j$ at level $l$ are given by \\
   \begin{align}
y_{j} &= V^{*}_{j}x_{j}, \label{eq:oldV}\\
\tilde{y}_{j} &= \tilde{V}^{*}_{j}x_{j} \label{eq:newV}
\end{align}  
respectively.
The fill-in $P2P_{ij}$ has no influence on the multipoles $y_{j'}$. So, the old and the new contribution of the multipoles of $j$ at the multipoles of its parent $j'$ must be equal and hence it follows that
\begin{equationwithLineno}\label{eq:P2M}
{V^{\dagger^{*}}_{j}} y_{j} = \tilde{V}^{\dagger^{*}}_{j} \tilde{y}_{j}
\end{equationwithLineno}
Using Equations~\eqref{eq:oldV}, ~\eqref{eq:newV}, and,~\eqref{eq:P2M}
\begin{equationwithLineno}\label{eq:M2Mupdate}
{V^{\dagger^{*}}_{j}} V^{*}_{j}x_{j} = \tilde{V}^{\dagger^{*}}_{j} \tilde{V}^{*}_{j}x_{j}
\end{equationwithLineno}
As Equation~\eqref{eq:M2Mupdate} holds true $\forall x_{j}\in\mathbb{C}^{k\times 1}$, it can be equivalently written as
\begin{equationwithLineno}\label{eq:M2Mupdate2}
{V^{\dagger^{*}}_{j}} V_{j}^{*} = \tilde{V}^{\dagger^{*}}_{j} \tilde{V}_{j}^{*}
\end{equationwithLineno}
Further, since $\tilde{V}_{j}^{*}\tilde{V}_{j} = I$, Equation~\eqref{eq:M2Mupdate2} can be written as
\begin{equationwithLineno}
\tilde{V}^{\dagger^{*}}_{j}  = V^{\dagger^{*}}_{j} V_{j}^{*} \tilde{V}_{j}  = V^{\dagger^{*}}_{j} R_{j}^{*}.
\end{equationwithLineno}
We then make the following assignment, which updates the old operator with the new one.
\begin{equationwithLineno}
{V^{\dagger^{*}}_{j}} := {\tilde{V}^{\dagger^{*}}_{j}} 
\end{equationwithLineno}

\textbf{L2L update.}
A similar analysis as done in updating the M2M on the L2L operator results in its update as follows:
\begin{equationwithLineno}
 U_{i}^{\dagger} := \tilde{U}_{i}^{\dagger} = {\tilde{U}_{i}}^{*}U_{i}U_{i}^{\dagger}= L_{i}U_{i}^{\dagger}.
\end{equationwithLineno}

\subsubsection{Compression and redirection of P2L fill-in}\label{ssec:P2L_fill_in}
Consider a P2L fill-in $P2L_{ij}$ where 
hypercubes $i$ and $j$ at level $l$ are well-separated.
Then $P2L_{ij}$ can be efficiently approximated by a low-rank matrix and this interaction can be redirected through an already existing interaction 
via the path $x_{j}\rightarrow z_{j}\rightarrow y_{j} \rightarrow y_{i}$, as shown in the Figure~\ref{fig:P2L_compression}. 
This redirection results in an update of (i) P2M $V^{*}_{j}$; (ii) M2L $A_{ij}$; (iii) M2M $V_{j}^{\dagger}$; 
(iv) Other M2Ls $\{A_{dj}: d\in \{\mathcal{IL}(j)\backslash i\}\bigcup \mathcal{N}(j)\}$. We now describe how each of these updates is done.

\begin{figure}[H]
  \begin{center}
\begin{tikzpicture}[->,>=stealth',auto,node distance=2cm,
  thick,main node/.style={circle,draw,font=\sffamily\Large\bfseries}]

  \node[main node] (2) [right of=1] {$y_{i}$};
  \node[main node] (3) [right of=2] {$y_{j}$};
  \node[main node] (4) [right of=3] {$y_{d}$};
  \node[main node] (3a) [below of=3] {$z_{j}$};
  \node[main node] (6) [below of=3a] {$x_{j}$};
  \node[main node] (8) [above of=3] {$z_{j'}$};
  
  \path[every node/.style={font=\sffamily\small}]
    (3) edge[draw=black, bend right] node [color=black, pos=0.5, above] {$A_{ij}$} (2)
    (3) edge[draw=black, bend left] node [color=black, pos=0.5, above] {$A_{dj}$} (4)
    (3) edge[draw=black,<-] node [color=black, pos=0.5,right] {$-I$} (3a)
    (6) edge[draw=black, dashed] node [color=black, pos=0.5,left] {$P2L_{ij}$} (2)
    (6) edge[draw=black] node [color=black, pos=0.5,right] {$V_{j}^{*}$} (3a)
    (3) edge[draw=black] node [color=black, pos=0.5,right] {$V_{j}^{\dagger^{*}}$} (8);
\end{tikzpicture}
  \end{center}
\caption{Illustration of $P2L_{ij}$ compression and its redirection. The fill-in is shown through a dashed edge.
Here $j'=\mathcal{P}(j)$. It is to be noted that a part of the graph, with only a few connections that get updated as a result of the fill-in redirection, is shown here.}
\label{fig:P2L_compression}
\end{figure}
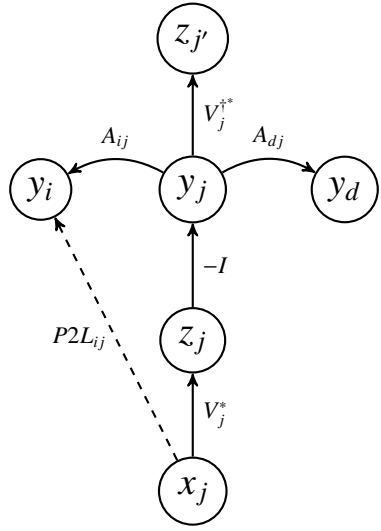

\textbf{P2M, M2L update.}
We find the new orthogonal row basis, that spans the existing row basis of $j$ at level $l$, i.e., $V_{j}^{*}$, and the rows of the fill-in $P2L_{ij}$, by finding the RRQR decomposition of the augmented matrix
\begin{equationwithLineno}
    [V_{j} |P2L_{ij}^*] = \tilde{V}_{j} R_{ij}
\end{equationwithLineno}
By expressing $R_{ij}$ as the augmented matrix $[R_{j} | \tilde{R}_{ij}]$, we have
\begin{subequations}
   \begin{align}
    V_{j} &= \tilde{V}_{j} R_{j}, \label{eq:P2L_4a}\\
    P2L_{ij}^* &= \tilde{V}_{j} \tilde{R}_{ij} ,\label{eq:P2L_4b}
\end{align}  
\end{subequations}
Using Equations~\eqref{eq:P2L_4a} and, ~\eqref{eq:P2L_4b}, $A_{ij}V_{j}^{*}+P2L_{ij}$ can be expressed as 
   \begin{align}
    A_{ij}V_{j}^{*}+P2L_{ij}
    &=A_{ij}R_{j}^{*}\tilde{V}_{j}^* +  \tilde{R}_{ij}^{*} \tilde{V}_{j}^*\\
    &=(A_{ij}R_{j}^{*} + \tilde{R}_{ij}^{*} )\tilde{V}_{j}^*.
\end{align}  
We then make the following assignments, which update the old operators with the new ones.
   \begin{align}
    A_{ij}&:= A_{ij}R_{j}^{*} + \tilde{R}_{ij}^{*}\\
    V_{j} &:= \tilde{V}_{j}
\end{align}  

\textbf{Other M2L updates.}
As a result of the redirection of the fill-in $P2L_{ij}$, the M2Ls $A_{dj}$ where $d\in \{\mathcal{IL}(j)\backslash i\}\bigcup \mathcal{N}(j)\}$ get updated. The updates follow the same lines described in Subsubsection~\ref{ssec:P2P_fill_in}.

\textbf{M2M update.}
As a result of the redirection of the fill-in $P2L_{ij}$, M2M $V_{j}^{\dagger*}$ gets updated. The updates follow the same lines described in Subsubsection~\ref{ssec:P2P_fill_in}.

\subsubsection{Compression and redirection of M2P fill-in}\label{ssec:M2P_fill_in}
Consider a M2P fill-in $M2P_{ij}$ where 
hypercubes $i$ and $j$ at level $l$ are well-separated.
Then $M2P_{ij}$ can be efficiently approximated by a low-rank matrix and this interaction can be redirected through an already existing interaction 
via the path $y_{j}\rightarrow y_{i}\rightarrow z_{i} \rightarrow x_{i}$, as shown in the Figure~\ref{fig:M2P_compression}. 
This redirection results in an update of 
(i) L2P $U_{i}$; (ii) M2L $A_{ij}$; (iii) L2L $U_{i}^{\dagger}$; 
(iv) Other M2Ls $\{A_{ic}: c\in \{\mathcal{IL}(i)\backslash j\}\bigcup \mathcal{N}(i)\}$. We now describe how each of these updates is done.

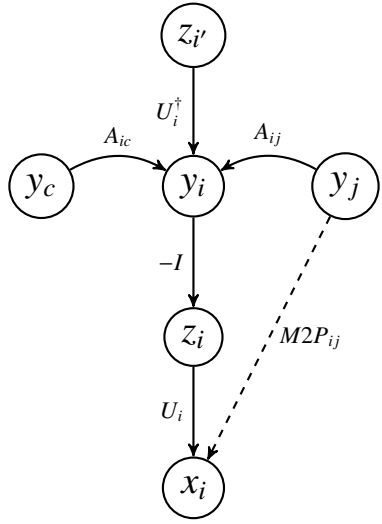
\begin{figure}[H]
  \begin{center}
\begin{tikzpicture}[->,>=stealth',auto,node distance=2cm,
  thick,main node/.style={circle,draw,font=\sffamily\Large\bfseries}]

  \node[main node] (1) {$y_{c}$};
  \node[main node] (2) [right of=1] {$y_{i}$};
  \node[main node] (2a) [below of=2] {$z_{i}$};
  \node[main node] (3) [right of=2] {$y_{j}$};
  \node[main node] (5) [below of=2a] {$x_{i}$};
  \node[main node] (7) [above of=2] {$z_{i'}$};
  
  \path[every node/.style={font=\sffamily\small}]
    (1) edge[draw=black, bend left] node [color=black, pos=0.5, above] {$A_{ic}$} (2)
    (3) edge[draw=black, bend right] node [color=black, pos=0.5, above] {$A_{ij}$} (2)
    (2) edge[draw=black,->] node [color=black, pos=0.5, left] {$-I$} (2a)
    (2a) edge[draw=black] node [color=black, pos=0.5,left] {$U_{i}$} (5)
    (7) edge[draw=black] node [color=black, pos=0.5,left] {$U_{i}^{\dagger}$} (2)
    (3) edge[draw=black, dashed] node [color=black, pos=0.5, right] {$M2P_{ij}$}  (5);
\end{tikzpicture}
  \end{center}
\caption{Illustration of $M2P_{ij}$ compression and its redirection. The fill-in is shown through a dashed edge.
Here $i'=\mathcal{P}(i)$. It is to be noted that a part of the graph, with only a few connections that get updated as a result of the fill-in redirection, is shown here.}
\label{fig:M2P_compression}
\end{figure}

\textbf{L2P, M2L update.}
We find the new orthogonal column basis, that spans the existing column basis of $i$ at level $l$ i.e., $U_{i}$, and the columns of the fill-in $M2P_{ij}$, by finding the RRQR decomposition of the augmented matrix
\begin{equationwithLineno}
    [U_{i} |M2P_{ij}] = \tilde{U}_{i} L_{ij}
\end{equationwithLineno}
By expressing $L_{ij}$ as the augmented matrix, $[L_{i} | \tilde{L}_{ij}]$, we have
\begin{subequations}
   \begin{align}
    U_{i} &= \tilde{U}_{i} L_{i}, \label{eq:L2P_4a}\\
    M2P_{ij} &= \tilde{U}_{i} \tilde{L}_{ij} ,\label{eq:L2P_4b}
\end{align}  
\end{subequations}
Using Equations~\eqref{eq:L2P_4a} and~\eqref{eq:L2P_4b}, $U_{i}A_{ij}+M2P_{ij}$ can be expressed as 
   \begin{align}
    U_{i}A_{ij}+M2P_{ij}
    &=\tilde{U}_{i} L_{i}A_{ij} + \tilde{U}_{i} \tilde{L}_{ij}\\
    &=\tilde{U}_{i} (L_{i}A_{ij} + \tilde{L}_{ij}).
\end{align}  
We then make the following assignments, which update the old operators with the new ones.
   \begin{align}
    A_{ij}&:= L_{i}A_{ij} + \tilde{L}_{ij}\\
    U_{i} &:= \tilde{U}_{i}
\end{align}  

\textbf{Other M2L updates.}
As a result of the redirection of the fill-in $M2P_{ij}$, the M2Ls $A_{ic}$ where $c\in \{\mathcal{IL}(i)\backslash j\}\bigcup \mathcal{N}(i)$ get updated. The updates follow the same lines as described in Subsubsection~\ref{ssec:P2P_fill_in}.

\textbf{L2L update.}
As a result of the redirection of the fill-in $M2P_{ij}$, the L2L $U_{i}^{\dagger}$ gets updated. The updates follow the same lines described in Subsubsection~\ref{ssec:P2P_fill_in}.

 \subsection{A more efficient elimination algorithm}
 For $d>1$, a fill-in corresponding to well-separated hypercubes, say $P2P_{pq}$, could get created or updated multiple times during the elimination process.
 It is because there could be many hypercubes $i^{(l)}$ such that hypercubes $p^{(l)}, q^{(l)}\in\mathcal{N}(i^{(l)})$.
 To avoid the compression and redirection multiple times, we choose not to compress and redirect as and when a fill-in gets created as in Algorithm~\ref{alg:elimination2}, but to update the fill-ins multiple times and compress and redirect only once, just before either $p^{(l)}$ or $q^{(l)}$ gets eliminated as in Algorithm~\ref{alg:elimination}. 

 In Algorithm~\ref{alg:elimination} vectors $vP2P$ and $vP2L$ are used to keep track of the fill-ins corresponding to well-separated hypercubes.
 For the fill-ins $P2P_{pq}^{(l)}$ and $P2P_{qp}^{(l)}$ where $p^{(l)}\in\mathcal{IL}(q^{(l)})$, only one ordered pair $(p^{(l)},q^{(l)})$ is stored in $vP2P$ as they always occur in a pair. Similarly for the fill-ins $P2L_{pq}^{(l)}$ an $M2P_{qp}^{(l)}$ where $p^{(l)}\in\mathcal{IL}(q^{(l)})$, only one ordered pair $(p^{(l)},q^{(l)})$ is stored in $vP2L$. Before a node $i^{(l)}$ gets eliminated, the vectors $vP2P$ and $vP2L$ are searched for an ordered pair with $i^{(l)}$ as one of its elements.
 If it exists then the associated fill-ins are compressed and redirected.
 
 \begin{algorithm}[!htbp]
	\caption{Efficient Elimination algorithm}\label{alg:elimination}
	\begin{algorithmic}[1]
		\Procedure{Efficient\_Elimination}{$n_{\max}$,$\epsilon_{A}$}
		\State {Form $\mathcal{T}^L$, where $L= \min \left\{l : \abs{x^{(l)}_{i}} < n_{\max};\forall \text{ hypercubes $i$ at level $l$}\right\}$}
		\State {Perform NNCA with tolerance $\epsilon_{A}$ to find the $L2L/L2P, M2L, M2M/P2M$ operators of all hypercubes at all levels}
		\State{Declare sets $vP2P$ and $vP2L$ that holds integer ordered pairs}
            \For{\texttt{$l=L:2$}} 
				\For{\texttt{$i=e_{1}^{(l)},e_{2}^{(l)},\mathellipsis,e_{2^{dl}}^{(l)}$}} 
				    \For{$(r^{(l)},s^{(l)})$ in $vP2P$}
				        \If{$(i==r^{(l)}||i==s^{(l)})$}
				            \State Compress $P2P_{rs}^{(l)}$ and $P2P_{sr}^{(l)}$ and update the relevant operators as in Subsection~\ref{ssec:P2P_fill_in}.
                                \State Erase $(r^{(l)},s^{(l)})$ in $vP2P$
				        \EndIf
				    \EndFor
				    \For{$(r^{(l)},s^{(l)})$ in $vP2L$}
				        \If{$(i==s^{(l)})$}
				        \State Compress $P2L_{rs}^{(l)}$ and $M2P_{sr}^{(l)}$ and update the relevant operators as in Subsections~\ref{ssec:P2L_fill_in} and~\ref{ssec:M2P_fill_in} respectively.
                            \State Erase $(r^{(l)},s^{(l)})$ in $vP2L$
				        \EndIf
				    \EndFor
				    \State{Eliminate $x_{i}^{(l)}$ and $z_{i}^{(l)}$ from the extended sparse system using Equations~\eqref{eq:EquationX2} and~\eqref{eq:EquationY}}
				    \State{$E_{i}^{(l)}:=1$}
					\For{$(p^{(l)},q^{(l)})\text{ in }\{(r^{(l)},s^{(l)}): r^{(l)}\in \mathcal{N}(i^{(l)}),s\in\mathcal{N}(i^{(l)})\}$}
						\If{$p^{(l)}$ is eliminated}
    						\If{$q^{(l)}$ is eliminated}
						        \State Results in the update of $M2L_{pq}^{(l)}$ and $M2L_{qp}^{(l)}$
					        \Else
					            \State Results in the update of $P2L_{pq}^{(l)}$ and $M2P_{qp}^{(l)}$
					            \If{$p^{(l)}$ and $q^{(l)}$ are well-separated}
					                \State $vP2L$.push\_back($(p^{(l)},q^{(l)})$)
					            \EndIf
				            \EndIf
						\Else
                            \If{$q^{(l)}$ is eliminated}
						        \State Results in the update of $P2L_{qp}^{(l)}$ and $M2P_{pq}^{(l)}$
                                \If{$p^{(l)}$ and $q^{(l)}$ are well-separated}
                                    \State $vP2L$.push\_back($(q^{(l)},p^{(l)})$)
                                \EndIf
					        \Else
					            \State Results in the update of $P2P_{pq}^{(l)}$ and $P2P_{qp}^{(l)}$
					            \If{$p^{(l)}$ and $q^{(l)}$ are well-separated}
					                \State $vP2P$.push\_back($(p^{(l)},q^{(l)})$)
					            \EndIf
				            \EndIf
			            \EndIf
					\EndFor
				\EndFor
			\EndFor
		\EndProcedure
	\end{algorithmic}
\end{algorithm}

\subsection{Back Substitution or Solve phase}
The third step of AIFMM is the back substitution or solve phase, wherein we solve for the multipoles at level $2$ and then find the unknowns by back substitution. The pseudo-code is described in Algorithm~\ref{alg:substitution}. 

 \begin{algorithm}[!htbp]
	\caption{Back Substitution Algorithm}\label{alg:substitution}
	\begin{algorithmic}[1]
		\Procedure{Back\_Substitution}{}
            \State{Solve for the multipoles at level $2$, $y^{(2)}$, directly}
            \For{\texttt{$l=2:L$}} 
				\For{\texttt{$i=e_{2^{dl}}^{(l)},e_{2^{dl}-1}^{(l)}, \mathellipsis, e_{1}^{(l)},$}} 
                \State Find $x_{i}^{(l)}$ and $z_{i}^{(l)}$ by back substitution using Equations~\eqref{eq:EquationX2} and~\eqref{eq:EquationY}.
                \State $E_{i}^{(l)}:=0$
                \State Find $\{y_{i_{c}}^{(l+1)}\}_{c=1}^{2^{d}}$ from $x_{i}^{(l)}$ using Equation~\eqref{eq:particleAtNonLeafLevel}.
				\EndFor
			\EndFor
		\EndProcedure
	\end{algorithmic}
\end{algorithm}

\begin{remark}
The elimination process, similar to the factorize phase in a direct solver, can be decoupled from the right-hand side. So the elimination phase can be considered as the factorize phase and the back substitution phase can be considered as the solve phase.
\end{remark}

\section{Numerical Results} \label{sec:numericalResults}
We perform a total of five experiments to demonstrate the performance of AIFMM as a direct solver and as a preconditioner. 
  \begin{table}[!htbp]
\centering
  \begin{tabularx}{\textwidth}{|p{13mm}|X|}
    \hline
    $N$ & System size that denotes the number of particles in the domain.\\ \hline
    $\epsilon_{A}$ & Tolerance set for NNCA and RRQR, of AIFMM.\\ \hline
    $r_{m}$ & Maximum rank of the compressed blocks, which includes the interactions and the fill-ins corresponding to well-separated hypercubes.\\ \hline
    $T_{Aa}$ & Time taken to construct the extended sparse matrix using NNCA.\\ \hline
    $T_{Af}$ & Time taken by the elimination phase of AIFMM excluding the time taken to perform the Schur complement operations on the rhs.\\ \hline
    $T_{As}$ & Sum of the time taken by the back substitution phase of AIFMM and the time taken to perform the Schur complement operations on the rhs, i.e., the respective operations that are to be performed on the rhs during the elimination phase.\\ \hline
    $E_{A}$ & Relative forward error of AIFMM measured using $\|.\|_{2}$.\\ \hline
    $\epsilon_{GMRES}$ & The relative residual $\frac{\|A\hat{x}-b\|_{2}}{\|b\|_{2}}$, that is used as the stopping criterion for GMRES, where $\hat{x}$ is the solution computed using GMRES.\\ \hline
    $T_{Ga}$ & For problems involving non-oscillatory Green's functions and the Helmholtz function at low frequency it is the time taken to construct the $\mathcal{H}^{2}$ matrix representation. For problems involving high frequency Helmholtz function it is the time taken to construct the DAFMM matrix~\cite{gujjulaDAFMM}.\\ \hline
    $T_{Gs}$ & Time taken to solve the system using GMRES\\ \hline
    $I_{G}$ & Number of iterations it takes for convergence by GMRES with no preconditioner.\\ \hline
    $E_{G}$ & Relative forward error of GMRES measured using $\|.\|_{2}$.\\ \hline
    $T_{Ha}$ & Time taken to assemble the matrix in HODLR form.\\ \hline
    $T_{Hf}$ & Time taken to factorize using HODLR.\\ \hline
    $T_{Hs}$ & Time taken to solve using HODLR.\\ \hline
    $E_{H}$ & Relative forward error of HODLR measured using $\|.\|_{2}$.\\ \hline
    $I_{pA}$ & Number of iterations it takes for convergence by GMRES with AIFMM as a preconditioner.\\ \hline
    $I_{pH}$ & Number of iterations it takes for convergence by GMRES with HODLR preconditioner.\\ \hline
    $I_{BD}$ & Number of iterations it takes for convergence by GMRES with block-diagonal preconditioner.\\ \hline
    $T_{BD}$ & Time taken to solve by GMRES with block-diagonal preconditioner.\\ \hline
    $T_{pAs}$ & Time taken to solve by GMRES with AIFMM preconditioner.\\ \hline
    $T_{pHs}$ & Time taken to solve by GMRES with HODLR preconditioner.\\ \hline
     relative & relative forward error in the solution measured using $\|.\|_{2}$.\\
     error & \\ \hline
\end{tabularx}
  \caption{List of notations followed in this section}
    \label{table:ResultsNotations}
 \end{table}
 
  In Experiment 1, the validation, convergence and various benchmarks of AIFMM are presented.
In Experiments 2 to 4, AIFMM is compared with HODLR~\cite{ambikasaran2013mathcal,ambikasaran2019hodlrlib}, a direct solver, and with GMRES~\cite{saad1985generalized,saad2003iterative}, an iterative solver. 
HODLR solver hierarchically partitions the matrix and constructs low-rank approximations of the off-diagonal blocks to a user-specified tolerance $\epsilon_{H}$. 

 In Experiment 5, AIFMM is demonstrated as a preconditioner. GMRES with AIFMM as preconditioner is compared with i) GMRES with no preconditioner ii) GMRES with HODLR as preconditioner iii) block-diagonal preconditioner.
HODLR and AIFMM are used as preconditioners by constructing low-accuracy direct solvers, i.e., a high value of $\epsilon_{H}$ and $\epsilon_{A}$ are used respectively.

  GMRES involves the computation of a matrix-vector product in each of its iterations. In Experiments 2 to 4, this computation is accelerated using NNCA-based fast $\mathcal{H}^{2}$ matrix-vector product, described in~\cite{gujjula2022nca}. While in Experiment 5, where we solve the high frequency scattering problem, we use NNCA-based Directional Algebraic Fast Multipole Method (DAFMM), described in~\cite{gujjulaDAFMM}. Let the compression tolerance of these fast summation techniques be denoted by $\epsilon_{G}$.

 All experiments were carried out on an Intel Xeon 2.5GHz processor. 
 In Experiments 1 to 4, we solve for $x$, in $Ax=b$, where
 \begin{itemize}
     \item $b$ is considered to be a random vector and
     \item the particles $\{u_{i}\}_{i=1}^{N}$ and $\{v_{i}\}_{i=1}^{N}$ are considered to be same and are distributed uniformly in the domain $[-1,1]^{2}$.
 \end{itemize}

Before presenting the experiments, we describe some notations that are used in this section in Table~\ref{table:ResultsNotations}.

 
 \subsection{Experiment 1: Validation and convergence of AIFMM}~\label{ssec:exp1}
Here we consider the 2D Helmholtz function with the wavenumber set to $1$. To have a well-conditioned matrix, we consider the entries of the matrix to be
  \begin{equationwithLineno}
 A_{i,j}=
    \begin{cases}
    \sqrt{1000N} &\text{if } i=j\\
    \frac{\iota}{4}H_{0}^{(1)}(\|x_{i}-x_{j}\|_{2}) &\text{else}
    \end{cases}.
 \end{equationwithLineno}
We plot $r_{m}$, assembly time, factorization time, solve time, and relative error versus $N$ in Figure~\ref{fig:convergence} for various values of $\epsilon_{A}$. The following inferences are to be noticed from the figure.
 \begin{enumerate}
    \item 
    The relative error for a given $\epsilon_{A}$ is almost constant as $N$ increases. 
    \item 
    The relative error decreases as $\epsilon_{A}$ decreases, which validates the convergence of AIFMM.
    \item 
    Assembly time, solve time, and factorization time scale linearly with $N$.
 \end{enumerate}
 
   \begin{figure}[!htbp]
      \begin{center}
          \begin{subfigure}[b]{0.4\textwidth}
            \includegraphics[width=\linewidth]{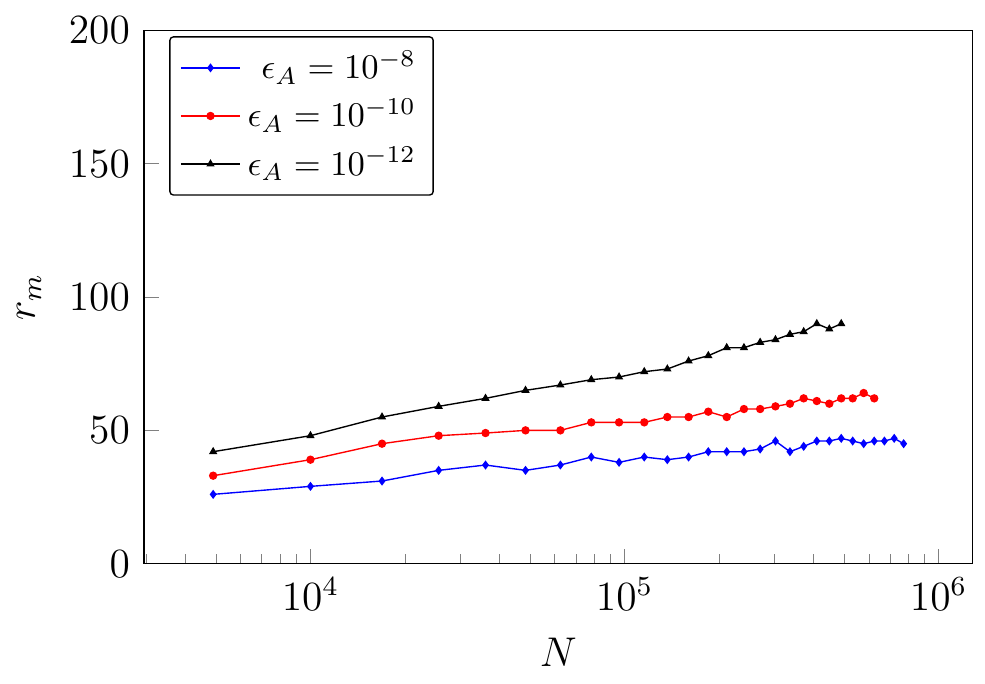}
          \end{subfigure}%
          \begin{subfigure}[b]{0.4\textwidth}
            \includegraphics[width=\linewidth]{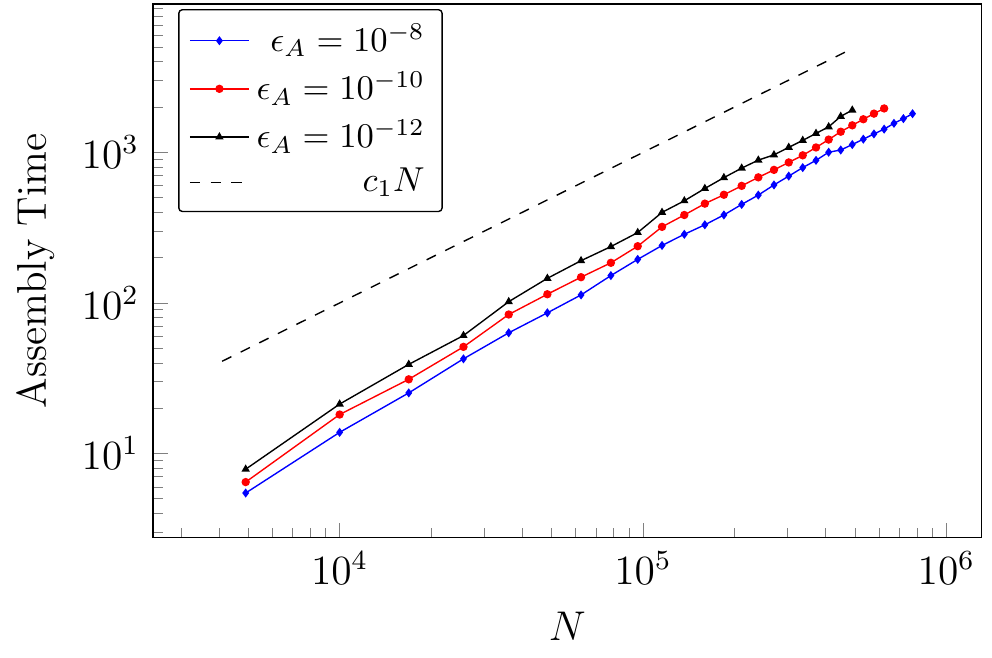}
          \end{subfigure}%
          
          \begin{subfigure}[b]{0.4\textwidth}
            \includegraphics[width=\linewidth]{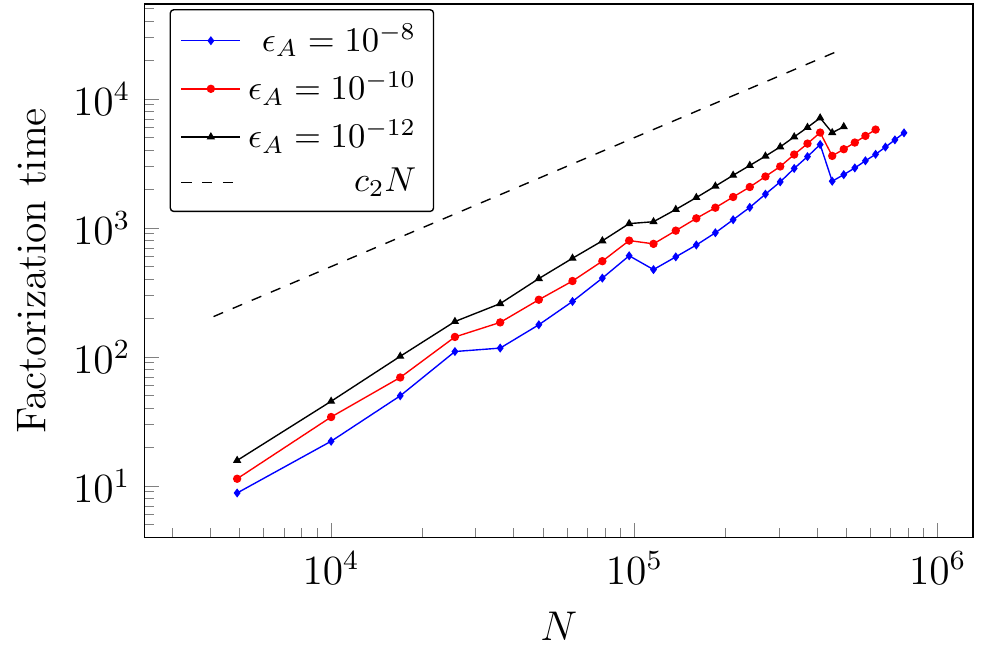}
          \end{subfigure}%
          \begin{subfigure}[b]{0.4\textwidth}
            \includegraphics[width=\linewidth]{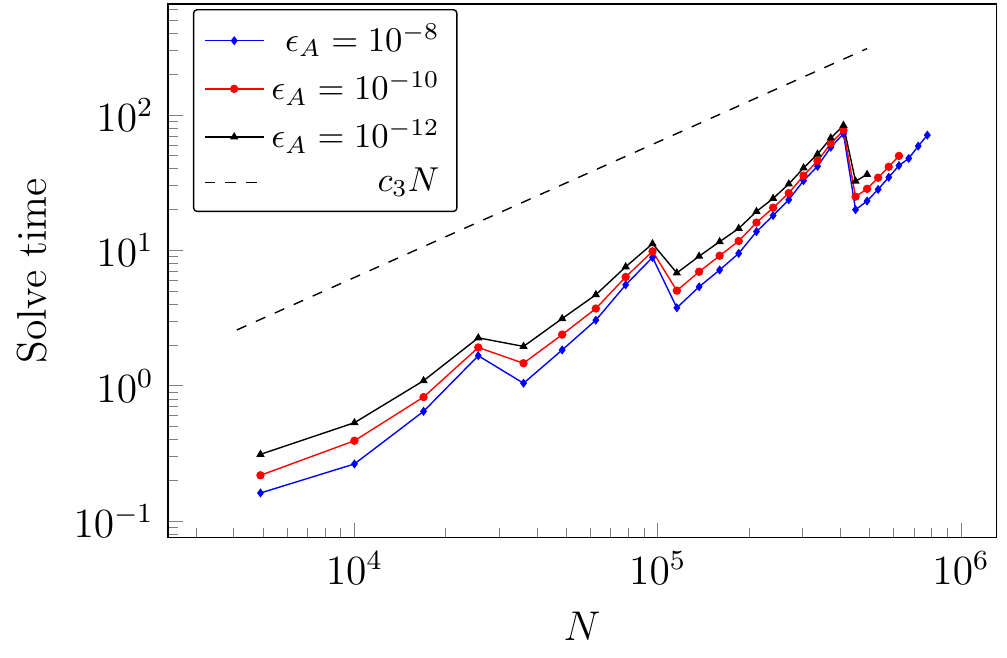}
          \end{subfigure}
          
          \begin{subfigure}[b]{0.4\textwidth}
            \includegraphics[width=\linewidth]{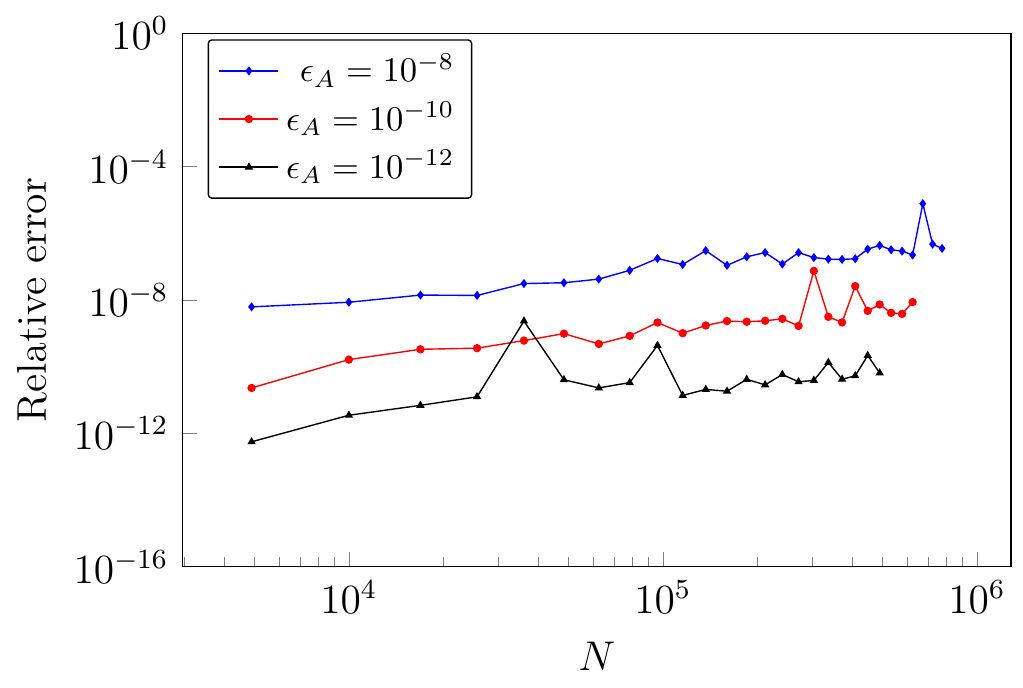}
          \end{subfigure}
        \end{center}
        \caption{Results obtained with Experiment 1; Various benchmarks of AIFMM plotted for different values of $\epsilon_{A}$} \label{fig:convergence}
        \end{figure}

 \subsection{Experiment 2: Comparison of AIFMM with HODLR and GMRES for the 2D Helmholtz kernel}
 Here we consider the same matrix as considered in Subsection~\ref{ssec:exp1}.
 $\epsilon_{A}$, $\epsilon_{G}$, and $\epsilon_{H}$ are set to $10^{-10}$. $\epsilon_{GMRES}$ is also set to $10^{-10}$. 
 We tabulate the various CPU times and the relative errors of the three solvers AIFMM, GMRES, and HODLR in Table~\ref{table:exp3}. Some of these benchmarks are also illustrated in Figure~\ref{fig:exp3}.

    \begin{table}[!htbp]
      \centering
      \resizebox{\columnwidth}{!}{
      \setlength\extrarowheight{0.9pt}
      \begin{tabular}{|c|c|c|c|c|c|c|c|c|c|c|c|c|c|c|c|}
        \hline
         & \multicolumn{4}{|c|}{Assembly} & \multicolumn{3}{|c|}{Factorization} & \multicolumn{5}{|c|}{Solve} & \multicolumn{3}{|c|}{Error}\\ \cline{2-16}
         $N$ & $T_{Ga}$ & $T_{Ha}$ & $T_{Aa}$ & $\frac{T_{Ha}}{T_{Aa}}$ & $T_{Hf}$ & $T_{Af}$ & $\frac{T_{Hf}}{T_{Af}}$ & $T_{Gs}$ & $T_{Hs}$ & $T_{As}$ & $\frac{T_{Gs}}{T_{As}}$ & $\frac{T_{Hs}}{T_{As}}$ & $E_{G}$ & $E_{H}$ & $E_{A}$\\ \hline
         \hline
4900 & 4.3 & 20.1 & 6.4 & 3.2 &      9.1 & 11.4 & 0.8 &     19.6 & 0.1 & 0.2 & 89.7 & 0.3 &      4e-11 & 9e-11 & 1e-11 \\ \hline 
16900 & 21.5 & 191.8 & 30.9 & 6.2 &      103.5 & 69.3 & 1.5 &     92.4 & 0.4 & 0.8 & 112.1 & 0.5 &      4e-10 & 3e-11 & 3e-10 \\ \hline 
36100 & 65.0 & 793.8 & 83.8 & 9.5 &      439.7 & 184.3 & 2.4 &     140.8 & 1.2 & 1.4 & 97.9 & 0.8 &      5e-10 & 3e-11 & 5e-10 \\ \hline 
62500 & 100.7 & 2231.1 & 143.3 & 15.6 &      1283.3 & 385.8 & 3.3 &     440.7 & 2.9 & 3.7 & 119.5 & 0.8 &      5e-10 & 1e-09 & 5e-10 \\ \hline 
115600 & 254.2 & 7211.8 & 318.1 & 22.7 &      4199.6 & 750.9 & 5.6 &     504.6 & 7.1 & 5.0 & 100.1 & 1.4 &      1e-09 & 3e-10 & 1e-09 \\ \hline 
160000 & 347.2 & 13347.6 & 450.1 & 29.7 &      7658.6 & 1178.5 & 6.5 &     1025.4 & 10.9 & 9.1 & 113.0 & 1.2 &      2e-09 & 6e-10 & 2e-09 \\ \hline 
240100 & 459.7 & - & 657.4 & - &      - & 2060.6 & - &     2288.1 & - & 20.5 & 111.8 & - &      3e-09 & - & 3e-09 \\ \hline 
336400 & 617.2 & - & 976.4 & - &      - & 3708.5 & - &     4552.1 & - & 45.8 & 99.4 & - &      3e-09 & - & 3e-09 \\ \hline 
448900 & 1099.2 & - & 1376.8 & - &      - & 3619.1 & - &     2636.3 & - & 25.0 & 105.4 & - &      5e-09 & - & 5e-09 \\ \hline 
577600 & 1393.8 & - & 1809.0 & - &      - & 5166.7 & - &     4373.4 & - & 41.5 & 105.3 & - &      4e-09 & - & 4e-09 \\ \hline 
672400 & 1586.7 & - & 2120.3 & - &      - & 6429.6 & - &     6201.4 & - & 56.6 & 109.7 & - &      7e-09 & - & 8e-09 \\ \hline
\end{tabular}
}
    \caption{Results obtained with experiment 2; CPU times and relative errors of the three solvers}
    \label{table:exp3}
\end{table}

    \begin{figure}[!htbp]
  \begin{center}
    \begin{subfigure}[b]{0.4\textwidth}
      \centering
            \includegraphics[width=\linewidth]{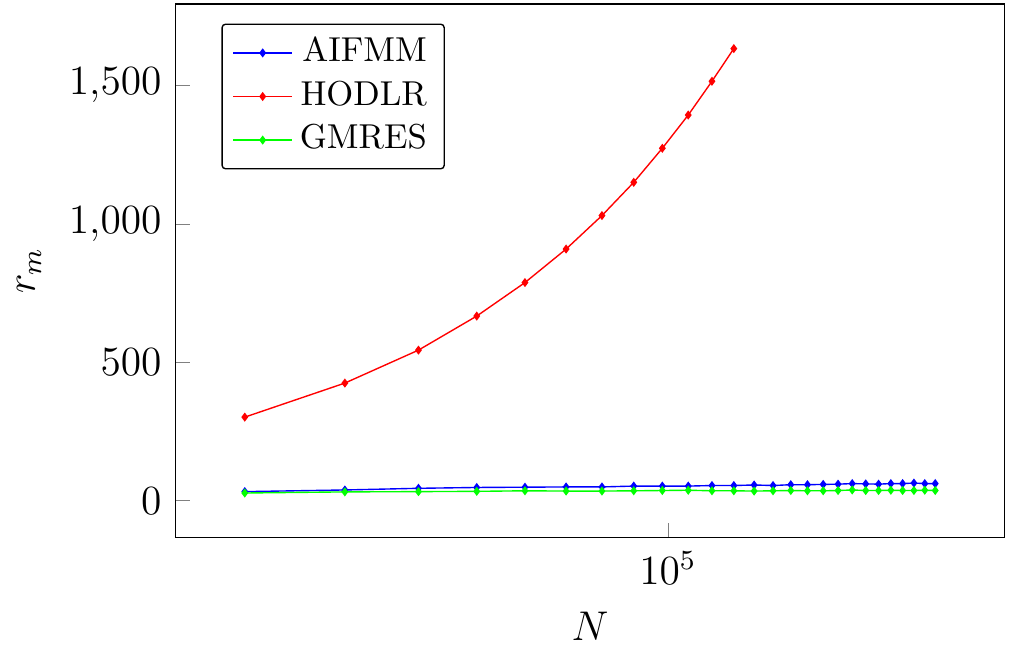}
      \end{subfigure}%
      \begin{subfigure}[b]{0.4\textwidth}
        \centering
            \includegraphics[width=\linewidth]{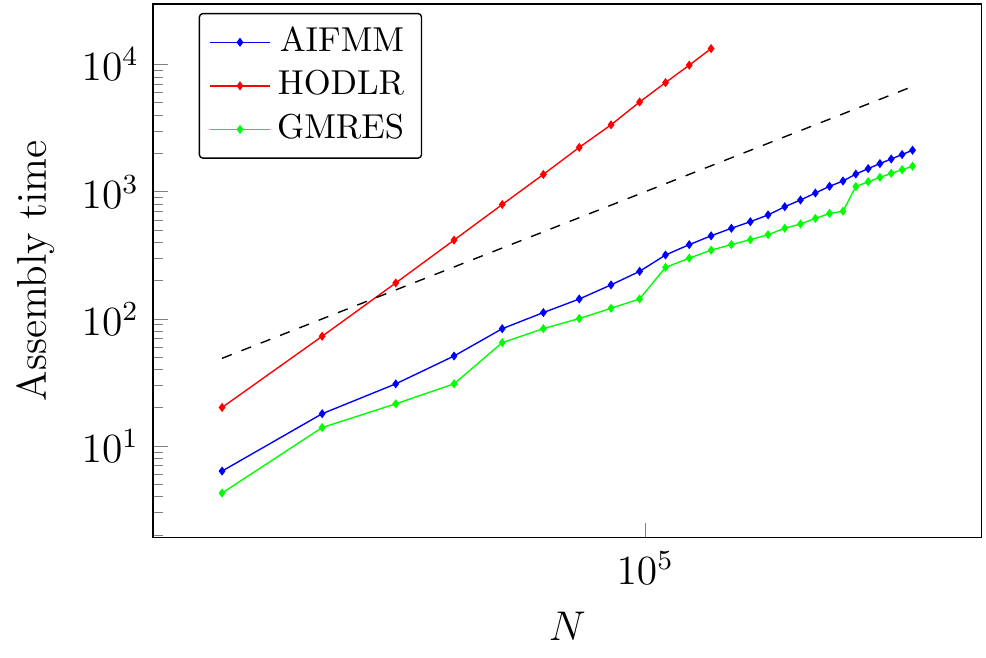}
        \end{subfigure}
        
      \begin{subfigure}[b]{0.4\textwidth}
        \centering
        \includegraphics[width=\linewidth]{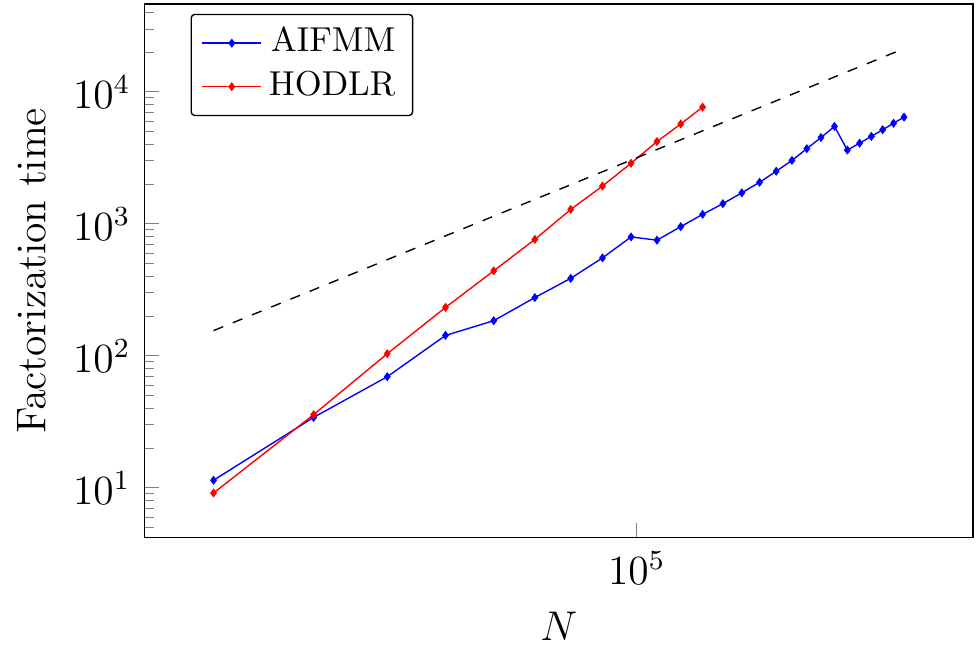}
        \end{subfigure}%
        \begin{subfigure}[b]{0.4\textwidth}
          \centering
              \includegraphics[width=\linewidth]{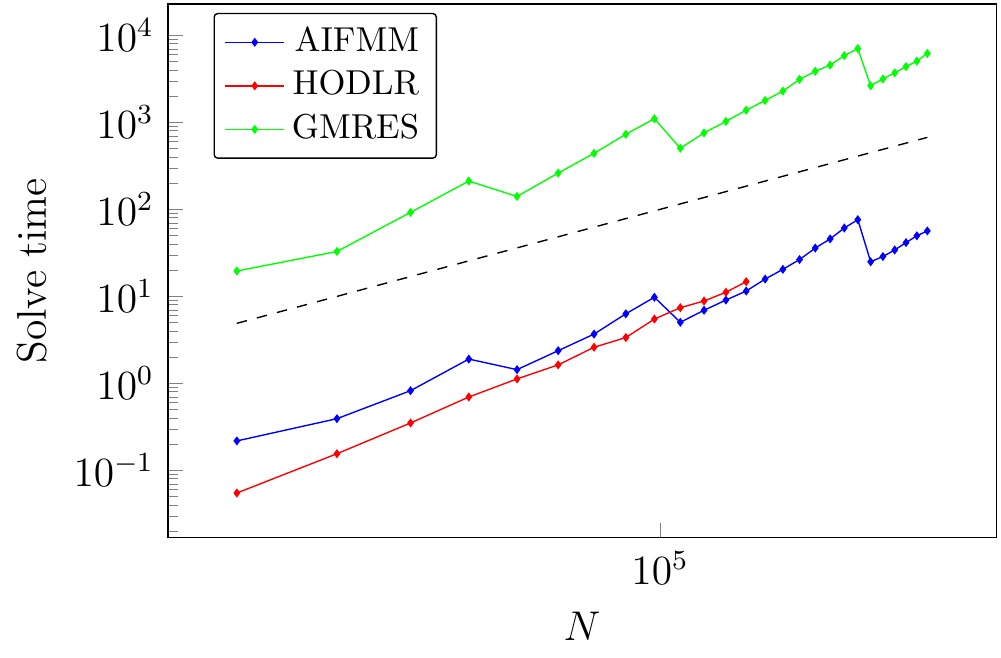}
          \end{subfigure}%
      \caption{Results obtained with Experiment 2; Plots of $r_{m}$, assembly time, factorisation time, and solve time versus $N$ of AIFMM in comparison to those of HODLR and GMRES}
      \label{fig:exp3}
  \end{center}
\end{figure}

 \subsection{Experiment 3: Comparison of AIFMM with HODLR and GMRES for the 2D Laplace kernel}\label{ssec:Exp2}
 Here we consider the 2D Laplace kernel. Again to ensure a well-conditioned matrix, we consider the entries of the matrix to be
  \begin{equationwithLineno}
 A_{i,j}=
    \begin{cases}
    \sqrt{1000N} &\text{if } i==j\\
    \frac{1}{\|x_{i}-x_{j}\|_{2}} &\text{else}
    \end{cases}.
 \end{equationwithLineno}
$\epsilon_{A}$, $\epsilon_{G}$ and $\epsilon_{H}$ are set to $10^{-10}$, $10^{-8}$, and $10^{-10}$ respectively. 
We used different compression tolerances to ensure that the relative errors of the three solvers AIFMM, GMRES, and HODLR are of the same order so that the CPU times of the solvers can be compared and an inference can be drawn on which solver performs better.
$\epsilon_{GMRES}$ is set to $10^{-10}$. 
We tabulate the various CPU times and the relative errors of the three solvers AIFMM, GMRES and HODLR in Table~\ref{table:exp2}. Some of these benchmarks are also illustrated in Figure~\ref{exp1VsN}.
 
    \begin{table}[!htbp]
      \centering
      \resizebox{\columnwidth}{!}{
      \setlength\extrarowheight{0.9pt}
      \begin{tabular}{|c|c|c|c|c|c|c|c|c|c|c|c|c|c|c|c|}
        \hline
         & \multicolumn{4}{|c|}{Assembly} & \multicolumn{3}{|c|}{Factorization} & \multicolumn{5}{|c|}{Solve} & \multicolumn{3}{|c|}{Error}\\ \cline{2-16}
         $N$ & $T_{Ga}$ & $T_{Ha}$ & $T_{Aa}$ & $\frac{T_{Ha}}{T_{Aa}}$ & $T_{Hf}$ & $T_{Af}$ & $\frac{T_{Hf}}{T_{Af}}$ & $T_{Gs}$ & $T_{Hs}$ & $T_{As}$ & $\frac{T_{Gs}}{T_{As}}$ & $\frac{T_{Hs}}{T_{As}}$ & $E_{G}$ & $E_{H}$ & $E_{A}$\\ \hline
         \hline
 4900 & 0.7 & 6.1 & 0.7 & 9.0 &      6.8 & 4.5 & 1.5 &     0.8 & 0.1 & 0.1 & 5.4 & 0.4 &      2e-08 & 5e-11 & 2e-08 \\ \hline 
16900 & 3.7 & 86.7 & 4.0 & 21.7 &      81.2 & 37.4 & 2.2 &     4.6 & 0.4 & 0.5 & 8.4 & 0.6 &      5e-08 & 5e-09 & 5e-08 \\ \hline 
36100 & 11.1 & 393.9 & 11.9 & 33.0 &      333.6 & 125.5 & 2.7 &     13.7 & 1.3 & 1.3 & 10.8 & 1.1 &      2e-07 & 5e-09 & 2e-07 \\ \hline 
62500 & 21.7 & 1147.4 & 22.7 & 50.6 &      966.1 & 275.7 & 3.5 &     38.3 & 2.7 & 2.4 & 15.7 & 1.1 &      2e-07 & 2e-06 & 2e-07 \\ \hline 
96100 & 34.2 & 2559.5 & 34.5 & 74.2 &      2030.0 & 526.0 & 3.9 &     204.1 & 5.1 & 4.6 & 44.1 & 1.1 &      8e-07 & 2e-07 & 7e-07 \\ \hline 
136900 & 69.4 & 5809.5 & 69.0 & 84.2 &      4187.1 & 844.9 & 5.0 &     118.9 & 8.4 & 6.1 & 19.4 & 1.4 &      3e-07 & 3e-07 & 3e-07 \\ \hline 
184900 & 91.9 & 10780.8 & 94.9 & 113.7 &      7265.3 & 1283.2 & 5.7 &     224.6 & 14.5 & 8.9 & 25.3 & 1.6 &      4e-07 & 4e-07 & 4e-07 \\ \hline 
240100 & 128.5 & - & 125.8 & - &      - & 1819.3 & - &     351.6 & - & 12.7 & 27.7 & - &      4e-07 & - & 4e-07 \\ \hline 
336400 & 179.2 & - & 183.8 & - &      - & 2914.8 & - &     648.5 & - & 21.4 & 30.3 & - &      6e-07 & - & 6e-07 \\ \hline 
490000 & 317.4 & - & 301.8 & - &      - & 4612.9 & - &     796.2 & - & 27.5 & 28.9 & - &      7e-07 & - & 7e-07 \\ \hline 
\end{tabular}
}
        \caption{Results obtained with experiment 3; CPU times and relative errors of the three solvers}
        \label{table:exp2}
     \end{table}

    \begin{figure}[!htbp]
  \begin{center}
    \begin{subfigure}[b]{0.4\textwidth}
      \centering
       \includegraphics[width=\linewidth]{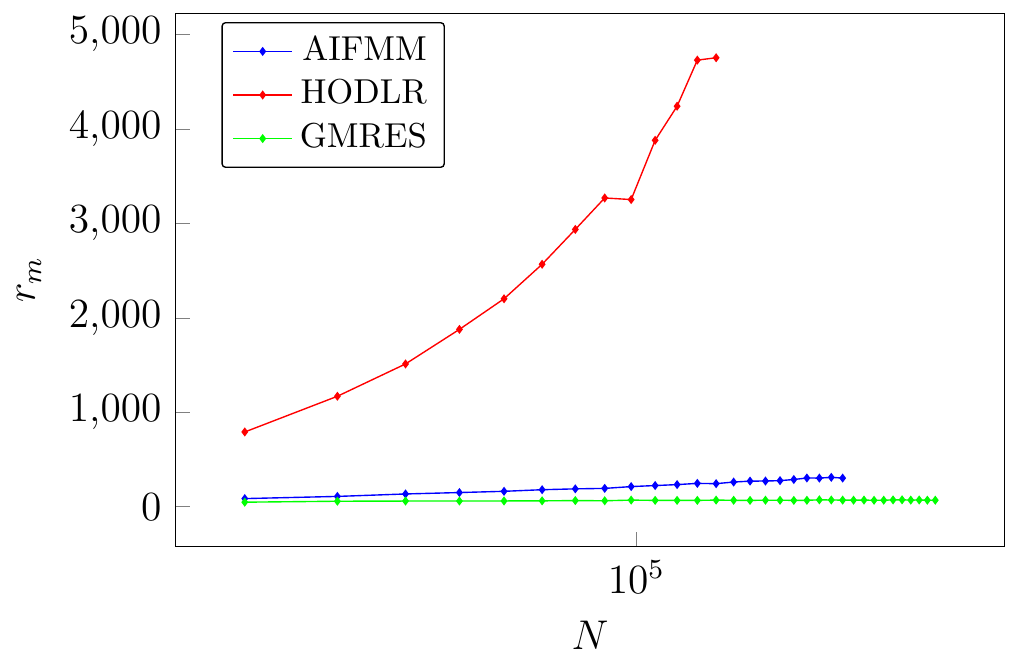}
      \end{subfigure}%
      \begin{subfigure}[b]{0.4\textwidth}
        \centering
        \includegraphics[width=\linewidth]{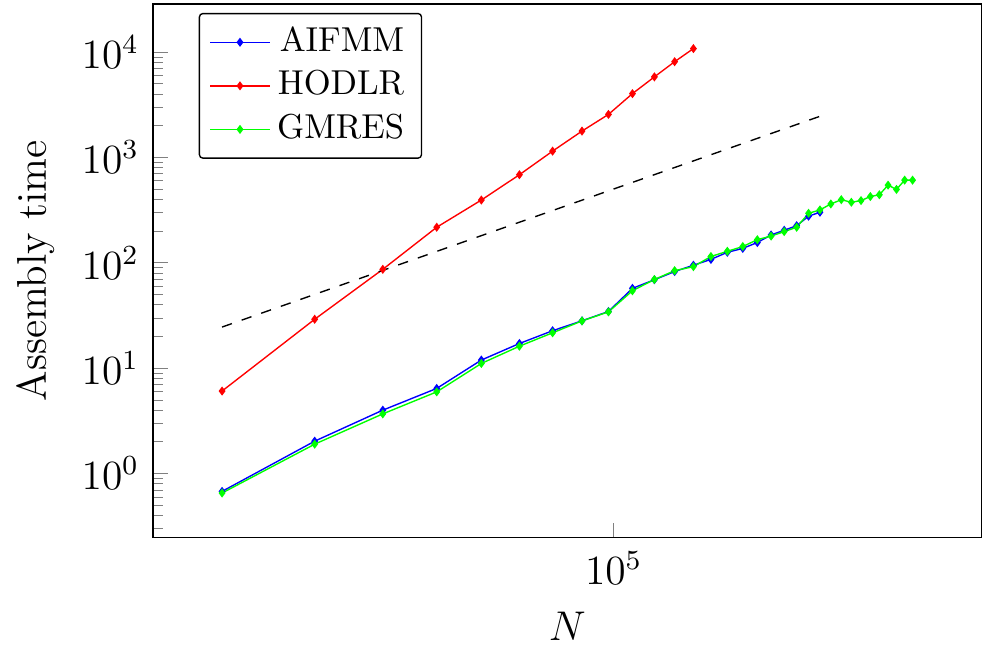}
        \end{subfigure}
        
      \begin{subfigure}[b]{0.4\textwidth}
        \centering
        \includegraphics[width=\linewidth]{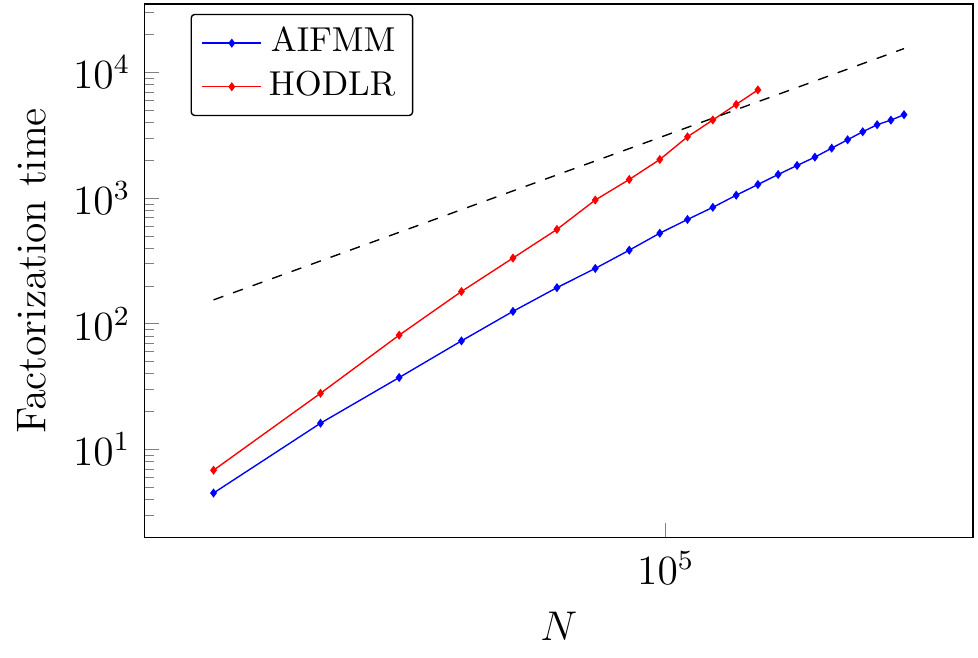}
        \end{subfigure}%
        \begin{subfigure}[b]{0.4\textwidth}
          \centering
          \includegraphics[width=\linewidth]{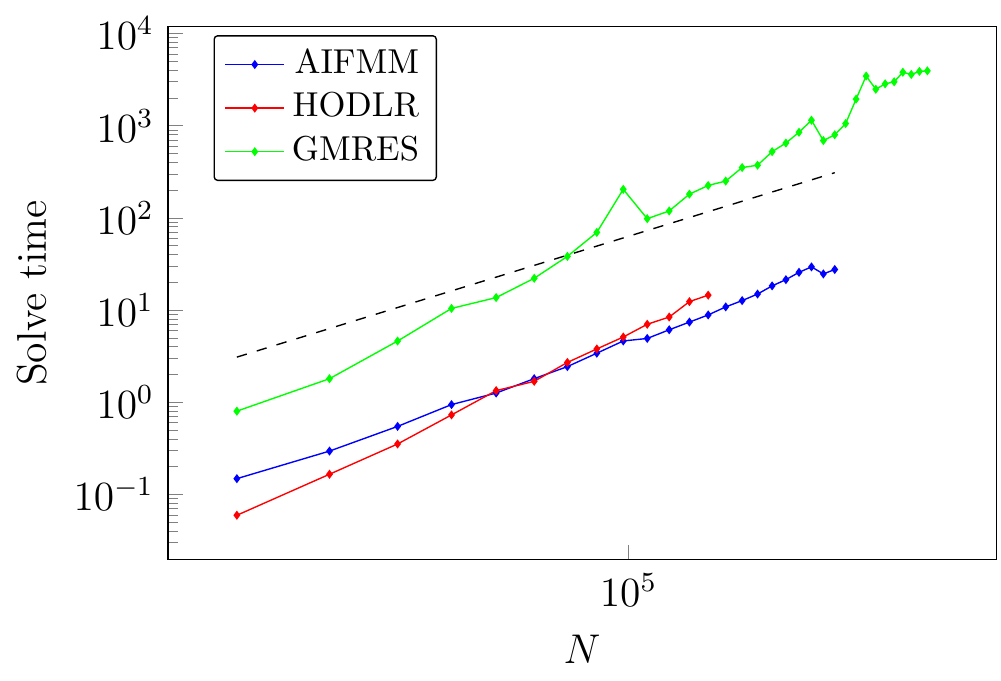}
          \end{subfigure}%
      \caption{Results obtained with Experiment 3; Plots of $r_{m}$, assembly time, factorisation time, and solve time versus $N$ of AIFMM in comparison to those of HODLR and GMRES}
      \label{exp1VsN}
  \end{center}
\end{figure}

 \subsection{Experiment 4: Comparison of AIFMM with HODLR and GMRES in solving an integral equation}
Consider the Fredholm integral equation of the second kind as defined in Equation~\eqref{eq:Fredholmeq},
\begin{equationwithLineno}\label{eq:Fredholmeq}
    \sigma_{x}+\int_{\Omega}K(x,y)\sigma_{y}dy = f(x)
\end{equationwithLineno}
where $K(x,y)=\log(\|x-y\|_{2})$ and $\Omega=[-1,1]^{2}$.
We discretize Equation~\eqref{eq:Fredholmeq} using the Nystrom discretization on a uniform grid, which yields a linear system of the form $Ax=b$. Here we consider $b$ to be a random vector. $\epsilon_{A}$, $\epsilon_{G}$, and $\epsilon_{H}$ are set to $10^{-10}$, $10^{-8}$, and $10^{-8}$ respectively.
$\epsilon_{GMRES}$ is set to $10^{-10}$. 
 We tabulate the various CPU times and the relative errors of the three solvers AIFMM, GMRES, and HODLR in Table~\ref{table:exp4}. Some of these benchmarks are also illustrated in Figure~\ref{fig:exp4}.
    \begin{table}[!htbp]
      \centering
      \resizebox{\columnwidth}{!}{
      \setlength\extrarowheight{0.9pt}
      \begin{tabular}{|c|c|c|c|c|c|c|c|c|c|c|c|c|c|c|c|}
        \hline
         & \multicolumn{4}{|c|}{Assembly} & \multicolumn{3}{|c|}{Factorization} & \multicolumn{5}{|c|}{Solve} & \multicolumn{3}{|c|}{Error}\\ \cline{2-16}
         $N$ & $T_{Ga}$ & $T_{Ha}$ & $T_{Aa}$ & $\frac{T_{Ha}}{T_{Aa}}$ & $T_{Hf}$ & $T_{Af}$ & $\frac{T_{Hf}}{T_{Af}}$ & $T_{Gs}$ & $T_{Hs}$ & $T_{As}$ & $\frac{T_{Gs}}{T_{As}}$ & $\frac{T_{Hs}}{T_{As}}$ & $E_{G}$ & $E_{H}$ & $E_{A}$\\ \hline
         \hline
4900 & 0.6 & 0.7 & 0.8 & 0.9 &      0.5 & 1.8 & 0.3 &     1.7 & 0.0 & 0.0 & 38.5 & 0.2 &      7e-11 & 4e-10 & 2e-10 \\ \hline 
16900 & 2.9 & 5.9 & 3.9 & 1.5 &      4.7 & 10.4 & 0.4 &     7.0 & 0.1 & 0.2 & 46.3 & 0.4 &      6e-10 & 2e-10 & 5e-10 \\ \hline 
36100 & 9.2 & 23.5 & 11.4 & 2.1 &      18.0 & 29.5 & 0.6 &     10.1 & 0.2 & 0.3 & 35.6 & 0.6 &      4e-09 & 6e-10 & 4e-09 \\ \hline 
62500 & 14.2 & 65.3 & 19.1 & 3.4 &      49.6 & 55.2 & 0.9 &     27.6 & 0.5 & 0.7 & 41.3 & 0.7 &      2e-09 & 3e-09 & 2e-09 \\ \hline 
96100 & 20.0 & 148.3 & 29.8 & 5.0 &      109.1 & 105.2 & 1.0 &     61.7 & 0.9 & 1.6 & 39.0 & 0.6 &      6e-09 & 2e-09 & 6e-09 \\ \hline 
160000 & 50.4 & 378.6 & 64.1 & 5.9 &      277.0 & 168.7 & 1.6 &     57.8 & 2.0 & 1.7 & 33.2 & 1.1 &      1e-08 & 1e-09 & 1e-08 \\ \hline 
240100 & 68.2 & 806.8 & 89.2 & 9.0 &      571.2 & 275.2 & 2.1 &     123.8 & 3.5 & 3.4 & 36.4 & 1.0 &      4e-09 & 3e-09 & 4e-09 \\ \hline 
409600 & 107.6 & - & 156.4 & - &      - & 686.3 & - &     343.9 & - & 11.8 & 29.2 & - &      6e-09 & - & 6e-09 \\ \hline 
672400 & 253.9 & - & 324.1 & - &      - & 905.1 & - &     307.6 & - & 10.4 & 29.5 & - &      6e-09 & - & 6e-09 \\ \hline 
883600 & 309.8 & - & 401.4 & - &      - & 1322.3 & - &     511.9 & - & 17.6 & 29.0 & - &      5e-09 & - & 5e-09 \\ \hline 
1000000 & 333.3 & - & 477.5 & - &      - & 1614.3 & - &     649.3 & - & 22.6 & 28.8 & - &      9e-09 & - & 9e-09 \\ \hline
\end{tabular}
}
        \caption{Results obtained with experiment 4; CPU times and relative errors of the three solvers}
        \label{table:exp4}
     \end{table}

    \begin{figure}[!htbp]
  \begin{center}
    \begin{subfigure}[b]{0.4\textwidth}
      \centering
       \includegraphics[width=\linewidth]{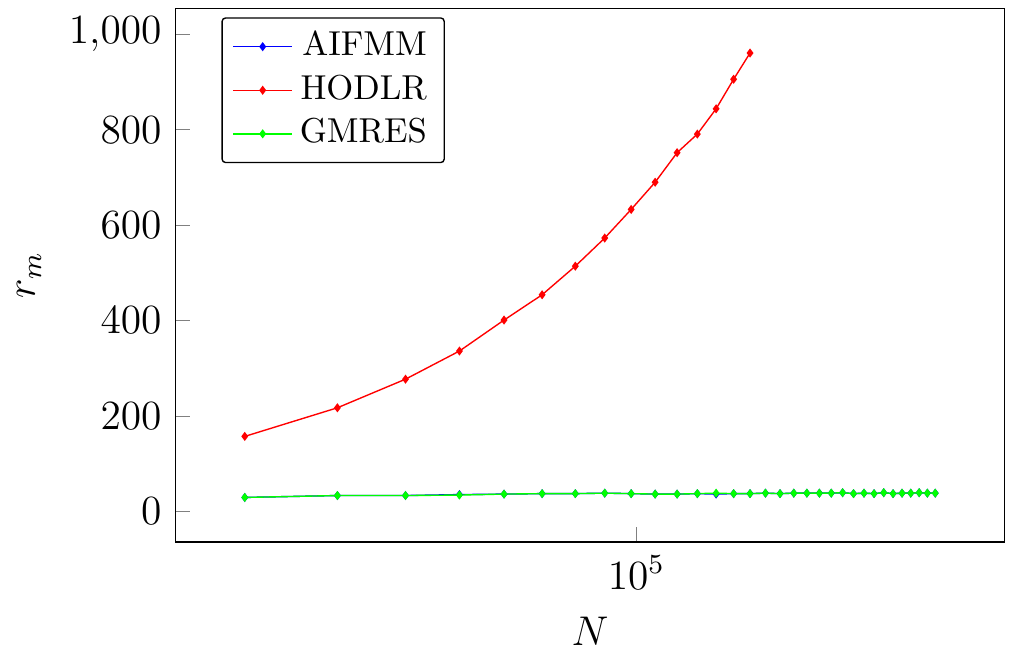}
      \end{subfigure}%
      \begin{subfigure}[b]{0.4\textwidth}
        \centering
        \includegraphics[width=\linewidth]{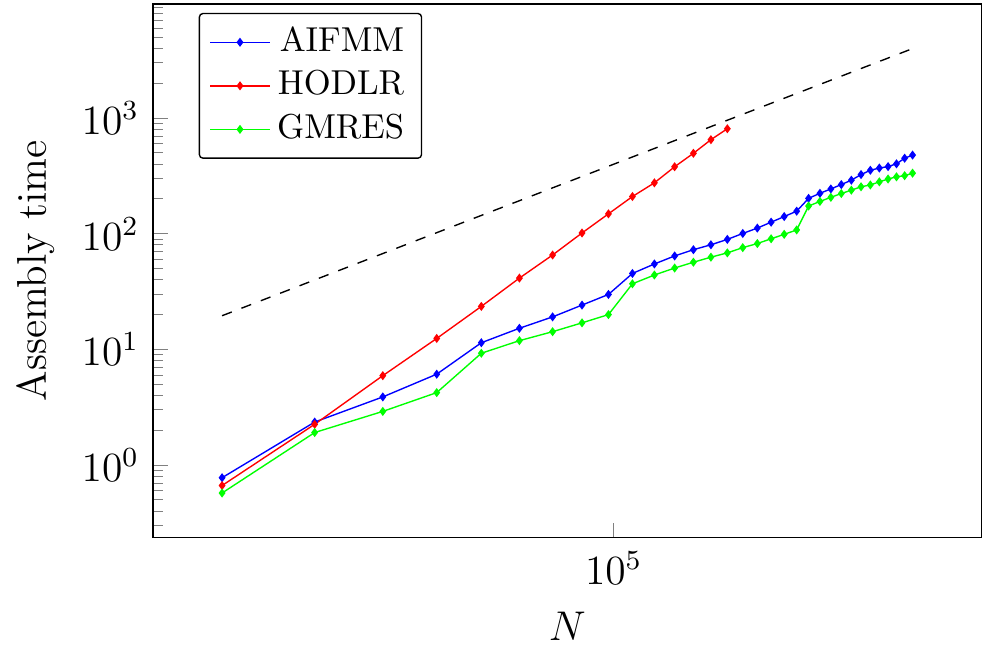}
        \end{subfigure}
        
      \begin{subfigure}[b]{0.4\textwidth}
        \centering
        \includegraphics[width=\linewidth]{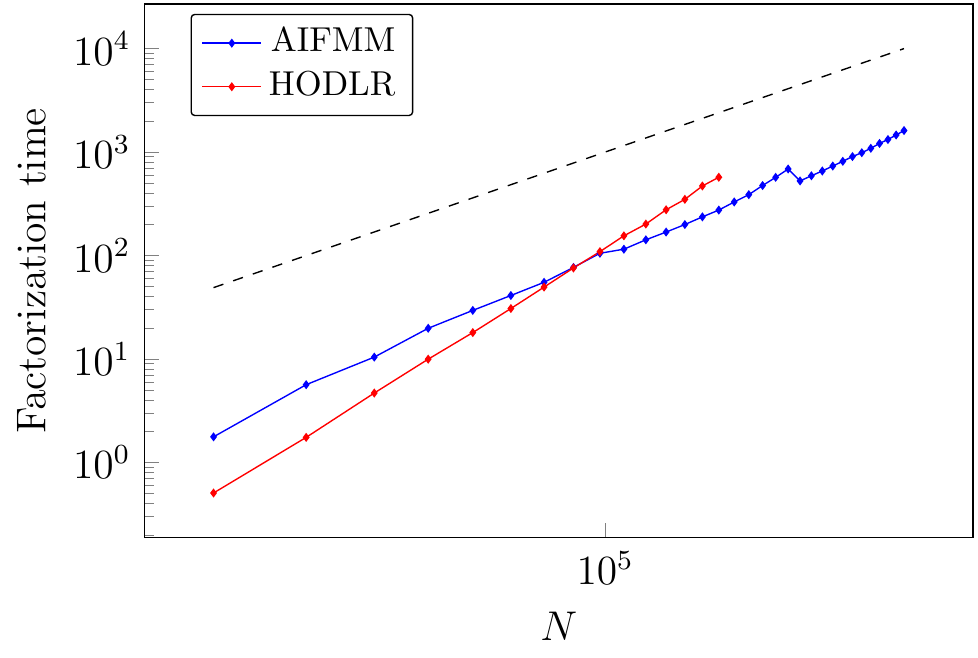}
        \end{subfigure}%
        \begin{subfigure}[b]{0.4\textwidth}
          \centering
          \includegraphics[width=\linewidth]{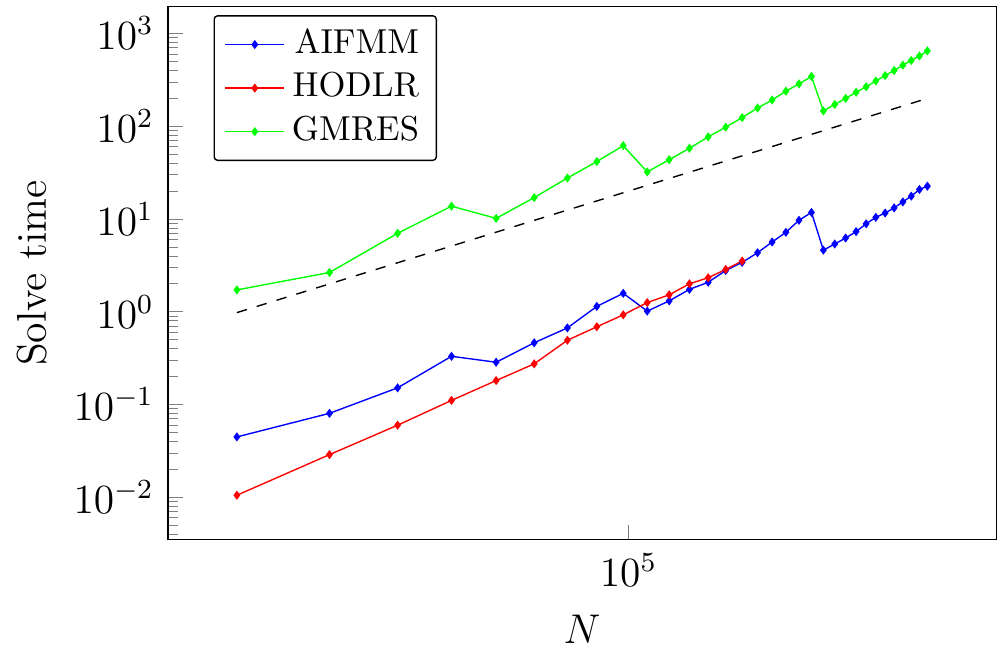}
          \end{subfigure}%
      \caption{Results obtained with Experiment 4; Plots of $r_{m}$, assembly time, factorisation time, and solve time versus $N$ of AIFMM in comparison to those of HODLR and GMRES}
      \label{fig:exp4}
  \end{center}
\end{figure}

\subsection{Wave scattering in 2D}
We now demonstrate AIFMM on the matrix arising in solving acoustic or electromagnetic scattering from penetrable media using the Lippmann-Schwinger equation. It arises in many applications such as medical imaging, sonar, radar, geophysics, remote sensing, etc.
\subsubsection{Formulation of the Lippmann-Schwinger equation in 2D}
We now brief the formulation of the Lippmann-Schwinger equation in 2D. For a detailed description of the formulation of the Lippmann-Schwinger equation, we refer the readers to~\cite{gujjulaDAFMM, ambikasaran2016fast}.

Let $q(x)$, having compact support in the domain $\Omega$, be the contrast function (or the susceptibility) of the penetrable medium. Let $u^{inc}(x)$ be the incident field and $u^{scat}(x)$ be the scattered field. Assume $\kappa$ to be the wavenumber of the incident field. The total field $u(x)$ satisfies the time-harmonic Helmholtz equation
\begin{equationwithLineno} \label{eq:1}
  \nabla^{2}u(x) +\kappa^{2}(1+q(x))u(x) = 0, \hspace{3mm} x\in \mathbb{R}^{2}.
\end{equationwithLineno}
Assuming the incident field, $u^{inc}(x)$, satisfies the homogeneous Helmholtz equation, the scattered field $u^{scat}(x)$ satisfies
\begin{equationwithLineno} \label{eq:2}
  \nabla^{2}u^{scat}(x) +\kappa^{2}(1+q(x))u^{scat}(x) = f(x)
\end{equationwithLineno}
where $f(x) = -\kappa^{2}q(x)u^{inc}(x)$.
Expressing $u^{scat}(x)$ as the volume potential in terms of an unknown density function $\psi(x)$ and the Green's function, as in Equation~\eqref{eq:4}
\begin{equationwithLineno} \label{eq:4}
 u^{scat}(x) = V[\psi](x) = \int_{\Omega}G_{\kappa}(x,y)\psi(y)dy,
\end{equationwithLineno}
results in the Lippmann-Schwinger equation
\begin{equationwithLineno} \label{eq:3}
  \psi(x) + \kappa^{2}q(x)V[\psi](x) = f(x).
\end{equationwithLineno}
To discretize the Lippmann-Schwinger equation, a balanced quad-tree is constructed. 
The tensor product Chebyshev nodes of size $p^{2}$ of each leaf box are considered to be the grid points. 
A local polynomial approximation of $\psi(x)$ in each leaf box $B$ is built as following
\begin{equationwithLineno} \label{eq:7}
  \psi(x) \approx \psi^{B}(x) = \sum_{l=1}^{N_{p}} c_{l}^{B} b_{l}\left(\frac{\zeta_{1}-\alpha_{1}^{B}}{\beta^{B}}, \frac{\zeta_{2}-\alpha_{2}^{B}}{\beta^{B}}\right)\hspace{3mm}
  \forall \hspace{1mm} x=(\zeta_{1},\zeta_{2})\in B.
\end{equationwithLineno}
 where $N_{p}$ is chosen to be $p(p+1)/2$, $(\alpha_{1}^{B},\alpha_{2}^{B})$ are the coordinates of the center of box $B$, $\beta^{B}$ is the half side length of $B$, and $\{b_{l}\}_{l=1}^{N_{p}}$ are the polynomial basis functions.
 Let $\{x_{i}^B\}_{i=1}^{p^{2}}$ be the gridpoints of box $B$, at which the unknown $\psi(x)$ is evaluated.
 Vector $\vec{\psi}^{B}=[\psi(x_{1}^B), \psi(x_{2}^B), \ldots,\psi(x_{p^{2}}^B)]^{T}$ is expressed in terms of vector $\vec{c}^{B} = [{c}^{B}_{1}, {c}^{B}_{2}, \ldots,{c}^{B}_{N_{p}}]^{T}$ as
\begin{equationwithLineno} \label{eq:9}
  \vec{\psi}^{B} = Q\vec{c}^{B}
\end{equationwithLineno}
where $Q \in \mathbb{R}^{p^2 \times N_p}$ is the interpolation matrix, whose entries are given by
\begin{equationwithLineno} \label{eq:9b}
  Q_{il} = b_{l}\left(\frac{\zeta_{i,1}^B-\alpha_{1}^{B}}{\beta^{B}}, \frac{\zeta_{i,2}^B-\alpha_{2}^{B}}{\beta^{B}}\right),  \;\; x_{i}^B = (\zeta_{i,1}^B,\zeta_{i,2}^B).
\end{equationwithLineno}
By taking the pseudo-inverse of $Q$, we obtain $\vec{c}^{B}$ in terms of $\vec{\psi}^{B}$
\begin{equationwithLineno} \label{eq:10}
  \vec{c}^{B} = Q^{\dagger}\vec{\psi}^{B}.
\end{equationwithLineno}
Using Equations~\eqref{eq:10} and~\eqref{eq:7}, we have
\begin{equationwithLineno} \label{eq:11}
  \psi^{B}(x) = \sum_{l=1}^{N_{p}}  Q^{\dagger}(l,:)\vec{\psi}^{B}  b_{l}\left(\frac{\zeta_{1}-\alpha_{1}^{B}}{\beta^{B}}, \frac{\zeta_{2}-\alpha_{2}^{B}}{\beta^{B}}\right).
\end{equationwithLineno}
Using Equation~\eqref{eq:11} we build an approximate of $V[\psi](x)$ as 
\begin{equationwithLineno} \label{eq:11b}
  V[\psi](x) \approx \sum_{\mathcal{L}}\int_{B}G_{\kappa}(x,y)\sum_{l=1}^{N_{p}}  Q^{\dagger}(l,:)\vec{\psi}^{B}b_{l}\left(\frac{\zeta_{1}-\alpha_{1}^{B}}{\beta^{B}}, \frac{\zeta_{2}-\alpha_{2}^{B}}{\beta^{B}}\right) dy.
\end{equationwithLineno}
where $\mathcal{L}$ is the set of all leaf boxes.
Using Equations~\eqref{eq:11b} and~\eqref{eq:3}, the discretized version of the Lippmann-Schwinger equation is obtained,
\begin{equationwithLineno} \label{eq:11c}
  \psi(x) +\kappa^{2}q(x) \sum_{\mathcal{L}}\int_{B}G_{\kappa}(x,y)\sum_{l=1}^{N_{p}}  Q^{\dagger}(l,:)\vec{\psi}^{B}b_{l}\left(\frac{\zeta_{1}-\alpha_{1}^{B}}{\beta^{B}}, \frac{\zeta_{2}-\alpha_{2}^{B}}{\beta^{B}}\right) dy\approx f(x).
\end{equationwithLineno}
 By enforcing Equation~\eqref{eq:3} at 
the grid points of all the leaf nodes of the tree, and using Equation~\eqref{eq:11b}, we obtain the linear system
\begin{equationwithLineno} \label{eq:5}
  A\vec{\psi} = \vec{f}
\end{equationwithLineno}
where $\vec{\psi}$ is a vector that contains function values of $\psi(x)$ evaluated at the grid points of leaf boxes of the quad-tree. 
The $(i,j)^{th}$ entry of $A$, that represents the contribution of the $j^{th}$ grid point at the $i^{th}$ grid point is given by
\begin{equationwithLineno} \label{eq:matA}
  A_{ij} = \delta_{ij} + \kappa^{2}q(x_{i})\int_{\Omega \in B}G_{\kappa}(x_i,y) \sum_{l=1}^{N_{p}}  Q^{\dagger}_{l,j'} b_{l}\left(\frac{y_{1}-\alpha_{1}^{B}}{\beta^{B}}, \frac{y_{2}-\alpha_{2}^{B}}{\beta^{B}}\right)  dy_{1} dy_{2}
\end{equationwithLineno} where $j'=1+(j-1) \bmod p^{2}$ and $B$ is a leaf box that contains the support of the grid point $x_{j}$.
The entries $f_{j}$ of the rhs vector $\vec{f}$ are given by $f_{j}=f(x_{j})$.

We solve for $\vec{\psi}$ and then use it to find $u^{scat}(x)$. 
$u^{scat}(x)$ is obtained by discretizing Equation~\eqref{eq:4} (in the same way that the Lippmann-Schwniger equation is discretized) and performing fast directional summation using the Directional Algebraic Fast Multipole Method (DAFMM)~\cite{gujjulaDAFMM}. 

We find the error in the solution, using function $E(x)$, defined as the residual of Equation~\eqref{eq:11c} normalized with $\kappa^{2}$.
\begin{equationwithLineno}\label{eq:LSerror}
  E(x)=\left|\frac{\psi(x)}{\kappa^{2}}+q(x)
  \sum_{\mathcal{L}}\int_{B}G_{\kappa}(x,y)\sum_{l=1}^{N_{p}}  Q^{\dagger}(l,:)\vec{\psi}^{B}b_{l}\left(\frac{y_{1}-\alpha_{1}^{B}}{\beta^{B}}, \frac{y_{2}-\alpha_{2}^{B}}{\beta^{B}}\right) dy - \frac{f(x)}{\kappa^{2}}\right|, \hspace{2mm} x\in\Omega.
\end{equationwithLineno}
We define vector $\vec{E}$, where the entries take the function values of $E(x)$ at the grid points of all the leaf nodes.
We report $\|\vec{E}\|_{2}$ as well as illustrate $E(x)$ pictorially.
We further define a few notations to represent this error for various solvers in Table~\ref{table:LSerr}.

   \begin{table}[!htbp]
\centering
  \begin{tabular}{|l|l|}
    \hline
    $E_{G}^{*}$ & Error $\|\vec{E}\|_{2}$ of GMRES\\ \hline
    $E_{pH}^{*}$ & Error $\|\vec{E}\|_{2}$ of GMRES with HODLR as preconditioner\\ \hline
    $E_{BD}^{*}$ & Error $\|\vec{E}\|_{2}$ of GMRES with block-diagonal preconditioner\\ \hline
    $E_{pA}^{*}$ & Error $\|\vec{E}\|_{2}$ of GMRES with AIFMM as preconditioner\\ \hline
  \end{tabular}
  \caption{Errors of various solvers for the Lippmann-Schwinger equation}
    \label{table:LSerr}
 \end{table}

\subsubsection{Experiment 5: AIFMM as a preconditioner in an iterative solver for Lippmann-Schwinger equation at high frequency}
In this experiment, we demonstrate AIFMM as a preconditioner in solving the Lippmann-Schwinger equation using GMRES.
 We consider Gaussian contrast defined as
\begin{equationwithLineno}\label{eq:gaussianContrast}
  q(x) = 1.5\exp(-160\left(x_1^2+x_2^2\right)).
\end{equationwithLineno}
$\kappa$ is set to $300$. The depth of the uniform quad tree is set to 6.
The leaf size $p^{2}$ is varied to get different system sizes as shown in Table~\ref{table:AIFMMTreeSetting}. The same tree that is used for discretization of the Lippmann-Schwinger equation is used for building AIFMM, HODLR, and DAFMM routines. $\epsilon_{G}$ and $\epsilon_{GMRES}$ are set to $10^{-10}$ and $10^{-8}$ respectively.

 \begin{table}[!htbp]
\centering
  \begin{tabular}{|c|c|c|c|}
    \hline
    $p^{2}$  & 36 & 64 & 100 \\ \hline 
    $N$ & 147456 & 262144 & 409600 \\ \hline
  \end{tabular}
  \caption{Different leaf sizes and the corresponding problem sizes}
    \label{table:AIFMMTreeSetting}
 \end{table}


We compared AIFMM as a preconditioner to a block-diagonal preconditioner and HODLR preconditioner. 
The block diagonal preconditioner is constructed by choosing the block size to be equal to the leaf size.
In Table~\ref{table:bdprecond}, we illustrate the CPU times and the errors of the four iterative solvers: GMRES solver with no preconditioner, GMRES with block-diagonal preconditioner, GMRES with HODLR as a preconditioner, and GMRES with AIFMM as a preconditioner. 
In Figures~\ref{fig:AIFMMField} and~\ref{fig:AIFMMError}, for $N=147456$, the real part of the scattered field, and the log plot of the error function obtained using GMRES with AIFMM as preconditioner with $\epsilon_{A}$ set to $10^{-5}$ are plotted. In Figures~\ref{fig:AIFMMIter} and~\ref{fig:AIFMMTime}, relative residual versus iteration count for different values of $\epsilon_{A}$ and time taken by the different iterative solvers are plotted.

\begin{figure}[!htbp]
      \begin{center}
          \begin{subfigure}[t]{0.5\textwidth}
              \includegraphics[width=\linewidth]{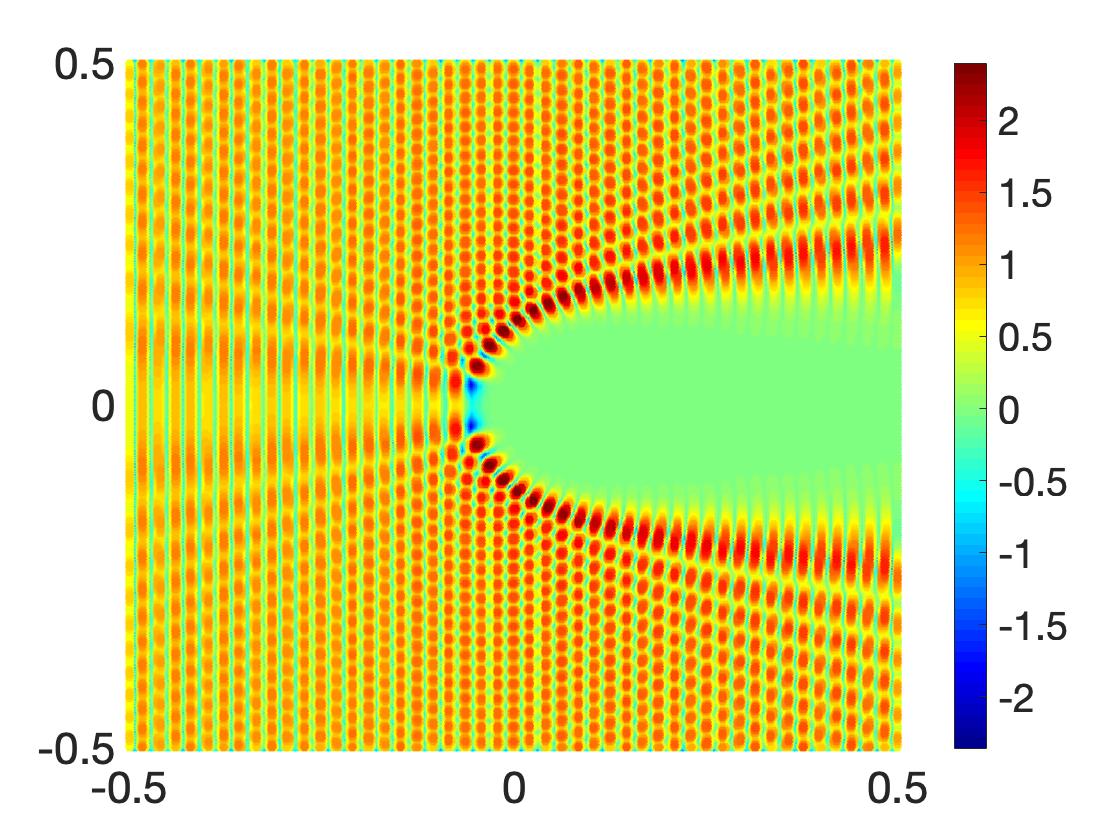}
              \caption{}
            \label{fig:AIFMMField}
          \end{subfigure}%
          \begin{subfigure}[t]{0.5\textwidth}
            \includegraphics[width=\linewidth]{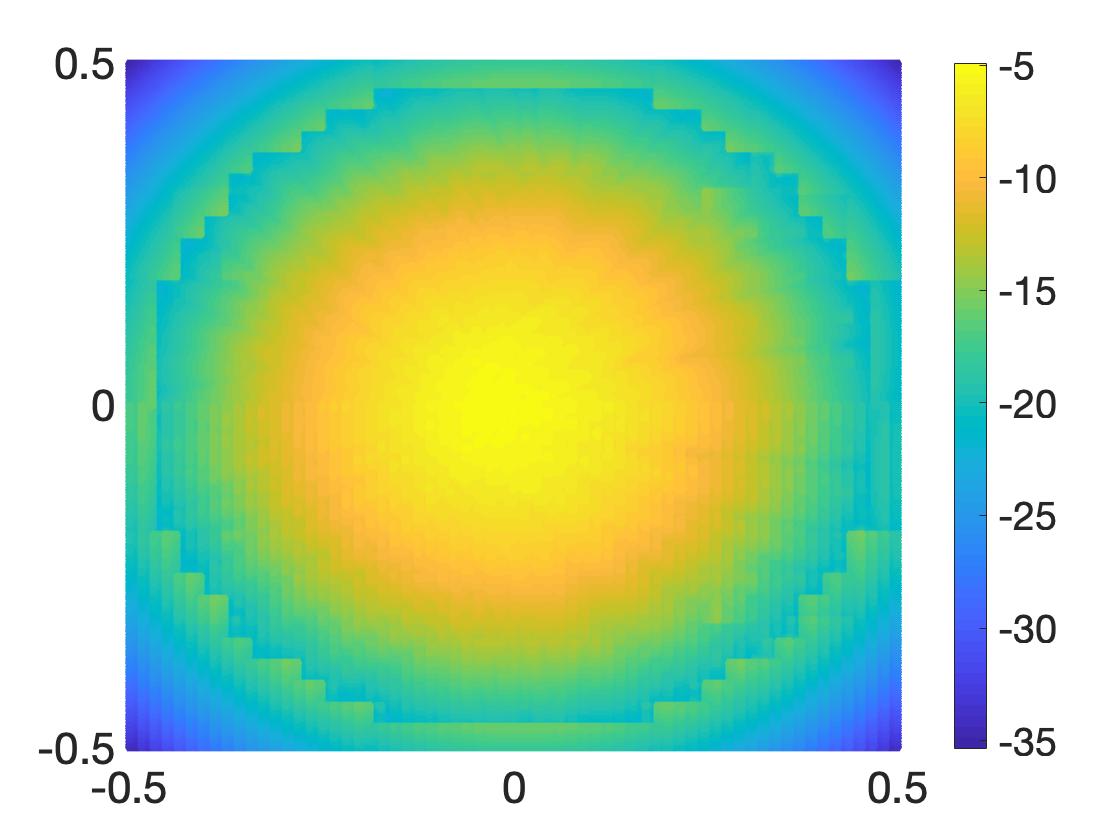}
            \caption{}
            \label{fig:AIFMMError}
          \end{subfigure}
      \hfill
          \begin{subfigure}[t]{0.5\textwidth}
            \includegraphics[width=\linewidth]{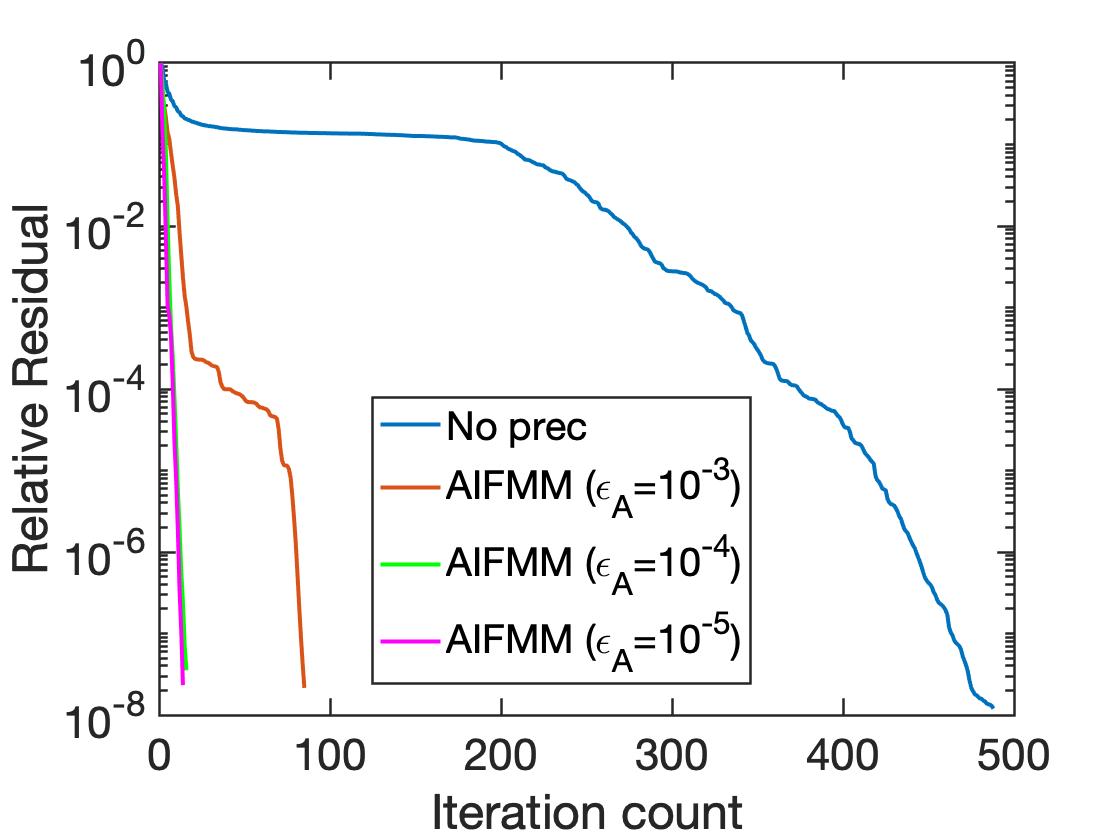}
            \caption{}
              \label{fig:AIFMMIter}
          \end{subfigure}%
          \begin{subfigure}[t]{0.5\textwidth}
            \includegraphics[width=\linewidth]{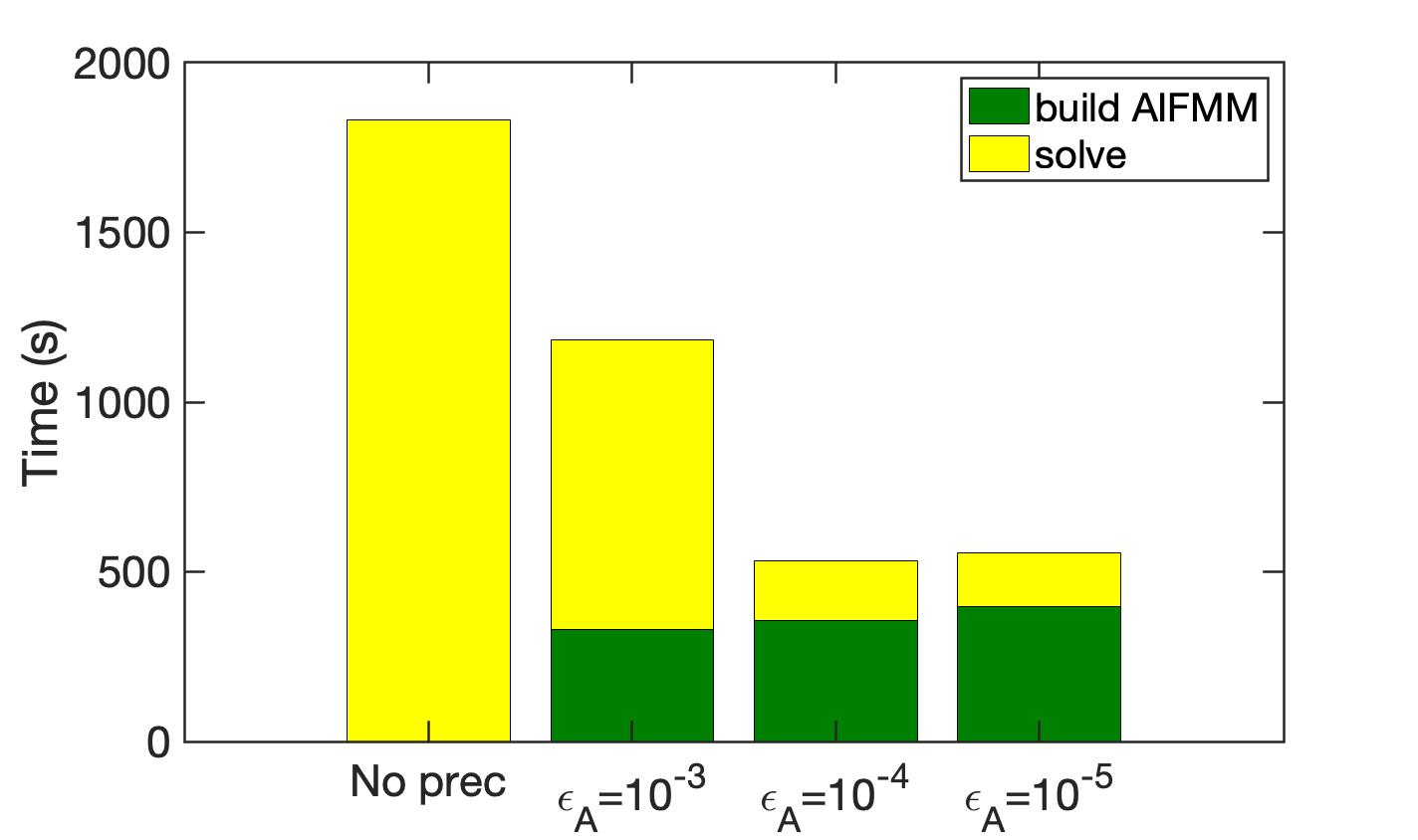}
            \caption{}
            \label{fig:AIFMMTime}
          \end{subfigure}
        \end{center}
        \caption{Results obtained with Experiment 5 with $N=147456$;
        a) real part of the scattered field b) log plot of the error function $E(x)$ c) relative residual versus iteration count for various values of $\epsilon_{A}$ d) Time taken by the iterative solver, where the green bar indicates the time taken to assemble and factorize the AIFMM preconditioner and the yellow bar indicates the taken by the GMRES solver that includes the time to apply the preconditioner}
        \label{fig:precondAIFMM300}
        \end{figure}

\begin{table}[!htbp]
\centering
\setlength\extrarowheight{2.5pt}
\begin{tabular}{|c|c|c|c|c|}
\hline
& $N$ & 147456 & 262144 & 409600   \\ \hline
 & $T_{Ga}$ & 856.04 & 699.70 & 1222.23 \\ \hline
\multirow{ 3}{*}{GMRES with} 
 & $T_{Gs}$ & 1822.34 & 3862.73 & 10353.70 \\
& $I_{G}$ & 488 & 365 & 366  \\
no preconditioner & $E_{G}^{*}$ & $1.8\times10^{-4}$ & - &  -\\ \hline
\multirow{ 3}{*}{Block Diagonal} 
& $T_{BD}$ & 1170 & 3347.6 & 8306.9\\
& $I_{BD}$ & 294 & 295 & 293\\
preconditioner & $E_{BD}^{*}$ & $1.8\times10^{-4}$ & -   &  -\\ \hline
\multirow{ 5}{*}{HODLR $(\epsilon_{H}=10^{-5})$} 
& $T_{Ha}$ & 133.85 & 242.76 & 442.34\\
GMRES with & $T_{Hf}$ & 33.78 & 65.24 & 116.52\\
& $T_{pHs}$ & 62.75 & 192.55 & 457.66\\
preconditioner & $I_{pH}$ & 16 & 17 & 15\\
& $E_{pH}^{*}$ & $1.8\times10^{-4}$ & -  & - \\ \hline
\multirow{ 7}{*}{AIFMM $(\epsilon_{A}=10^{-5})$} 
& $T_{Aa}$ & 43.10 & 95.50 & 167.49 \\
& $T_{Af}$ & 360.69 & 1711.02 & 5843.42 \\
GMRES with & $T_{pAs}$ & 151.77 & 535.00 & 1644.67 \\
& $I_{pA}$ & 14 & 12 & 12  \\
preconditioner & $E_{pA}^{*}$ & $1.8\times10^{-4}$ & -   &  -  \\
 & $\frac{T_{Gs}}{T_{Aa}+T_{Af}+T_{pAs}}$ & 3.28 & 1.65  & 1.35 \\
& $\frac{T_{BD}}{T_{Aa}+T_{Af}+T_{pAs}}$ & 2.11 & 1.43  & 1.08\\ \hline
\end{tabular}
\caption{Results obtained with Experiment 5; CPU times, relative error, and number of iterations it takes using GMRES with no preconditioner and GMRES with AIFMM, HODLR, and block diagonal preconditioners}
\label{table:bdprecond}
\end{table}

\subsection{Inferences}
 The following inferences are to be noticed from Figures~\ref{exp1VsN} to~\ref{fig:exp4} and Tables~\ref{table:exp2} to~\ref{table:exp4}.
 \begin{enumerate}
    \item 
    The maximum rank of HODLR is proportional to $N$, whereas that of AIFMM does not scale with $N$.
    \item 
    Assembly time, solve time, and factorization time scale linearly with $N$ for AIFMM and GMRES, whereas those of HODLR do not scale linearly.
    \item AIFMM is faster than HODLR for the examples considered. The speedup can be observed from Tables~\ref{table:exp2} to~\ref{table:exp4}.
    \item The assembly time of GMRES and AIFMM are nearly equal and the solve time of GMRES is higher than that of AIFMM. When the total CPU time is considered, $T_{Ga}+T_{Gs}$ for GMRES and $T_{Aa}+T_{Af}+T_{As}$ for AIFMM, GMRES is faster than AIFMM for the examples considered. But when one is interested in multiple right-hand sides, it is advantageous to use AIFMM over GMRES, as the solve time of AIFMM is lower than that of GMRES.
 \end{enumerate}

From Table~\ref{table:bdprecond}, it can be noticed that AIFMM performs well as a preconditioner for the high frequency scattering problem and is better than the block diagonal preconditioner, but not as good as the HODLR preconditioner.

In summary, for the problems considered we observed that
\begin{itemize}
    \item The time complexity of AIFMM scales linearly with $N$.
    \item In problems involving the low frequency Helmholtz function and non-oscillatory Green's functions, AIFMM performs better than HODLR as a direct solver. And AIFMM performs better than GMRES when one considers multiple right hand sides.
    \item In the high frequency scattering problem, HODLR as a preconditioner performs better than AIFMM as a preconditioner. And AIFMM as a preconditioner performs better than the block diagonal preconditioner.
\end{itemize}

\section{Conclusions}
A completely algebraic, linear complexity, direct solver for FMM matrices is presented. The various FMM operators are obtained using NNCA, that algebraically obtains nested bases. The advantages of an algebraic technique are (i) the ranks obtained are lower because the method is domain and problem specific; (ii) it can be used in black box fashion, independent of the application. The key ideas of the AIFMM are i) to construct an extended sparse system; (ii) and then perform elimination and substitution, wherein in the elimination phase, the fill-ins corresponding to well-separated hypercubes are compressed and redirected using RRQR. 
Various numerical experiments were presented to demonstrate the scaling and accuracy of AIFMM as a direct solver. It is shown that AIFMM is faster than HODLR, a direct solver. Further, when multiple right hand sides are to be solved for, then AIFMM is better than GMRES. It is also shown that for the high frequency scattering problem, it can be used as a preconditioner and it performs better than the block-diagonal preconditioner. 

\subsection*{Acknowledgments}
    The authors acknowledge HPCE, IIT Madras for providing access to the AQUA cluster.
    Vaishnavi Gujjula acknowledges the support of Women Leading IITM (India) 2022 in Mathematics (SB22230053MAIITM008880). Sivaram Ambikasaran acknowledges the support of Young Scientist Research Award from Board of Research in Nuclear Sciences, Department of Atomic Energy, India (No.34/20/03/2017-BRNS/34278) and MATRICS grant from Science and Engineering Research Board, India (Sanction number: MTR/2019/001241).

 \bibliographystyle{elsarticle-num} 
 \bibliography{references}





\end{document}